\documentclass[10pt]{article}
\usepackage{hyperref}
\usepackage[latin1]{inputenc}
\usepackage{bbm}
\usepackage{stmaryrd}
\usepackage{latexsym}
\usepackage{amsfonts}
\usepackage{amsmath,amssymb,amscd}
\usepackage{yfonts}
\usepackage{dsfont}
\usepackage[dvips]{graphicx}
\usepackage{epsfig}
\usepackage{psfrag}
\usepackage[font={small}]{caption}
\usepackage{mathabx}

\DeclareFontFamily{U}{txsyc}{}
\DeclareFontShape{U}{txsyc}{m}{n}{
   <-> txsyc%
}{}
\DeclareFontShape{U}{txsyc}{bx}{n}{
   <-> txbsyc%
}{}
\DeclareFontShape{U}{txsyc}{l}{n}{<->ssub * txsyc/m/n}{}
\DeclareFontShape{U}{txsyc}{b}{n}{<->ssub * txsyc/bx/n}{}
\DeclareSymbolFont{symbolsC}{U}{txsyc}{m}{n}
\SetSymbolFont{symbolsC}{bold}{U}{txsyc}{bx}{n}
\DeclareFontSubstitution{U}{txsyc}{m}{n}
\DeclareMathSymbol{\df}{\mathrel}{symbolsC}{"42}
\DeclareMathSymbol{\fd}{\mathrel}{symbolsC}{"43}
\DeclareMathSymbol{\lJoin}{\mathrel}{symbolsC}{"58}
\DeclareMathSymbol{\rJoin}{\mathrel}{symbolsC}{"59}

\newcommand{\ES}{\mathbb{E}}

\newcommand{\ER}{\mathbb{R}}
\newcommand{\EN}{\mathbb{N}}
\newcommand{\PE}{\mathbb{P}}

\newcommand{\EE}{\mathbb{E}}

\newcommand{\PP}{\mathbb{P}}

\newcommand{\RR}{\mathbb{R}}

\newcommand{\me}{\medskip}

\newcommand{\bq}{\begin{eqnarray*}}
\newcommand{\bqn}[1]{\begin{eqnarray}\label{#1}}
\newcommand{\eq}{\end{eqnarray*}}
\newcommand{\eqn}{\end{eqnarray}}

\newcommand{\thistitlepagestyle}{}

\newcommand{\ttsim}{\raise.17ex\hbox{$\scriptstyle\mathtt{\sim}$}}

\newcommand{\varlimit}{{\cal V}}

\newcommand{\Hbeta}{(\mathbf{H}^{\gamma}_{\beta})}
\newcommand{\Hr}{(\mathbf{H_r})}
\newcommand{\Hs}{(\mathbf{H}_s)}

\newcommand{\Hsc}{(\mathbf{H}_{SC}(\alpha))}
\newcommand{\Hscl}{(\mathbf{H}_{SC}(\lambda))}
\newcommand{\Hsig}{(\mathbf{H _{\boldsymbol{\sigma,p}}})}
\newcommand{\Hsiginfty}{(\mathbf{H _{\boldsymbol{\sigma,\infty}}})}
\newcommand{\Hsub}{(\mathbf{H _{\boldsymbol{Gauss,\sigma}}})}
\newcommand{\Hellip}{(\mathbf{H _{\mathcal{E}}})}
\newcommand{\tgn }{\widetilde{\gamma}_n}
\newcommand{\tgnp }{\widetilde{\gamma}_{n+1}}
\newcommand{\tGn }{\widetilde{\Gamma}_n}
\newcommand{\tGnp }{\widetilde{\Gamma}_{n+1}}
\newcommand{\tzt}{\widetilde{z}(t)}
\newcommand{\babar}{\bar}
\newcommand{\tXn}{\widetilde{X}_n}
\newcommand{\wch}{\widecheck}
\newcommand{\Ntilde}{\tilde{N}}
\def\tgnp{\widetilde{\gamma}_{n+1}}
\def\tXnp{\widetilde{X}_{n+1}}
\def\tYn{\widetilde{Y}_n}
\def\tYnp{\widetilde{Y}_{n+1}}
\def\tZn{\widetilde{Z}_n}
\def\tZnp{\widetilde{Z}_{n+1}}

\def\cZn{\widecheck{Z}_n}
\def\cXn{\widecheck{X}_n}
\def\cYn{\widecheck{Y}_n}

\def\cXnp{\widecheck{X}_{n+1}}
\def\cYnp{\widecheck{Y}_{n+1}}
\def\cZnp{\widecheck{Z}_{n+1}}

\newtheorem{pro}{Proposition} 
\newtheorem{cor}[pro]{Corollary}
\newtheorem{lem}[pro]{Lemma}
\newtheorem{theo}[pro]{Theorem}
\renewcommand{\thepro}{\arabic{pro}}

\newenvironment{rem}
{\par\me\refstepcounter{pro}\noindent{\bf Remark \thepro\ }}
{\par\hfill $\blacksquare$\par\me\noindent}

\usepackage{color}

\title{Stochastic Heavy ball}
\author{Sébastien Gadat${}^\ddagger$,  Fabien Panloup${}^\dagger$
and Sofiane Saadane${}^\dagger$
}
\date{\box1
 \box2
}

\begin{document}

 \setbox1=\vbox{
 \large
 \begin{center}
 ${}^\ddagger$ Toulouse School of Economics, UMR 5604\\
Université de Toulouse,  France.\\
 \end{center}
 }
\setbox2=\vbox{
\large
\begin{center}
 ${}^\dagger$
Institut de Mathématiques de Toulouse, UMR 5219\\
Université de Toulouse and CNRS, France\\
\end{center}
} 
\setbox3=\vbox{
\hbox{\{gadat$^\ddagger$,panloup${}^\dagger$,saadane${}^\dagger$\}@math.univ-toulouse.fr\\}
\vskip2mm
\hbox{$^\ddagger$Toulouse School of Economics\\}
\hbox{Université Toulouse Capitole \\}
\hbox{21, allées de Brienne\\} 
\hbox{3100 Toulouse, France\\}
\vskip2mm
\hbox{${}^\dagger$ Institut de Mathématiques de Toulouse\\}
\hbox{Université Paul Sabatier\\}
\hbox{118, route de Narbonne\\} 
\hbox{31062 Toulouse Cedex 9, France\\}

}
\setbox5=\vbox{
\box3
}

\maketitle
\thistitlepagestyle
\abstract{This paper deals with a natural stochastic optimization procedure derived from the so-called Heavy-ball method differential equation, which was introduced by Polyak in the 1960s with his seminal contribution \cite{polyak}. The Heavy-ball method is a second-order dynamics that was investigated to minimize convex functions $f$. The family of second-order methods  recently received 
a large amount of attention, until the famous contribution of Nesterov \cite{Nesterov1}, leading to the explosion
of large-scale optimization problems.
This work provides an in-depth description of the stochastic heavy-ball method, which is an adaptation of the deterministic one when only unbiased evalutions of the gradient are available and used throughout the iterations of the algorithm. We first describe some almost sure convergence results in the case of general non-convex coercive functions $f$. We then examine the situation of convex and strongly convex potentials and derive some non-asymptotic results about the stochastic heavy-ball method. We end our study with limit theorems on several rescaled algorithms.
}
\bigskip

{\small
 

}\par

\section{Introduction}

\smallskip
Finding the minimum of a function $f$ over a set $\Omega$ with an iterative  procedure is very popular among numerous scientific communities and has many applications in optimization, image processing, economics and statistics, to name a few. We refer to  \cite{NemirovskiYudin} for a general survey on optimization algorithms and discussions related to complexity theory, and to  \cite{nesterov_book,boyd} for a more focused presentation on convex optimization problems and solutions.
The most widespread approaches rely on some first-order strategies, with a sequence $(X_k)_{k \geq 0}$ that evolves over $\Omega$ with a first-order recursive formula $X_{k+1}=\Psi[X_k,f(X_k),\nabla f(X_k)]$ that uses a local approximation of $f$ at point $X_k$, where this approximation is built  with the knowledge of $f(X_k)$ and $\nabla f(X_k)$ alone. Among them, we refer to the steepest descent strategy in the convex unconstrained case, and to the Frank-Wolfe \cite{FrankWolfe} algorithm in the compact convex constrained case. A lot is known about first-order methods concerning their rates of convergence and their complexity. In comparison to second-order methods, first-order methods are generally slower and are significantly degraded on ill-conditioned optimization problems. However, the complexity of each update involved in first-order methods is relatively limited and therefore useful when dealing with a large-scale optimization problem, which is generally expensive in the case of Interior Point and Newton-like methods. 
 A second-order ``optimal" method was proposed in \cite{Nesterov1} in the 1980s' (also see \cite{BeckTeboulle} for an extension of this method with  proximal operators). The so-called \textit{Nesterov Accelerated Gradient Descent} (NAGD)  has particularly raised considerable interest due to its numerical simplicity, to its low complexity and to its mysterious behavior, making this method very attractive for large-scale machine learning problems. Among the available interpretations of NAGD, some recent advances have been proposed concerning the second-order dynamical system by \cite{candes}, being a particular case of the generalized Heavy Ball with Friction method  (referred to as HBF in the text), as previously pointed out in \cite{cabot1,cabot2}. In particular, as highlighted in \cite{cabot1}, NAGD may be seen as a specific case of HBF after a time rescaling $t=\sqrt{s}$, thus making  the acceleration explicit through this change of variable, as well as being closely linked to the modified Bessel functions when $f$ is quadratic. 
\medskip
 
A growing field of interest related to these optimization algorithms concerns the development of efficient procedures when only noisy gradients are available at each iteration of the procedure. On the practical side, this question was first introduced in the seminal contributions on stochastic approximation and optimization of \cite{RobbinsMonro} and \cite{KW}. Even though the Robbins-Monro algorithm is able to achieve an optimal $O(1/n)$ rate of convergence for strongly convex functions, its ability is highly sensitive to the step sizes used. This remark led \cite{polyakjuditsky}  to develop an averaging method that makes it possible to use longer step sizes of the Robbins-Monro algorithm, and  to then average these iterates with a Cesaro procedure so that this method produces optimal results in the minimax sense (see \cite{NemirovskiYudin}) for convex and strongly convex minimization problems, as pointed out in \cite{bach}.

On the theoretical side, numerous studies have addressed a dynamical system point of view and studied the close links between stochastic algorithms and their deterministic counterparts for some general function $f$ (\textit{i.e.}, even non convex). These links originate in the famous Kushner-Clark Theorem (see \cite{KushnerYin}) and successful improvements have been obtained using differential geometry by \cite{BenaimHirsh,Benaim} on the long-time behavior of stochastic algorithms. In particular, a growing field of interest concerns the behavior of self-interacting stochastic algorithms (see, among others, \cite{BenaimLedouxRaimond} and \cite{GadatPanloup}) because these non-Markovian processes produce interesting features from the modeling point of view (an illustration may be found in \cite{Gadat_Panloup}).

Several theoretical contributions to the study of second-order stochastic optimization algorithms exist. \cite{Lan} explores some adaptations of the NAGD in the stochastic case for composite (strongly or not) convex functions. Other authors \cite{GhadimiLan13,GhadimiLan16} obtained  convergence results for the stochastic version of a variant of NAGD for non-convex optimization for gradient Lipschitz functions but these methods cannot be used for the analysis of the Heavy-ball algorithm.
Finally, a recent work \cite{Yang} proposes a unified study of some stochastic momentum algorithms while assuming  restrictive conditions on the noise of each gradient evaluation and on the constant step size used. It should be noted that \cite{Yang} provides a preliminary result on the behavior of the stochastic momentum algorithms in the non-convex case with possible multi-well situations.
Our work aims to study the properties of a stochastic optimization algorithm naturally derived from the generalized heavy ball with friction method.

\medskip
Our paper is organized as follows: Section \ref{sec:HBF} introduces the stochastic algorithm as well as the main assumptions needed to obtain some results on this optimization algorithm. For the sake of readability, these results are then provided in Section \ref{sec:main_res} without too many technicalities.
The rest of the paper then deals with the proof of these results. Section \ref{sec:AS} is dedicated to the almost sure converegence result we can obtain in the case of a non-convex function $f$ with several local minima. Section \ref{sec:rates} establishes the convergence rates of the stochastic heavy ball in the strongly convex case. Section \ref{sec:TCL} provides a central limit theorem in a particular case of the algorithm. Appendix \ref{sec:appendix} consists of some important results on the supremum of certain random variables needed for the non-convex case.

 \section{Stochastic Heavy Ball}\label{sec:HBF}
 We begin with a brief description of what is known about the underlying dynamical system.
 
 \subsection{Deterministic Heavy Ball}

This method  introduced by Polyak in \cite{polyak}  is inspired from the physical idea of producing some inertia on the trajectory to speed up the evolution of the underlying dynamical system: a ball evolves over the graph of a function $f$ and is submitted to both damping (due to a  friction on the graph of $f$) and acceleration. More precisely, this method is a second-order dynamical system described by the following O.D.E.:
\begin{equation}\label{eq:HBF}
\ddot x_t + \gamma_t \dot x_t + \nabla f(x_t) = 0,
\end{equation}
where $(\gamma_t)_{t \geq 0}$ corresponds to the damping coefficient, which is a key parameter of the method. In particular, it is shown in \cite{cabot1} that the trajectory converges only under some restrictive conditions on the function $(\gamma_t)_{t \geq 0}$, namely:
\begin{itemize}
\item 
if $\int\limits_0^{+\infty} \gamma_s ds=\infty$, then $(f(x_t))_{t \geq 0}$ converges,  
\item  if $\int\limits_0^\infty e^{-\int\limits_0^t \gamma_sds}dt<\infty$, then $(x_t)_{t \geq 0}$ converges towards one of the minima of any convex function $f$. 
\end{itemize} 
Intuitively, these conditions translate the oscillating nature of the solutions of \eqref{eq:HBF} into a quantitative setting for the convergence of the trajectories: if $\gamma_t \longrightarrow 0$ is sufficiently fast, then the trajectory cannot converge (the limiting case being $\ddot x + \nabla f(x) = 0$).
These properties lead us to consider two natural families of functions $(\gamma_t)_{t \geq 0}$: $\gamma_t = r/t$ with $r>1$ and $\gamma_t=\gamma>0$. To convert \eqref{eq:HBF} into a tractable iterative algorithm, it is necessary to rewrite this O.D.E. using a coupled momentum equation. Consistent with \cite{cabot2}, \eqref{eq:HBF} is \textit{equivalent} to the following integro-differential equation:
\begin{equation}\label{eq:memory}
\dot{x}_t  = - \frac{1}{k(t)} \int_{0}^t h(s) \nabla f(x_s) ds,
\end{equation}
where $h$ and $k$ are two memory functions related to $\gamma$. In the natural situation of two positive increasing functions $h$ and $k$, if $(x_t)_{t \geq 0}$ is a solution of \eqref{eq:memory}, then $(\tilde{x}_s)_{s \geq 0}$ is solution of \eqref{eq:HBF} with:
$$
\tilde{x}_s = x_{\tau(s)} \quad \text{and} \quad \dot{\tau}(s) = \sqrt{(kh^{-1})(\tau(s)) } \quad \text{with} \quad 
\gamma_s= \frac{\dot k h + k \dot h}{2 h^{3/2} k^{1/2}} \circ \tau(s).
$$
We can consider two typical situations where the deterministic HBF \eqref{eq:HBF} converges (see \cite{cabot1} for further details):

\begin{itemize}
\item The \textit{ exponentially memoried HBF} corresponds to the choice $k(t) = \lambda e^{\lambda t}$ and $h(t) = e^{\lambda t}$ and to a constant damping function $\gamma_s= \sqrt{\lambda}$ when the time scale is given by $\tau(s) = \sqrt{\lambda s}$. Note that in this situation, the two convergence conditions are satisfied since:
$$
\int\limits_0^{+\infty} \gamma_s ds = \int\limits_0^{+\infty} \sqrt{\lambda} ds =+ \infty \qquad \text{and}\qquad  \int\limits_0^\infty e^{-\int\limits_0^t \gamma_sds}dt
=  \int\limits_0^\infty e^{-\sqrt{\lambda} t} dt < \infty.
$$
\item 
The \textit{ polynomially memoried  HBF} corresponds to the choice $k(t) =  t^{\alpha+1}$ and $h(t)=(\alpha+1) t^{\alpha}$ and is associated with an asymptotically vanishing damping $\gamma_s = \frac{2 \alpha+1}{s}$ and a time scale $\tau(s) = \frac{s^2}{4(\alpha+1)}$, where the choice $\alpha=1$ is associated with the NAGD (see \cite{candes} and their ``magic" constant $3=2\alpha+1$ in that case).
\end{itemize}

\subsection{Stochastic HBF}
All these remarks lead to the consideration of a natural stochastic version of \eqref{eq:memory} when $h=\dot k$. As pointed out by \cite{GadatPanloup}, the introduction of an auxiliary function $y_t=k(t)^{-1} \int_{0}^t h(s) \nabla f(x_s) ds$ makes it possible to obtain a first-order Markov evolution because
$
\dot y_t = r_t (\nabla f(x_t) - y_t)$ with $r_t=\frac{h(t)}{k(t)}$. 
 Hence, we define the stochastic Heavy Ball system as $(X_0,Y_0)=(x,y)\in\ER^{2d}$ and:

\begin{equation}
\label{eq:HBF_sto}
\left\{
\begin{aligned}
X_{n+1} &= X_n - \gamma_{n+1}Y_n\\
Y_{n+1} &= Y_n+ \gamma_{n+1}r_n(\nabla f(X_n)-Y_n) + \gamma_{n+1}r_n\Delta M_{n+1},
\end{aligned}
\right.
\end{equation}
where the natural filtration of the sequence $(X_n,Y_n)_{n\ge0}$ is denoted  $({\cal F}_n)_{n\ge1}$ and:
\begin{itemize}
\item  $(\Delta M_{n})$ is a sequence of  ${\cal F}_n)$-martingale increments. For applications, $\Delta M_{n+1}$ usually represents the difference between the 	``true" value of $\nabla f(X_n)$ and the one observed at iteration $n$ denoted $\partial_x F(X_n,\xi_n)$, where $(\xi_n)_n$ is a sequence of i.i.d. random variables and $F$ is an $\ER^d$-valued measurable function such that:
$$
\forall u \in \mathbb{R}^d \qquad 
\mathbb{E}\left[ \partial_x F( u ,\xi)\right] = \nabla f(u )
$$
In this case, 
 \begin{equation}\label{eq:standrmcase}
 \Delta M_{n+1}=\nabla f(X_n)-\partial_x F(X_n,\xi_n).
 \end{equation}
The randomness appears in  the second component of the algorithm \eqref{eq:HBF_sto}, whereas it was handled in the first component in \cite{GadatPanloup}. We will introduce some assumptions on   $f$ and  on the martingale sequence later.
\item $(\gamma_n)_{n \geq 1}$ corresponds to the step size used in the stochastic algorithm, associated with the ``time" of the algorithm represented by:
$$
\Gamma_n = \sum_{k=1}^n \gamma_k \qquad \text{such that} \qquad \lim_{n \longrightarrow + \infty} \Gamma_n = +\infty.
$$
For the sake of convenience, we also define:
$$
\Gamma_n^{(2)} = \sum_{k=1}^n \gamma_k^2,
$$
which may converge or not according to the choice of the sequence $(\gamma_k)_{k \geq 1}$.
\item $(r_n)_{n \geq 1}$ is a deterministic sequence that mimics the function $t \longmapsto r_t$ defined as:
\begin{equation}\label{eq:defr}
r_n = \frac{h(\Gamma_n)}{k(\Gamma_n)}.
\end{equation}
In particular, when an exponentially weighted HBF with $k(t)=e^{r t}$ is chosen, we have $r_n=r>0$, regardless of the value of $n$. In the other situation where $k(t) = t^r$, we obtain $r_n=r \Gamma_n^{-1}$.
\end{itemize}

\subsection{Baseline assumptions}
We introduce  some of the general assumptions we will work with below. Some of these conditions are very general, whereas others   are more specifically dedicated to the analysis of the strongly convex situation. We will use  the notation $\|.\|$ (resp. $\|.\|_F$)  below  to refer to the Euclidean norm on $\mathbb{R}^d$ (resp. the Frobenius norm on $\mathcal{M}_{d,d}(\mathbb{R})$). Finally, when  $A \in \mathcal{M}_{d,d}(\mathbb{R})$, $\|A\|_{\infty}$ will refer to the maximal size of the modulus of the coefficients of $A$: $\|A\|_{\infty}:=\sup_{i,j} |A_{i,j}|$. Our theoretical results will obviously not involve all of these hypotheses simultaneously. 

\paragraph{Function $f$} We begin with a brief enumeration of assumptions on the function $f$.

\vspace{0.5em}
\noindent
$\bullet$ Assumption $\Hs:$ $f$ is a function in $\mathcal{C}^2(\mathbb{R}^d,\mathbb{R})$ such that:

$$
 \lim_{|x| \longrightarrow + \infty} f(x) = + \infty \qquad \text{and} \qquad \|D^2 f\|_\infty:=\sup_{x\in\ER^d} \|D^2 f(x) \|_F<+\infty
 \qquad \text{and} \qquad \|\nabla f\|^2 \leq c_f f.
 $$
 \vspace{0.5em}

The assumption $\Hs$ is weak: it essentially requires that $f$  be smooth, coercive and have, at the most,
 a quadratic growth on $\infty$. In particular, no convexity hypothesis is made when $f$ satisfies $\Hs$. It would be possible to extend most of our results to the situation where $f$ is $L$-smooth (with a $L$-Lipschitz gradient), but we preferred to work with a slightly more stringent condition to avoid additional technicalities.
%
%

\vspace{0.5em}

\noindent
$\bullet$ Assumption $\Hsc:$  $f$ is a convex function such that 
$\alpha =\inf_{x\in\ER^d} \text{Sp}\left(D^2 f(x) \right)>0$ and $D^2f$ is Lipschitz.
\vspace{0.5em}

In particular, $\Hsc$ implies that $f$ is 
 $\alpha$-strongly convex, meaning that:
$$
\forall (x,y) \in \mathbb{R}^d \times \mathbb{R}^d \qquad f(x) \geq f(y) + \langle \nabla f (y),x-y\rangle + \frac{\alpha}{2} \|x-y\|^2.
$$
Of course, $\Hsc$ is still standard and is the most favorable case when dealing with convex optimization problems, leading to the best possible achievable rates. $\Hsc$ translates the fact that the spectrum of the Hessian matrix at point $x$, denoted by $\text{Sp}\left(D^2 f(x) \right)$, is lower bounded by $\alpha>0$, uniformly over $\mathbb{R}^d$. The fact that $D^2 f$ is  assumed to be Lipschitz will be useful to achieve convergence rates in Section \ref{sec:non_quadra}.

\paragraph{Noise sequence $(\Delta M_{n+1})_{n \geq 1}$}
We will essentially use  three types of  assumptions alternatively on the noise of the stochastic algorithm \eqref{eq:HBF_sto}. The first and second assumptions are concerned with a concentration-like hypothesis. The first one is very weak and asserts that the noise has a bounded $\mathbb{L}^2$ norm.

\vspace{0.5em}

\noindent
$\bullet$ Assumption $\Hsig:$ ($p\ge1$) For any integer $n$, we have:
$$\mathbb{E}(\Vert\Delta M_{n+1}\Vert^{p}|\mathcal{F}_n) \leq \sigma^2(1+f(X_n))^p.$$

\vspace{0.5em}

The assumption $(\mathbf{H _{\boldsymbol{\sigma,2}}})$ is  a standard   convergence assumption for general stochastic algorithms. For some non-asymptotic rates of convergence results, we will rely on $\Hsig$ for any $p\ge1$. In this case, 
we will denote the assumption by $\Hsiginfty$. Finally, let us note that the condition could be slightly alleviated by replacing the right-hand member by $\sigma^2(1+f(X_n)+|Y_n|^2)^p$. However, in view of the standard case \eqref{eq:standrmcase}, this improvement has little interest in practice, which explains our choice.

\vspace{0.5em}

\noindent
$\bullet$ Assumption $\Hsub:$ For any integer $n$, the Laplace transform of the noise satisfies:
$$
\forall t \geq 0 \qquad 
\mathbb{E}\left[  \exp(t \Delta M_{n+1})\vert \mathcal{F}_n \right] \leq e^{\frac{\sigma^2 t^2}{2}}.
$$

\vspace{0.5em}

This hypothesis is much stronger than $\Hsig$ and translates a sub-Gaussian behavior of $(\Delta M_{n+1})_{n \geq 1}$. In particular, it can be easily shown that $\Hsub$ implies $\Hsig$. Hence, $\Hsub$ is somewhat restrictive and will be used only to obtain one important result in the non-convex situation for the almost sure limit of the stochastic heavy ball with multiple wells.

\vspace{0.5em}

\noindent
$\bullet$ Assumption $\Hellip:$ For any iteration $n$, the noise of the stochastic algorithm satisfies:
$$
\forall v \in \mathcal{S}^1_{\mathbb{R}^d} \qquad 
  \mathbb{E} \left( \left| \langle \Delta M_{n},v \rangle \right| \, \vert X_n,Y_n \right)  \geq c_v > 0.$$
  
\vspace{0.5em}

This assumption will be essential  to derive an almost sure convergence result towards minimizers of $f$. Roughly speaking, this assumption states that the noise is uniformly elliptic given any current position of the algorithm at step $n$: the projection of the noise has a non-vanishing component  over all directions $v$. We will use this assumption to guarantee the ability of \eqref{eq:HBF_sto} to get out of any unstable point.

\paragraph{Step sizes} One important step in the use of stochastic minimization algorithms relies on an efficient choice of the step sizes involved in the recursive formula (\textit{e.g.} in Equation \ref{eq:HBF_sto}). We will deal with the following sequences  $(\gamma_n)_{n \geq 0}$ below.

\vspace{0.5em}

\noindent
$\bullet$ Assumption $\Hbeta:$ The sequence $(\gamma_n)_{n \geq 0}$ satisfies:

$$
\forall n \in \mathbb{N}\qquad \gamma_n=\frac{\gamma}{n^{\beta}} \qquad \text{with} \qquad \beta \in (0,1],
$$
leading to:
$$
 \forall \beta \in (0,1) \qquad \Gamma_n \sim \frac{\gamma}{1-\beta} n^{1-\beta} \qquad \text{whereas} \qquad \Gamma_n \sim \gamma \log n \quad \text{when} \quad \beta=1.
$$

\paragraph{Memory size}

We consider  the exponentially and polynomially-weighted HBF as a unique stochastic algorithm parameterized by the memory function $(r_n)_{n \geq 1}$. 
From the definition of $r_n$ given in \eqref{eq:defr}, we note that in the exponential case, $r_n=r$ remains constant while the inertia brought by the memory term in the polynomial case $(r_n)_{n \in \mathbb{N}}$ is defined by $r_n = \frac{r}{\Gamma_n}.$ 
Under Assumption  $\Hbeta$, we can show that regardless of the memory, we have:
$$
\sum_{n \in \mathbb{N}} \gamma_n r_n = + \infty.
$$
This is true when $r_n=r$ because $\gamma_n = \gamma n^{-\beta}$ with $\beta\leq 1$. It is also true when we deal with a polynomial memory since in that case:
\begin{itemize}
\item   if $\beta< 1$, then $\gamma_n r_n \sim \gamma n^{-\beta} \times r(1-\beta) \gamma^{-1} n^{-1+\beta} \sim r (1-\beta) n^{-1}$
\item  if $\beta=1$, then $\gamma_n r_n \displaystyle \sim \frac{r}{ n \log n }$ and $\sum_{k \leq n} \gamma_k r_k \sim \log(\log n)$.
\end{itemize}
 \noindent
 Similarly, we also have that in the polynomial case, regardless of $\beta$:
 $$
 \sum_{n} \gamma_n^2 r_n < + \infty,
 $$
 although this bound holds in the exponential situation when $\beta>1/2$. 
 Below, we will use these properties on the sequences $(\gamma_n)_{n \geq 0}$ and $(r_n)_{n \geq 0}$ and define the next set of assumptions:

\vspace{0.5em} 

 \noindent
$ \bullet$ Assumption $\Hr$: The sequence $(r_n)_{n \geq 0}$ is a non-increasing sequence such that:
 
 $$\sum_{n\ge 1}\gamma_{n+1} r_n = +\infty \qquad \text{and} \qquad \sum_{n\ge1}\gamma_{n+1}^2 r_n<+\infty \qquad \text{and} \qquad  \limsup_{n\rightarrow+\infty}\frac{1}{2\gamma_{n+1}}\left(\frac{1}{r_{n}}-\frac{1}{r_{n-1}}\right)=:c_r<1.$$
 
\vspace{0.5em}

\noindent
In the exponential case,   $c_r=0$, whereas if $r_n=r/{\Gamma}_n$, it can be shown that $c_r=\frac{1}{2r}$ and the last point is true when $r>1/2$.
In any case, $r_{\infty}$ will refer to the limiting value of $r_n$ when $n \longrightarrow + \infty$, which is either $0$ or $r>0$.

\subsection{Main results}\label{sec:main_res}

 Section \ref{sec:AS} is dedicated to the situation of a general coercive function $f$. We obtain the almost sure convergence of the stochastic HBF towards a critical point of $f$.

\begin{theo}\label{theo:as}
Assume that $f$ satisfies $\Hs$, that $(\mathbf{H _{\boldsymbol{\sigma,2}}})$ holds and that and  the sequences $(\gamma_n)_{n \geq 1}$ and $(r_n)_{n\ge1}$ are  chosen such that $\Hbeta$ and $\Hr$ are fulfilled. If for any $z$, 
$\{x, f(x)=z\}\cap\{x,\nabla f(x)=0\}$ is locally finite, then $(X_n)$ a.s. converges  towards a critical point of $f$.
\end{theo} 
 
\noindent This result obviously implies the convergence when $f$ has a unique critical point. In the next theorem, we focus on the case where this uniqueness assumption fails, under the additional elliptic assumption 
$\Hellip$.

\begin{theo}\label{theo:asmin}
Assume that $f$ satisfies $\Hs$, that  the noise is elliptic, \textit{i.e.},  $\Hellip$ holds, and  the sequence $(\gamma_n)_{n \geq 1}$ is  chosen such that $\Hbeta$ and $\Hr$ are fulfilled. If for any $z$, $\{x, f(x)=z\}\cap\{x,\nabla f(x)=0\}$ is locally finite, we have:
\begin{itemize}
\item[$(a)$] If  $r_n=r$ (exponential memory) and  $(\mathbf{H _{\boldsymbol{\sigma,2}}})$ holds, then $(X_n)$ a.s. converges  towards a local minimum of $f$.
\item[$(b)$] If $r_n=r \Gamma_n^{-1}$ and the noise is sub-Gaussian, \textit{i.e.}, $\Hsub$ holds,  then $(X_n)$ $a.s.$ converges  towards a local minimum of $f$ when $\beta<1/3$.
\end{itemize}
\end{theo} 
\begin{rem} $\rhd$ 
The previous result provides some guarantees when $f$ is a multiwell potential. In $(a)$, we consider the exponentially weighted HBF and show that the convergence towards a local minimum of $f$ always holds under the additional
assumption $\Hellip$.
To derive this result, we will essentially use the former results of \cite{BrandiereDuflo} on ``homogeneous" stochastic algorithms.

\noindent $\rhd$ 
Point $(b)$ is concerned by polynomially-weighted HBF and deserves more comment:
\begin{itemize}
\item  First, the result is rather difficult because of the time inhomogeneity of the stochastic algorithm, which can be written as $Z_{n+1} = Z_n + \gamma_{n+1} F_n(Z_n)  + \gamma_{n+1} \Delta M_{n+1}$: the drift term $F_n$ depends on $Z_n$ and on the integer $n$, which will induce technical difficulties in the proof of the result. In particular, the assumption $\beta<1/3$ will be necessary to obtain a good lower bound of the drift term in the unstable manifold direction with the help of the Poincar\'e Lemma near hyperbolic equilibrium of a differential equation.
\item Second, the sub-Gaussian assumption 
$\Hsub$ is less general than $(\mathbf{H _{\boldsymbol{\sigma,2}}})$ even though it is still a reasonable assumption within the framework of a stochastic algorithm. To prove $(b)$, we will need to control the fluctuations of the stochastic algorithm around its deterministic drift, which will be quantified by the expectation of the random variable $\sup_{k \geq n} \gamma_k^{2} \|\Delta M_{k}\|^2$. The sub-Gaussian assumption will be mainly used to obtain an upper bound of such an expectation, with the help of a coupling argument. Our proof will follow a strategy used in \cite{Pemantle} and \cite{Benaim} where this kind of expectation has to be upper bounded. Nevertheless, the novelty of our work is also to generalize the approach to unbounded martingale increments: the arguments of \cite{Pemantle,Benaim} are only valid for a bounded martingale increment, which is a somewhat restrictive framework.
\end{itemize}
\end{rem}
 \noindent
In Section \ref{sec:rates}, we focus on the consistency rate under stronger assumptions on the convexity of $f$. In the exponential memory case, we are able to control the quadratic error and to establish a CLT for the stochastic algorithm under the general assumption  $\Hsc$. In the polynomial case, the problem is more involved and we propose  a result for the quadratic error only when $f$ is a quadratic function (see Remark \ref{rem:commentsL2rate} for further comments on this restriction).
More precisely,  using the notation $\lesssim$ to refer to an inequality, up to a universal multiplicative constant, we establish the following results.

\begin{theo}\label{theo:rates}
Denote by $x^\star$ the unique minimizer of $f$ and assume that  $\Hbeta$, $\Hs$, $\Hsc$ and $(\mathbf{H _{\boldsymbol{\sigma,2}}})$  hold, we have:
\begin{itemize}
\item[$(a)$]When  $r_n=r$ (exponential memory) and       $\beta<1$, we have:
$$
\mathbb{E} \left[\|X_n-x^\star\|^2  + \|Y_n\|^2\right] \lesssim \gamma_n
$$
If  $\Hsiginfty$ holds and $\beta=1$, set $\alpha_r= r\left(1-\sqrt{1-\frac{(4\underline{\lambda}) \wedge r}{r}}\right)$ where $\underline{\lambda}$ denotes the smallest eigenvalue of $D^2 f(x^\star)$.
 We have, for any $\varepsilon>0$:
\begin{equation*}
\mathbb{E}  \left[\|X_n-x^\star\|^2  + \|Y_n\|^2\right] \lesssim \begin{cases} n^{-1}&\textnormal{if $\gamma\alpha_r>1$}\\
n^{-\alpha_r+\varepsilon} &\textnormal{if $\gamma\alpha_r\le 1$}.
\end{cases}
\end{equation*}
\item[$(b)$] Let $f:\ER^d\rightarrow\ER$ be a quadratic function. Assume that $r_n=r \Gamma_n^{-1}$ (polynomial memory) with $\beta<1$.  Then, if  $r>\frac{1+\beta}{2(1-\beta)}$, we have:
$$
\mathbb{E}  \left[\|X_n-x^\star\|^2  + \Gamma_n \|Y_n\|^2 \right]\lesssim \gamma_n
$$
 When $r_n=r \Gamma_n^{-1}$ (polynomial memory) and $\beta=1$, we have:
$$
\mathbb{E}\left[ \|X_n-x^\star\|^2  + \log n \|Y_n\|^2\right] \lesssim \frac{1}{\log n}.
$$
\end{itemize}
\end{theo}
For $(a)$, the case $\beta<1$ is a consequence of Proposition \ref{pro:controVPP} (or Proposition \ref{prop:vitesseexpo} in the quadratic case), whereas the (more involved) case $\beta=1$ is dealt with Propositions \ref{prop:vitesseexpo} and \ref{pro:vitesseexpononquad} for the quadratic and the non-quadratic cases, respectively. We first stress that that when $\beta<1$, the noise only needs to satisfy $\Hsig$ to obtain our  upper bound. When we deal with $\beta=1$, we could prove a positive result in the quadratic case when we only assume $\Hsig$. Nevertheless, the stronger assumption $\Hsiginfty$ is necessary to produce a result in the general strongly convex situation.
 Finally, $(b)$ is a consequence of Proposition \ref{prop:quadpol}.

\begin{rem} \label{rem:commentsL2rate} $\rhd$ It is worth noting that in $(a)$ ($\beta=1$), the dependency of the parameter $\alpha_r$  in $D^2f$ only appears through the smallest eigenvalue of $D^2 f(x^\star)$. In particular, it does not depend on $\displaystyle \inf_{x\in\ER^d}\underline{\lambda}_{D^2 f(x)}$ as it could be expected in this type of result. In other words, we are almost able to  retrieve the conditions that appear when $f$ is  quadratic. This optimization of the constraint is achieved with a ``power increase'' argument, but this involves a stronger assumption $\Hsiginfty$  on the noise.  \smallskip

\noindent $\rhd$ The restriction to quadratic functions in the polynomial case may appear surprising. In fact, the ``power increase'' argument does not work in this non-homogeneous
case. However, when $\beta<1$, it would be possible to extend to non-quadratic functions  through a Lyapunov argument (on this topic, see Remark \ref{rem:polylyapou}), but under some 
quite involved conditions on $r$, $\beta$ and  the  Hessian of $f$. Hence, we chose to only focus  on the quadratic case and to try to obtain some potentially optimal conditions on $r$ and $\beta$ only (in particular, there is no dependence to 
the spectrum of $D^2 f$). The interesting point is that it is possible to preserve the standard rate order when $\beta<1$ but under  the  constraint  $r>\frac{1+\beta}{2(1-\beta)}$, which increases with $\beta$. In particular, the rate $\mathcal{O}(n^{-1})$ cannot be attained in this case (see Remark \ref{rem:ratepolyy}  for more details).\smallskip

\noindent


\end{rem}

 \noindent
Finally, we conclude by a central limit theorem related to the stochastic algorithm the \textit{exponential memory case}.

\begin{theo}\label{theo:TCL} Assume $\Hs$ and $\Hsc$ are true. Suppose that $r_n=r$ and that $\Hbeta$ holds with $\beta\in(0,1)$ or, $\beta=1$ and $\gamma\alpha_r>1$.   Assume that $(\mathbf{H _{\boldsymbol{\sigma,p}}})$ holds with $p>2$ when $\beta<1$ and $p=\infty$ when $\beta=1$.
Finally, suppose that the following condition is fulfilled:
\begin{equation}\label{eq:hypdeltamnproba}
\ES\left[(\Delta M_{n+1}) (\Delta M_{n+1})^t|{\cal F}_{n-1}\right] \xrightarrow{n\rightarrow+\infty}\varlimit\quad \textnormal{in probability}
\end{equation}
where $\varlimit$ is a symmetric positive $d\times d$-matrix.  Let $\sigma$ be a $d\times d$-matrix such that $\sigma\sigma^t=\varlimit.$
  Then,  \begin{itemize}
\item[$(i)$] The normalized algorithm $\left(\frac{X_n}{\sqrt{\gamma_n}},\frac{Y_n}{\sqrt{\gamma_n}}\right)_n $ converges in law to a centered Gaussian distribution $\mu_{\infty}^{(\beta)}$, which is the invariant distribution  of the (linear) diffusion with infinitesimal generator $\mathcal{L}$ defined on ${\cal C}^2$-functions by:
$$\mathcal{L}g(z) =
 \left\langle \nabla g(z),\left(\frac{1}{2\gamma }1_{\{\beta=1\}} I_{2d}+H\right) z\right\rangle 
 + \frac{1}{2}{\rm Tr}(\Sigma^T D^2 g(z)\Sigma) $$
with 
$$H=\begin{pmatrix}
0&-I_d \\
  r D^2 f(x^\star)&-r I_d
\end{pmatrix}\quad\textnormal{and}\quad 
 \Sigma=\begin{pmatrix}
0&0\\ 0& \sigma
\end{pmatrix}.$$
 \item[$(ii)$] In the simple situation where $\varlimit =\sigma_0^2  I_d$ ($\sigma_0>0$) and $\beta<1$. In this case, the covariance  of $\mu_{\infty}^{(\beta)}$  is given by 
$$
\frac{\sigma_0^2}{2} \left( \begin{matrix} \{D^2f(x^\star)\}^{-1} & \mathbf{0}_{d\times d} \\
 \mathbf{0}_{d\times d} & r I_d
 \end{matrix}\right)
$$

In particular, 
$$\frac{X_n}{\sqrt{\gamma_n}}\Longrightarrow {\cal N}(0,\frac{\sigma_0^2}{2} \{ D^2f(x^\star)\}^{-1} ).$$
\end{itemize}
\end{theo}
\begin{rem} 
$\rhd$ As a first comment of the above theorem, let us note that in the  fundamental example where:
 $$\Delta M_{n+1}=\nabla f(X_n)-\partial_x F(X_n,\xi_n), \quad n\ge1,$$
 the additional assumption \eqref{eq:hypdeltamnproba}  is a continuity assumption. Actually,  in this case:
 $$\ES[\Delta M_n\Delta M_n^t|{\cal F}_{n-1}]=\bar{\varlimit}(X_n), \quad \textnormal{with $\bar{\varlimit}(x)={\rm Cov}(F(x,\xi_1))$.}$$
Thus, since $X_n\rightarrow x^\star$ $a.s.$,  Assumption  \eqref{eq:hypdeltamnproba} is equivalent to the continuity of  $\bar{\varlimit}$ in $x^\star$ so that: 
$$\varlimit=\bar{\varlimit}(x^\star).$$ 
$\rhd$
Point $(ii)$ of Theorem \ref{theo:TCL} reveals the behavior of the asymptotic variance of $Y$ increases with $r$. This translates the fact that the instantaneous speed coordinate $Y$ is proportional to $r$ in Equation \eqref{eq:HBF_sto}, which then implies a large variance of the $Y$ coordinate when we use an important value of $r$.

\noindent
$\rhd$ When $\beta=1$, it is also possible (but rather technical) to make the limit variance explicit.
The expression obtained with the classical stochastic gradient descent with step-size $\gamma n^{-1}$ and  Hessian $\lambda$,  the asymptotic variance is $\gamma/(2 \lambda \gamma -1)$, whose optimal value is attained when $\gamma=\lambda^{-1}$ (it attains   the Cramer-Rao lower bound). 
Concerning now the stochastic HBF, for example, when $d=1$  and $r \geq 4\lambda$ (the result is still valid in higher dimensions, see Section \ref{sec:TCL}), we can show that:
$$
\lim_{n \longrightarrow + \infty}
\gamma_n^{-1} \mathbb{E}[X_n^2] = \sigma_0^2 \frac{ 2  \lambda r \gamma^3}{ (\gamma r - 1) ( 2  \lambda\gamma - \check{\alpha}_{-}) (2  \lambda\gamma - \check{\alpha}_{+})},
$$
where $ \check{\alpha}_{+}= 1+\sqrt{1-\frac{4 \lambda}{r}}$ and $ \check{\alpha}_{-} =  1-\sqrt{1-\frac{4 \lambda}{r}}$.
Similar expressions may be obtained when $r<4 \lambda$. Note also that we assumed that $\gamma \alpha_r >1$, and it is easy to check that this condition implies that $\gamma r>1$ because $\alpha_r \leq  r$, regardless of $r$. In the meantime, this condition also implies that $2 \lambda \gamma > \check{\alpha}_{+} \geq \check{\alpha}_{-}$.

Finally,  This explicit value could be used  to find the optimal calibration of the parameters to obtain the best asymptotic variance. Unfortunately, the expressions are rather technical and we can see that such calibrations are far from being independent of $\lambda$, the a priori unknown Hessian of $f$ on $x^\star$.

\noindent
\end{rem}
 
\section{Almost sure convergence of the stochastic heavy ball}\label{sec:AS}
 
In this section, the baseline assumption on the function $f$ is $\Hs$, and we are thus interested in the almost sure convergence of the stochastic HBF. In particular, \textit{we do not make any convexity assumption on $f$}.

\noindent
Below, we will sometimes use  standard and  sometimes more intricate normalizations for the coupled process $Z_n=(X_n,Y_n)$. These normalizations will be of a different nature and, to be as clear as possible, we will always use the same notation $\cZn$ and $\breve{Z}_n$ to refer to a rotation of the initial vector $Z_n$, whereas $\tZn$ will introduce a scaling in the $Y_n$ component of $Z_n$ by a factor $\sqrt{r_n}$.

\subsection{Preliminary result}
We first state a useful upper bound that makes it possible to derive a Lyapunov-type control for the mean evolution of the stochastic algorithm $(X_n,Y_n)_{n \geq 1}$ described by \eqref{eq:HBF_sto}. This result is based on the important function $(x,y) \longmapsto V_n(x,y)$ that depends on two parameters $(a,b) \in \mathbb{R}_+^2$ defined by:
\begin{equation}\label{eq:defVn}
V_n(x,y) = (a+b r_{n-1})f(x)+\frac{a}{2r_{n-1}} \| y\|^2 -b\langle\nabla f(x),y\rangle.
\end{equation}
We will show that $V_n$  plays the role of a (potentially time-dependent) Lyapunov function for the sequence $(X_n,Y_n)_{n \geq 1}$. The construction of $V_n$ shares a lot of similarity with other Lyapunov functions built to control second-order systems. If the two first terms are classical and generate a $-\|y\|^2$ term, the last one is more specific to hypo-coercive dynamics and was already used in \cite{Haraux}. Recent works fruitfully exploit this kind of Lyapunov function  (see, among others, the kinetic Fokker-Planck equations in \cite{Villani} and the memory gradient diffusion in \cite{GadatPanloup}). This function is obtained by the introduction of some Lie brackets of differential operators, leading to the presence of $\langle \nabla f(x),y\rangle$ that generates a mean reverting effect on the variable $x$.

\begin{lem} \label{lem:11111}Assume that  $(\mathbf{H _{\boldsymbol{\sigma,2}}})$  and $\Hs$ hold and suppose that $c_r<1$. Then,  for any  $(a,b) \in \mathbb{R}_+^2$ such that:

\begin{equation}\label{eq:condabr}
\frac{a}{b} > \left( \frac{1}{2}  \vee  \frac{\|D^2 f\|_\infty}{1-c_r} \vee  r_{\infty} (c_f-1)\right),
\end{equation}
we have:
\begin{itemize}
\item[$(i)$]
A constant $C_1>0$ and an integer  $n_0\in\EN$ exist such that for any $n\ge n_0$,
\begin{equation}\label{eq:minoVn}
\forall x,y\in\ER^d,\quad V_n(x,y)\ge C_1 \left(f(x)+\frac{\|y\|^2}{r_{n-1}}\right).
\end{equation}
\item[$(ii)$] Some positive constants $C_2$, $C_3$ and $c_{a,b}$ exist such that:
\begin{align}
 \ES[V_{n+1}&(X_{n+1},Y_{n+1})|{\cal F}_n]\nonumber\\&
 \le V_n(X_n,Y_n)(1+C_2 \gamma_{n+1}^2r_{n})-c_{a,b}\gamma_{n+1}\|Y_n\|^2-b\gamma_{n+1} r_{n}\|\nabla f(X_n)\|^2+C_3\gamma_{n+1}^2 r_n.\label{eq:meanrev}
 \end{align}
 \end{itemize}
 
\end{lem}
\noindent
\underline{\textit{Proof:}}

\noindent
\underline{Point $(i)$:} For any non-negative $u,v$, the elementary inequality  $ uv\le \frac{\rho}{2} u^2+\frac{1}{2\rho} v^2$ holds for any $\rho>0$. We apply this inequality with $u=\|\nabla f(x)\|$, $v=\|y\|$ and $\rho= 2r_{n}$ and obtain:
$$|\langle\nabla f(x),y\rangle|\le  r_{n-1} \|\nabla f(x)\|^2+ \frac{a}{4 r_{n-1}} \|y\|^2.$$
It follows from Assumption $\Hs$ that $\|\nabla f\|^2 \leq c_f f$. Using the above inequality, we obtain  that for any $x,y\in\ER^d$:
$$V_n(x,y)\ge (a+br_{n-1}(1-c_f)) f(x)+\frac{1}{2r_{n-1}}\left[ a - \frac{b}{2} \right]\| y\|^2.$$
Choosing now $a$ and $b$ such that  $a> b/2$ and $a>b r_{\infty}(c_f-1)$, we obtain the first assertion follows from \eqref{eq:condabr}.
\hfill  $\diamond$

\noindent
\underline{Point $(ii)$:}
The Taylor formula ensures the existence of $\xi_{n+1,1}$ and $\xi_{n+1,2}$ in $[X_n,X_{n+1}]$ such that:
\begin{align*}
&V_{n+1}(X_{n+1},Y_{n+1})=(a+b r_n) \left(f(X_n)-\gamma_{n+1} \langle \nabla f(X_n),Y_n\rangle+\frac{\gamma_{n+1}^2}{2} Y_n^{t} D^2 f(\xi_{n+1,1}) Y_n\right)\\
&+\frac{a}{2 r_n}\left(\|Y_n\|^2+2{}\gamma_{n+1}r_n\left(\langle Y_n,\nabla f(X_n)\rangle-\|Y_n\|^2+\langle Y_n+\gamma_{n+1}r_n(\nabla f(X_n)\rangle-Y_n),\Delta M_{n+1}\rangle\right)+\gamma_{n+1}^2r_n^2 \|\Delta M_{n+1}\|^2\right)\\
&-b \left\langle \nabla f(X_n)-\gamma_{n+1} D^2f(\xi_{n+1,2}) Y_n, Y_n+ \gamma_{n+1}r_n\left(\nabla f(X_n)-Y_n + \Delta M_{n+1}\right)\right \rangle.
\end{align*}
Combining the similar terms leads to:

\begin{equation*}\label{eq:djkskljfl}
\begin{split}
V_{n+1}(X_{n+1},Y_{n+1})&=V_n(X_n,Y_n)- b(r_{n}-r_{n-1}) f(X_n)\\
& +\gamma_{n+1}\langle\nabla f(X_n),Y_n\rangle\left(\underset{=0}{\underbrace{-a-b r_{n}+a+b r_{n}}}\right) -\gamma_{n+1} Y_n^t D_{n+1} Y_n-\gamma_{n+1} r_{n} b\|\nabla f(X_n)\|^2\\
&+\gamma_{n+1} r_n \Delta N_{n+1}+\gamma_{n+1} \Delta R_{n+1},
\end{split}
\end{equation*}
where $(\Delta N_n)_{n\ge 1}$ is a sequence of martingale increments, $D_n$ is a $d\times d$-matrix defined by:
$$D_{n+1}= a\left(1-\frac{1}{2\gamma_{n+1}}\left(\frac{1}{r_{n}}-\frac{1}{r_{n-1}}\right)\right) I_d-b D^2 f(\xi_{n+1,2}),$$
and $\Delta R_{n+1}$ is a remainder term. Using $\Hs$, we know that $D^2f$ is bounded, and we have the following bound for $\Delta R_{n+1}$:
$$\|\Delta R_{n+1}\|\le C_2\gamma_{n+1}r_{n}\left(\|Y_n\|^2+ \|\Delta M_{n+1}\|^2+\| \nabla f(X_n)\|. \|Y_n\|\right),$$
where $C_2$ is a deterministic positive constant independent of $n$. The fact that $(r_{n})_{n \geq 1}$ is a bounded sequence combined with Assumptions $(\mathbf{H _{\boldsymbol{\sigma,2}}})$  and $\Hs$ yields
$\ES[\|\Delta R_{n+1}\||{\cal F}_n]\le C_2\gamma_{n+1}r_{n}\left(1+ \|Y_n\|^2+ f(X_n)\right).$
It follows that:
$$\forall n \geq n_0 \qquad \ES[\|\Delta R_{n+1}\||{\cal F}_n]\le C_2\gamma_{n+1} r_{n} V_n(X_n,Y_n).$$
Second, the condition given by \eqref{eq:condabr} shows that an integer $n_1\ge n_0$  and a constant $c_{a,b}>0$ exist such that: 
 $$D_{n+1} Y_n^{\otimes 2}\ge c_{a,b}\|Y_n\|^2.$$
Using the previous bounds in $V_{n+1}(X_{n+1},Y_{n+1})$ and the fact that $(r_n)_{n\in\mathbb{N}}$ is non-increasing  shows that:
$$
\exists n_2  \geq n_1\quad  \forall n \geq n_2: \quad
 \ES[V_{n+1} (X_{n+1},Y_{n+1})|{\cal F}_n]
 \le V_n(X_n,Y_n)(1+C\gamma_{n+1}^2r_n)-c_{a,b}\gamma_{n+1}\|Y_n\|^2-b\gamma_{n+1} r_n\|\nabla f(X_n)\|^2. 
$$
 \hfill $\diamond \square$

\noindent 
Note that if $\Hr$ holds, then  Equation \eqref{eq:meanrev} provides a strong repelling effect on the system $(x,y)$ because in that case, $\sum \gamma_{n+1} r_n = + \infty$. This makes it possible to obtain a more precise  a.s. convergence result, stated below.

\begin{cor}\label{coro:asconv}  If $(\mathbf{H _{\boldsymbol{\sigma,2}}})$  and $\Hs$ hold and $(r_n)_{n\ge1}$ satisfies $\mathbf{(H_r)}$,
then we have:
\begin{itemize}
\item[(i)]
$$\sup_{n\ge 1} \left(\ES[f(X_n)]+\frac{1}{r_n}\ES[\|Y_n\|^2]\right)<+\infty$$
\item[(ii)] $(V_n(X_n,Y_n))_{n\ge 1}$ is $a.s.$-convergent to $V_\infty\in\ER_+$. In particular, $(X_n)_{n\ge1}$ and $(Y_n/{\sqrt{r_n}})_{n\ge1}$ are $a.s.$-bounded.
\item[(iii)] $\displaystyle\sum_{n\ge1}\gamma_{n+1}r_n \left(\frac{\|Y_n\|^2}{r_{n}}+\|\nabla f(X_n)\|^2\right)<+\infty$ $a.s.$
\item[(iv)] $({Y}_n/\sqrt{r_n})_{n\ge0}$ tends to $0$ since $n\rightarrow+\infty$ and every limit point of $(X_n)_{n\ge0}$ belong to $\{x,\nabla f(x)=0\}$. Furthermore, if for any $z$, $\{x, f(x)=z\}\cap\{x,\nabla f(x)=0\}$ is locally finite,
$(X_n)_{n\ge0}$ converges towards a critical point of $f$.
\end{itemize}
\end{cor}

\begin{proof}
\underline{Proof of $(i)-(ii)-(iii)$:}
 Under the conditions on $(r_n)$, we can check that   some positive $a$ and $b$ exist such that the conclusions of the previous lemma hold true. We then deduce that:
 \begin{align*}
 \ES[V_{n+1}&(X_{n+1},Y_{n+1})|{\cal F}_n]\\&
 \le V_n(X_n,Y_n)(1+C \alpha_{n+1} )-U_{n+1},
 \end{align*}
 with $\alpha_n = \gamma_n^2 r_n$ and $U_{n+1} = c_{a,b}\gamma_{n+1} \|Y_n\|^2+b\gamma_{n+1} r_n\|\nabla f(X_n)\|^2$.
Subsequently, using the Robbins-Siegmund Theorem (see, $e.g.$, Theorem \ref{theo:Robbins_Siegmund} in Section \ref{subsec:Rob-Sieg}, borrowed from \cite{Duflo}), we deduce, on the one hand, that  $\sup_{n \geq 1} \mathbb{E}[V_n(X_n,Y_n)] < +\infty$ and that $(V_n(X_n,Y_n))_{n\ge1}$ almost surely (and in $L^1$) converge towards a random variable $V_{\infty} \in \mathbb{R}_+$.  In particular, the coercivity of $f$ implies the $a.s.$-boundedness of  $(X_n)_{n \geq 0}$. On the other hand, the Robbins-Siegmund Theorem also implies that:
$$\sum_{n\ge1}\gamma_{n+1}r_n \left(\frac{\|Y_n\|^2}{r_{n}}+\|\nabla f(X_n)\|^2\right)<+\infty \quad a.s.$$
Hence, the three first statements follow.  \hfill $\diamond$

\noindent
\underline{Proof of $(iv)$:}  The proof relies on the so-called \textit{ODE method} (see, $e.g.$, \cite{Benaim}). Set $r_\infty=\lim_{n\rightarrow+\infty} r_n$. We deal with  cases $r_\infty>0$ and $r_\infty=0$ separately.\smallskip

\noindent \textbf{Case $\bf{r_\infty>0}$ (exponential memory)}:  Set  ${\Gamma_n}=\sum_{k=0}^n{\gamma}_k$ with the convention ${\gamma}_0=0$. Denote by $(\bar{z}(t))_{t\ge0}$ the interpolated process defined by $\bar{z}({\Gamma}_n)=Z_n=(X_n,Y_n)'$, $n\ge0$, with linear interpolations between times ${\Gamma}_n$ and ${\Gamma}_{n+1}$ and let  $\bar{z}^{(n)}$ be the associated \textit{shifted-sequence} defined by:
$$\bar{z}^{(n)}(t)=\bar{z}(t+{\Gamma}_n)\quad t\ge 0.$$ 

Setting $\varepsilon_n=(0, (r_{n-1}-r_\infty)(\nabla f(X_n)-Y_n)+\Delta M_n)'$ and  $h(x,y)=(-y,r_\infty(\nabla f(x)-y))'$, we have:
$$Z_{n+1}=Z_n+\gamma_{n+1} (h(Z_n)+\varepsilon_{n+1}).$$
Set $N(n,t)=\inf\{k\ge n, \gamma_{n+1}+\ldots+\gamma_k\ge t\}$ (with the convention $\inf \emptyset= n$). Then, since $(Z_n)_{n\ge0}$ is $a.s.$-bounded,  it is a classical result on stochastic algorithm theory (see, $e.g.$, \cite{Duflo}, Theorem 9.2.8 and the remark below) that if for any $T>0$,
\begin{equation}\label{condcompacite}
\limsup_{n\rightarrow+\infty}\sup_{t\in[0,T]}\left\|\sum_{k=n+1}^{N(n,t)+1}{\gamma}_k\varepsilon_k\right\|=0\quad a.s.,
\end{equation}
then
$(\bar{z}^{(n)})_{n\ge0}$ is relatively compact (for the topology of uniform convergence on compact sets) and its limit points are solutions to  the ODE $\dot{z}=h(z)$.  
Let us prove \eqref{condcompacite}. Let $T>0$. Using the Cauchy-Schwarz inequality, we have, for every $t\in[0,T]$:
\begin{equation}\label{eq:nntAPT}
\begin{split}
\sum_{k=n+1}^{N(n,t)+1}{\gamma}_k (\|\nabla f(X_{k-1})\|+\|Y_k\|)&\le 
\sqrt{2}\left(\sum_{k=n+1}^{N(n,t)+1}{\gamma}_k\right)^{\frac{1}{2}}\left(\sum_{k=n+1}^{N(n,t)+1}{\gamma}_k\left(\|\nabla f(X_{k-1})\|^2+\|Y_{k-1}\|^2\right)\right)^{\frac{1}{2}}\\
&\le \sqrt{2(T+\gamma_1)}\left(\sum_{k=n+1}^{+\infty}{\gamma}_k\left(\|\nabla f(X_{k-1})\|^2+\|Y_{k-1}\|^2\right)\right)^{\frac{1}{2}}\xrightarrow{n\rightarrow+\infty}0,
\end{split}
\end{equation}
where the last convergence follows from $(iii)$. On the basis of Assumption $(\mathbf{H _{\boldsymbol{\sigma,2}}})$  and $(iii)$,  we also note that $(\langle \sum_{k=1}^n\gamma_k \Delta M_k\rangle)_{n\ge1}$ is $a.s.$-convergent so that $\sum{\gamma}_n\Delta M_n$. It easily follows that:
  \begin{equation*}
\limsup_{n\rightarrow+\infty}\sup_{t\in[0,T]}\left\|\sum_{k=n+1}^{N(n,t)+1}{\gamma}_k\Delta M_k\right\|=0\quad a.s.
\end{equation*}
and that   \eqref{condcompacite} is satisfied. 
Now, we again deduce  from  \eqref{eq:nntAPT} that for any $T>0$,
$$\sup_{t\le T}\|\bar{z}^{(n)}(t)-\bar{z}^{(n)}(0)\|=\sup_{t\le T}\|\bar{z}^{(n)}(t)-Z_n\|\xrightarrow{n\rightarrow+\infty}0$$
so that  each limit point is stationary. At this stage, we have thus proven that every limit point of $(\bar{z}^{n})_{n\ge0}$ is a stationary solution to $\dot{z}=h(z)$. This implies that any limit point $Z_\infty$ of $(Z_n)_{n\ge0}$ satisfies $h(Z_\infty)=0$ (and thus $Y_\infty=\nabla f(X_\infty)=0$). Actually, let $(Z_{n_k})_{k\ge1}$ be a convergent subsequence of the ($a.s.$ bounded) sequence $(Z_n)_{n\ge0}$ and denote  its limit by  $Z_\infty$. Up to a second extraction, $(\bar{z}^{(n_k)})$ converges to a stationary solution  $\bar{z}^{\infty}$ of $\dot{z}=h(z)$. As a consequence, $h(\bar{z}^{\infty}(t))=0$ for any $t\ge0$. In particular,
$h(\bar{z}^{\infty}(0))=h(Z_\infty)=0$. By $(ii)$ and the fact that $(Y_n)_{n\ge0}$ converges to $0$, we also deduce that $(f(X_n))_{n\ge0}$ is $a.s.$-convergent. To conclude the proof, it remains to observe that the set of possible limits of subsequences of $(X_n)_{n\ge1}$ is connected. This is true since $X_n-X_{n-1}=-\gamma_{n} Y_{n-1}\rightarrow0$ as $n\rightarrow+\infty$. \hfill $\diamond$

\vspace{0.5em}

\noindent \textbf{Case $\bf{r_\infty=0}$ (polynomial memory)}:   In this case, the proof is somewhat similar  but the identification of the asymptotic dynamics  requires an appropriate normalization of $Y_n$\footnote{In fact, due to the asymptotic stationarity, the limiting dynamics is not  intrinsic.}. Let us set:
$$\tgn=\gamma_n \sqrt{r_n},\; \tGn=\sum_{k=0}^n\tilde{\gamma}_k,\, \tXn=X_n, \,\tYn=\frac{Y_n}{\sqrt{r_n}}.$$
Also set by $\tZn=(\tXn,\tYn)'$. The dynamic  of $\tZn$ is described by  Lemma \ref{lem:decompZNcheche} below. We denote as $(\tzt)_{t\ge0}$  the  interpolated process, $i.e.$ defined by $\tilde{z}(\tGn)=\tZn$, $n\ge0$,  with linear interpolations between times $\tGn$ and $\tGnp$ and let  $\tilde{z}^{(n)}$ be the associated \textit{shifted-sequence} defined by
$$\tilde{z}^{(n)}(t)=\tilde{z}(t+\tGn)\quad t\ge 0.$$ 
 With this setting, the idea  is to show that the sequence $(\tilde{z}^{(n)}(t))_{t\ge0}$ is tight with limits being stationary solutions of a homogeneous O.D.E. $\dot{z}=\tilde{h}(z)$ ($\tilde{h}$ being the drift to be determined). The sequence $(\tZn)_{n \geq 0}$ satisfies Lemma \ref{lem:decompZNcheche} that shows that $\tZnp=\tZn+\tgnp \left(\tilde{h}(\tZn)+\tilde{\varepsilon}_{n+1}\right)$ with $\tilde{h}(\tilde{x},\tilde{y}) := (-\tilde{y},\nabla f(\tilde{x}))'$
 and:
 $$
 \tilde{\varepsilon}_{n+1}= \begin{pmatrix} 0\\
\upsilon_n^{(1)}\nabla f(\tXn)+\upsilon_n^{(2)}\tYn+\sqrt{\frac{r_n}{r_{n+1}}}\Delta M_{n+1}
\end{pmatrix},
 $$
 where $\upsilon_n^{(1)}$ and $\upsilon_n^{(2)}$ are given in the statement of Lemma \ref{lem:decompZNcheche}.

 \noindent
On the basis of Assumption $\Hr$, we know that:
 $\limsup_{n\rightarrow+\infty}\frac{1}{2\gamma_{n+1}}\left(\frac{1}{r_{n+1}}-\frac{1}{r_n}\right)<1$ so that:
$$\upsilon_n^{(1)}=O\left(\frac{r_n-r_{n+1}}{r_{n+1}}\right)=O(\tgnp \sqrt{r_n}) \qquad \text{and} \qquad \upsilon_n^{(2)}=O\left(\sqrt{r_n}\right).$$
Thus, $(\upsilon_n^{(1)})_{n\ge1}$ and  $(\upsilon_n^{(2)})_{n\ge1}$ converge to $0$ as $n\rightarrow+\infty$. We can now repeat the arguments used in the situation $r_\infty>0$ and we obtain:
\begin{equation*}
\limsup_{n\rightarrow+\infty}\sup_{t\in[0,T]}\left\|\sum_{k=n+1}^{\widetilde{N}(n,t)+1}\tilde{\gamma}_k\tilde{\varepsilon}_k\right\|=0\quad a.s.,
\end{equation*} 
 where $\widetilde{N}(n,t)=\inf\{k\ge n, \widetilde{\gamma}_{n+1}+\ldots+\widetilde{\gamma}_k\ge t\}$. 
We can still combine  \eqref{eq:nntAPT} and $(iii)$ to obtain  $\sup_{t\le T}|\tilde{z}^{(n)}(t)-\tilde{z}^{(n)}(0)|\xrightarrow{n\rightarrow+\infty}0$
 for any $T>0$. We conclude that  $(\tilde{z}^{(n)})_{n\ge0}$ is relatively compact and that its limits are stationary solutions of $\dot{z}=\tilde{h}(z)$. The end of the proof is exactly the same as in the case $r_\infty>0$. \hfill $\diamond \square$
 \end{proof}



\begin{lem} \label{lem:decompZNcheche}
$$\tZnp=\tZn+\tgnp \left(\tilde{h}(\tZn)+\tilde{\varepsilon}_{n+1}\right)$$
where
$\tilde{h}(\tilde{x},\tilde{y})=(-\tilde{y},\nabla f(\tilde{x})-\sqrt{r_\infty} \tilde{y})'$ and 
$$\tilde{\varepsilon}_{n+1}= \begin{pmatrix} 0\\
\upsilon_n^{(1)}\nabla f(\tXn)+\upsilon_n^{(2)}\tYn+\sqrt{\frac{r_n}{r_{n+1}}}\Delta M_{n+1}
\end{pmatrix},
$$
with 
$$\upsilon_n^{(1)}=\sqrt{\frac{r_n}{r_{n+1}}}-1\quad \textnormal{and}\quad \upsilon_n^{(2)}=\frac{1}{\wch{\gamma}_{n+1}}\upsilon_n^{(1)}+\left(\sqrt{r_\infty}-\frac{r_n}{\sqrt{r_{n+1}}}\right).$$
\end{lem}
\begin{proof}
First, the fact that $\tXnp=\tXn-\tilde{\gamma}_{n+1}\tYn$ is obvious. Second,
$$\tYnp=\tYn \sqrt{\frac{r_n}{r_{n+1}}+}\tilde{\gamma}_{n+1}\left(\sqrt{\frac{r_n}{r_{n+1}}}\nabla f(\tXn)-\frac{r_n}{\sqrt{r_{n+1}}}\tYn+\sqrt{\frac{r_n}{r_{n+1}}}\Delta M_{n+1}\right).$$
The lemma follows. \hfill $\square$
\end{proof}

\subsection{Convergence to a local minimum\label{sec:asconv}}
To motivate the next theoretical result, we address the result of Corollary \ref{coro:asconv}. We have shown the almost sure convergence of \eqref{eq:HBF_sto} towards a  point of the form $(x_{\infty},0)$  in both  exponential and polynomial cases where $x_{\infty}$ is a critical point of $f$. This result is obtained under very weak assumptions on $f$ and on the noise $(\Delta M_{n+1})_{n \geq 1}$ and is rather close to Theorems 3-4 of \cite{Yang} (obtained within a different framework). Unfortunately, it this only provides a very partial answer to the problem of minimizing $f$ because 
nothing is said  about the stability of the limit of the sequence $(X_n)_{n \geq 0}$ by Corollary \ref{coro:asconv}: the attained critical point  may be a local maximum, a saddle point or a local minimum. 

This result is made more precise below and we establish  some sufficient guarantees for the a.s. convergence of $(X_n)$ towards a \textit{minimum} of $f$, even if $f$ possesses some local traps with the additional assumption $\Hellip$. This proof follows the approach described in  \cite{BrandiereDuflo} and \cite{Benaim} but requires some careful adaptations because of the hypo-elliptic noise of the algorithm (there is no noise on the $x$-component) for both the exponentially and polynomially-weighted memory. Moreover, the linearization of the inhomogeneous drift around a critical point of $f$ in the polynomial memory case is a supplementary difficulty we need to bypass.
 
 Note that some recent works on stochastic algorithms (see, \textit{e.g.}, \cite{Rechtlol})
deal with the convergence to minimizers of $f$ of \textit{deterministic} gradient descent with a randomized initialization. In our case, we will obtain a rather different result because of the randomization of the algorithm at each iteration. Note, however that the main ingredient of the proofs below will be the stable manifold theorem (the Poincar\'e Lemma on stable/unstable hyperbolic points of \cite{Poincare}) and its consequence around hyperbolic points. This geometrical result is also used in \cite{Rechtlol}.

\subsubsection{Exponential memory $r_n=r>0$}\label{sec:asexpo}

The exponential memory case may be (almost) seen as an application of Theorem 1 of \cite{BrandiereDuflo}. More precisely, 
if $Z_n=(X_n,Y_n)$ and $h(x,y) = (-y,r \nabla f(x) - r y)$, then the underlying stochastic algorithm may be written as:
$$
Z_{n+1} = Z_n + \gamma_n h(Z_n) + \gamma_n \Delta M_n,
$$
 When $r_n=r>0$ (exponential memory), Corollary \ref{coro:asconv} applies and $Z_n \xrightarrow {a.s.} Z_{\infty}=(X_\infty,0)$ where $X_{\infty}$ is a critical point of $f$.
For the analysis of the dynamics around a critical point of the drift, the critical poinf of $f$ is denoted $x_0$ and  we can linearize the drift around $(x_0,0) \in \mathbb{R}^d \times \mathbb{R}^d$ as:
$$
h(x,y) = 
\left(\begin{matrix} 0 & -I_d \\ r D^2(f)(x_0) & -r I_d
\end{matrix} \right) \left( \begin{matrix}x-x_0 \\ y\end{matrix} \right) + O(\|x-x_0\|^2),
$$
where $I_d$ is the $d\times d$ identity-squared matrix and $D^2(f)(x_0)$ is the Hessian matrix of $f$ at point $x_0$. When $x_0$ is not a local minimum of $f$, the spectral decomposition of $D^2(f)(x_0)$  leads to the spectral decomposition:
$$
\exists P \in \mathcal{O}_d(\mathbb{R}) \qquad D^2(f)(x_0) = P^{-1} \Lambda P,
$$
where $\Lambda$ is a diagonal matrix with at least one negative eigenvalue $\lambda<0$.
Considering now $\cZn=(\cXn,\cYn)$ where $\cXn = P X_n$ and $\cYn= P Y_n$, we have:
$$
\cZnp = \cZn + \gamma_n \tilde{h}(\cZn) + \gamma_n P \Delta M_n,
$$
where $\wch{h}$ may be linearized as:
$$
\wch{h}(\wch{x},\wch{y}) = 
\left(\begin{matrix} 0 & -I_d \\ r \Lambda & -r I_d
\end{matrix} \right) \left( \begin{matrix} \wch{x} - \wch{x}_0 \\ \wch{y}\end{matrix} \right) + O(\|\wch{x}-\wch{x}_0\|^2) \qquad \text{where} \qquad \wch{x}_0 = P x_0.
$$
In particular, if $e_{\lambda}$ is an eigenvector associated with the eigenvalue $\lambda <0$ of $D^2f(x_0)$, we can see that the linearization of $\tilde{h}$ on the space  $Span(e_{\lambda}) \otimes (1,0,\ldots,0)$ acts as:
$$
A_{\lambda,r} = \left(\begin{matrix} 0 & -1 \\ r \lambda & -r 
\end{matrix} \right).
$$
Its spectrum is $Sp(A_{\lambda,r}) = -\frac{r}{2} \pm \sqrt{\frac{r^2}{4} - r \lambda}$. The important fact is that when $\lambda <0$, 
 the eigenvalue $-\frac{r}{2} + \sqrt{\frac{r^2}{4} - r \lambda}$ is positive and whose corresponding eigenspace  is
$
 E_{\lambda}^+ = \left(1,\frac{1}{2} - \sqrt{\frac{1}{4} - \lambda / r} \right).$
 In the initial space $\mathbb{R}^d \times \mathbb{R}^d$ (without applying the change of basis through $P \otimes P$), the corresponding eigenvector is:
 
 $$ e_{\lambda}^+ = e_{\lambda} \otimes \left( \frac{1}{2} - \sqrt{\frac{1}{4} - \lambda / r}  \right) e_{\lambda} 
$$
Consequently, when $x_0$ is not a local minimum of $f$, it generates a hyperbolic equilibrium of $h$ and we can apply the 
 ``general" local trap Theorem 1 of  \cite{BrandiereDuflo}.
 If $\Pi_{E_{\lambda}^+}$ denotes the projection on the eigenspace $Span(e_{\lambda}^+)$, then the noise in the direction $E_{\lambda}^+$ is:

$$
\xi_n^+ = \Pi_{E_{\lambda}^+} \left(0,\Delta M_n \right) = 
\frac{\langle  \Delta M_n,e_{\lambda} \rangle}{\| e_\lambda\|^2} e_\lambda.
$$
Now, Assumption $(\mathbf{H _{\boldsymbol{\mathcal{E}}}})$ implies that: 
$$
\liminf_{n\longrightarrow + \infty} \mathbb{E} \left\| \Pi_{E_{\lambda}^+} \left(0,\Delta M_n \right) \right\| \geq c_{e_\lambda} >0.
$$
We can then apply Theorem 1 of \cite{BrandiereDuflo} and conclude the following result.

\begin{theo}\label{theo:asexpo}

If $(\mathbf{H _{\boldsymbol{\sigma,2}}})$ , $\Hs$ and $\Hellip$ hold and $r_n=r$, then $X_n$ a.s. converges  towards a local minimum of $f$.
\end{theo}

\subsubsection{Polynomial memory $r_n=r \Gamma_n^{-1} \longrightarrow 0$}\label{sec:aspoly}
We introduce a key normalization of the speed coordinate and define the rescaled process:
$$
\tXn = X_n \qquad \text{and} \qquad \tYn = \sqrt{\Gamma_n} Y_n.
$$
We can note that $\tYn = \sqrt{r} Y_n r_n^{-1/2}$ and the important conclusion brought by $(iv)$ of Corollary \ref{coro:asconv} is that
$(\tXn,\tYn)\xrightarrow {a.s.} (X_\infty,0)$  still holds (under the assumptions of Corollary \ref{coro:asconv}) 
We can write the recursive upgrade of the couple $(\tXn,\tYn)$. The evolution of $(\tXn)_{n \geq 0}$ is easy to write:
$
\tXnp= \tXn - \frac{\gamma_{n+1}}{\sqrt{\Gamma_n}} \tYn.
$
The recursive formula satisfied by $(\tYn)_{n \geq 0}$ is:
\begin{eqnarray*}
\tYnp&=&\sqrt{\Gamma_{n+1}} \left[ Y_n + \gamma_{n+1}r_{n+1}\left(\nabla f(X_n)-Y_n + \Delta M_{n+1}\right)\right] \\
& = & \frac{\sqrt{\Gamma_{n+1}}}{\sqrt{\Gamma_n}} \tYn + r \frac{\gamma_{n+1}}{\sqrt{\Gamma_n}} \times \frac{\sqrt{\Gamma_{n+1}}}{\sqrt{\Gamma}_n} \nabla f (\tXn) - r \frac{\gamma_{n+1}}{\sqrt{\Gamma_n}} \times \frac{\sqrt{\Gamma_{n+1}}}{ \Gamma_n} \tYn  +r \frac{\gamma_{n+1}}{\sqrt{\Gamma_n}} \times \frac{\sqrt{\Gamma_{n+1}}}{\sqrt{\Gamma}_n}  \Delta M_{n+1} 
\end{eqnarray*}
Hence, the couple $(\tXn,\tYn)$ evolves as an almost standard stochastic algorithm, whose step size is  $\tgnp = \gamma_{n+1} \Gamma_n^{-1/2}$: 

\begin{equation}\label{eq:algo_normalized}                                                     
\left\{
\begin{aligned}
\tXnp &= \tXn - \tgnp \tYn\\
\tYnp &= \tYn + r \tgnp  \nabla f(\tXn)  + \tgnp q_{n+1} \Delta M_{n+1} + \tgnp U_{n+1},
\end{aligned}
\right.
\end{equation}
where $q_{n+1} = \sqrt{\Gamma_{n+1}/\Gamma_n} = 1+o(n^{-1})$ as $n \longrightarrow + \infty$ and $(U_{n+1})_{n \geq 1}$ is defined by:
$$U_{n+1}=\frac{1/2-r q_{n+1}+o(n^{-1})}{ \sqrt{\Gamma_n}} \tYn + r (q_{n+1}-1) \nabla f(\tXn).$$
This dynamical system is related to the deterministic one $\left\{
\begin{aligned}
\dot{x}_t &= - y_t\\
\dot{y}_t &=  r \nabla f(x_t)
\end{aligned}\right.$ or equivalently:
\begin{equation}\label{eq:ode_normalized}
 \dot{z}_t=F(z_t) \quad \text{with} \quad F(z)=F(x,y)=(-y,r\nabla f (x)).
\end{equation}

It is easy to see that when $x_{\infty}$ is a local maximum of $f$, then the above drift is unstable near $z_{\infty} = (x_{\infty},0)$.	
Unfortunately, Theorem 1 of \cite{BrandiereDuflo} cannot be applied because of the size of the remainder terms involved in \eqref{eq:algo_normalized} and the a.s. convergence of $(X_n,Y_n)_{n \geq 0}$ requires further investigation.
From \cite{Benaim}, we borrow a tractable construction of a  ``Lyapunov" function $\eta$ in the neighborhood of each hyperbolic point, which translates a mean repelling effect of the unstable points. This construction still relies on the Poincar\'e Lemma (see \cite{Poincare} and \cite{Hartman} for a recent reference). Again, in the neighborhood of any hyperbolic point, we will treat the projection $\Pi_{+}$ as a projection on the unstable manifold.

\begin{pro}[\cite{Benaim}]\label{prop:poincare}
For any local maximum point $x_{\infty}$ of $f$, a compact neighborhood  $\mathcal{N}$ of  $z_{\infty} = (x_{\infty},0)$ and a positive function $\eta \in \mathcal{C}^2(\mathbb{R}^d \times \mathbb{R}^d,\mathbb{R}_{+}^*)$ exist such that:
\begin{itemize}
\item[$(i)$] $\forall z=(x,y) \in \mathcal{N}, D\eta(z): \mathbb{R}^d \times \mathbb{R}^d \longrightarrow \mathbb{R}^d \times \mathbb{R}^d$ is Lipschitz, convex and positively homogeneous.
\item[$(ii)$] Two constants $k>0$ and $c_1>0$ and a neighborhood $U$ of $(0,0)$ exist such that:
$$
\forall z \in \mathcal{N} \quad \forall u \in U \qquad \eta(z+u) \geq \eta(z) + \langle D\eta(z) , u\rangle  - k \|u\|^2,
$$
and if $\lfloor \, \rfloor_{+}$ denotes the positive part:
$$
\forall z \in \mathcal{N}  \quad \forall u \in U \qquad \lfloor D\eta(z)(u)\rfloor_{+} \geq c_1 \|\Pi_{+}(u)\|.
$$
\item[$(iii)$] A positive constant $\kappa$ exists such that:
$$ \forall z \in \mathcal{N} \qquad \langle D\eta(z),F(z)\rangle \geq \kappa \eta(z)$$
\end{itemize}
\end{pro}

When $d=1$, it is possible to check that if $\lambda$ is a negative eigenvalue
of the Hessian of $f$ around a local maximum $x_{\infty}$, then the drift may be linearized in $(-y,\lambda (x-x_\infty))$ and a reasonable approximation of $\eta$ is given by $\eta(x,y) = \frac{1}{2} \|y-\sqrt{-\lambda} x \|^2$. Nevertheless, the situation is more involved in higher dimensions and the construction of the function $\eta$ relies on the stable manifold theorem.
We are now able to state the next important result.
\begin{theo}\label{theo:aspoly}
Assume that the noise satisfies $\Hsub$ and $\Hellip$, that the function satisfies $\Hs$, and that $\gamma_n = \gamma n^{-\beta}$ with $\beta<1/3$, then $(X_n)_{n \geq 0}$ a.s. converges  towards a local minimum of $f$. 
\end{theo}

The proof relies on an argument of \cite{Pemantle,Benaim} even though it requires major modifications to deal with the time inhomogeneity of the process and the unbounded noise, which are assumed in these previous works. We denote  $\mathcal{N}$ as any neighborhood of $z_{\infty}$ and consider any integer $n_0 \in \mathbb{N}$. We then introduce $\tZn =(\tXn,\tYn)$ and the stopping time:
$$
T := \inf \left\{n \geq n_0 \, : \, \tZn \notin \mathcal{N} \right\}.
$$
We will show that $\mathbb{P}(T<+\infty) = 1$, which implies the conclusion.
We introduce two sequences $(\Omega_n)_{n \geq n_0}$ and $(S_n)_{n \geq n_0}$:
\begin{equation}\label{eq:def_omega_S}
\Omega_{n+1} = [\eta(\tZnp)-\eta(\tZn)] \mathbf{1}_{n < T} + \tilde{\gamma}_{n+1}\mathbf{1}_{n\geq T} \quad \text{and}\quad S_n=\eta(\tilde{Z}_{n_0})+\sum_{k=n_0+1}^n \Omega_k.
\end{equation}
Note that the construction of $\eta$ implies that $z \longmapsto D\eta(z)$ is Lipschitz, so that the following inequality holds:
$$
 \eta(z+u)-\eta(z) \geq  \langle D\eta(z), u \rangle - \frac{\|D\eta\|_{Lip} \|u\|^2}{2}.
$$
 This inequality provides some information when $u$ is small. In the meantime, $\eta$ is positive so that:
\begin{equation}\label{eq:minoration_2}
\forall \alpha \in (0,1] \quad \exists k_{\alpha}>0 \quad \forall (z,u) \in \mathcal{N}\times\mathbb{R}^d \qquad 
 \eta(z+u)-\eta(z) \geq  \langle D\eta(z), u \rangle - k_{\alpha} \|u\|^{1+\alpha}
\end{equation} 
The family of inequalities described in \eqref{eq:minoration_2} will be used with an appropriate value of $\alpha$ in the next result.

\begin{pro}\label{prop:key}
The random variables $(\Omega_n)_{n \geq 0}$ satisfy the following conditions:
\begin{itemize}
\item[$(i)$] A constant $c$ exists such that:
$$
\mathbb{E}[\Omega_{n+1}^2 \vert \mathcal{F}_n ] \leq c \tilde{\gamma}_{n+1}^2
$$
\item[$(ii)$] A sequence $(\epsilon_n)_{n \geq 0}$ exists such that:
$$
\mathbf{1}_{S_n \geq \epsilon_n} \mathbb{E} [\Omega_{n+1} \vert \mathcal{F}_n ]   \geq 0,
$$
with $\epsilon_n \sim c n^{-(1-\alpha)/2}$ for a large enough $c$   and $\alpha = (1-\beta)/(1+\beta)$.
\item[$(iii)$] Assume that $\beta < \frac{1}{3}$, then $(S_n^2)_{n \geq 0}$ has a submartingale increment:
$$
 \mathbb{E} [S_{n+1}^2 - S_n^2 \vert \mathcal{F}_n ] \geq a \tilde{\gamma}_{n+1}^2
$$
for a small enough constant $a$.
\end{itemize}
\end{pro}
\noindent
\textit{Proof:}

\noindent
\underline{Proof of $(i)$.} 
When $n \geq T$, we have $\Omega_{n+1} = \tilde{\gamma}_{n+1}$ by definition and the conclusion follows. In the other situation when $n \leq T$, we use the Lipschitz continuity of $\eta$: if $m = \sup_{z \in \mathcal{N}} \|D \eta(z)\|$, then Equation \eqref{eq:algo_normalized} yields:
\begin{equation*}
\|\eta(\tZnp)-\eta(\tZn)\|^2 \leq 4 m^2 
\tilde{\gamma}_{n+1}^2 
\left[ \|\tYn\|^2 + r^2 \|\nabla f(\tXn)\|^2+ q_{n+1}^2 \|\Delta M_{n+1}\|^2+\|U_{n+1}\|^2\right].
\end{equation*}
The neighborhood $\mathcal{N}$ being compact, we deduce from the previous inequality that a
constant $C>0$ exists such that:
$$
\mathbf{E} \left[ \|\Omega_{n+1}\|^2 \mathbf{1}_{n < T} \, \vert \mathcal{F}_n \right] \leq  
\mathbf{E} \left[ \|\eta(\tZnp)-\eta(\tZn)\|^2   \mathbf{1}_{n < T} \vert \mathcal{F}_n \right]  \leq C \tilde{\gamma}_{n+1}^2,
$$
where we used a uniform upper bound on $\mathbb{E}[ \|\Delta M_{n+1} \|^2  \mathbf{1}_{n < T} \vert \mathcal{F}_n]$,
 leading to the proof of $(i)$. \hfill $\diamond$

\noindent
\underline{Proof of $(ii)$.}
 Note that $\mathbf{1}_{n<T}$ and $ \mathbf{1}_{n \geq T}$ are $\mathcal{F}_n$ measurable and we have:
$$ 
 \mathbf{1}_{n \geq T} \mathbb{E}\left[ \Omega_{n+1} \vert \mathcal{F}_n \right]  = 
  \mathbf{1}_{n \geq T} \tilde{\gamma}_{n+1} \geq 0.
 $$
 On the complementary set, we also have:
\begin{eqnarray*}
\mathbf{1}_{n<T} \mathbb{E}\left[ \Omega_{n+1} \vert \mathcal{F}_n \right]  &
\geq& \mathbf{1}_{n<T} \mathbb{E}\left[  [\eta(\tZnp)-\eta(\tZn)] \vert \mathcal{F}_n \right] =   \mathbf{1}_{n<T}
\mathbb{E}\left[ \eta(\tZnp)-\eta(\tZn) \vert \mathcal{F}_n \right]
\end{eqnarray*}
Hence, we can use the lower bound given by \eqref{eq:minoration_2}: for any value of $\alpha \in (0,1]$:
\begin{eqnarray*}
 \mathbf{1}_{n<T} \mathbb{E}\left[ \Omega_{n+1} \vert \mathcal{F}_n \right] 
&\geq  &\mathbf{1}_{n<T} \left[ \tilde{\gamma}_{n+1} \langle D\eta(\tZn),  F(\tZn) \rangle + \tilde{\gamma}_{n+1}  \langle D\eta(\tZn),  \mathbb{E} [ \Delta M_{n+1} \vert \mathcal{F}_n ]+U_{n+1}\rangle \right] \\
&  & - \mathbf{1}_{n<T} k_{\alpha} \tilde{\gamma}_{n+1}^{1+\alpha} \left[  \|\tYn\| + r \|\nabla f(\tXn)\|+ q_{n+1} \|\Delta M_{n+1}\|+\|U_{n+1}\| \right] ^{1+\alpha}
\end{eqnarray*}
where we used the triangle inequality  in the last line to derive an upper bound of $\|\tZnp-\tZn\|$. When $n<T$, $\tZn$ is bounded and we  have $\mathbb{E} [\|\Delta M_{n+1}\|^2 \vert \mathcal{F}_n] \leq  \sigma^2 M$ for a large enough $M$. Hence, the H\" older inequality implies that:
$$
\mathbb{E}[ \|\Delta M_{n+1}\|^{1+\alpha} \vert \mathcal{F}_n] \leq \sigma^{1+\alpha} M^{\frac{1+\alpha}{2}}.
$$
Therefore, we can find a large enough constant $C_1>0$ such that:
$$
 \mathbf{1}_{n<T} \mathbb{E}\left[ \Omega_{n+1} \vert \mathcal{F}_n \right] \geq \mathbf{1}_{n<T}\left[  \tilde{\gamma}_{n+1} \langle D\eta(\tZn),  F(\tZn) \rangle -  m \tilde{\gamma}_{n+1} \|U_{n+1}\| - C_1 \tilde{\gamma}_{n+1}^{1+\alpha} \right].
$$
The  lower bound $(iii)$ of Proposition  \ref{prop:poincare} and the definition of $U_{n+1}$ implies   that a constant $C_2$ exists such that:
$$
 \mathbf{1}_{n<T} \mathbb{E}\left[ \Omega_{n+1} \vert \mathcal{F}_n \right] \geq \mathbf{1}_{n<T} \tilde{\gamma}_{n+1} \left[ \kappa \eta(\tZn) - C_1 \tilde{\gamma}^{\alpha}_{n+1} - \frac{C_2}{\sqrt{\Gamma_{n}}} \right]
$$
We now choose $\alpha$ so that $\tilde{\gamma}_{n+1}^{\alpha} \simeq \Gamma_{n}^{-1/2}$, which corresponds to the choice:
$$
\alpha=\frac{1-\beta}{1+\beta}.
$$
Defining  $\epsilon_n = \kappa^{-1} \left[ C_1 \tilde{\gamma}_{n+1} + C_2 \Gamma_n^{-1/2}\right]$, we then deduce that if $n < T$, then $S_n= \eta(\tZn)$ so that:
$$
\mathbf{1}_{S_n \geq \epsilon_n} \mathbb{E}\left[ \Omega_{n+1} \vert \mathcal{F}_n \right] \geq 0,
$$
which concludes the proof. In particular, $\epsilon_n$ must be chosen on the order $\tilde{\gamma}_{n+1}^{\alpha}$ (or on the order $\Gamma_n^{-1/2} \sim n^{-(1-\beta)/2}$). 
\hfill $\diamond$

\noindent
\underline{Proof of $(iii)$.}
Observe that $S_{n+1}^2 - S_n^{2} = \Omega_{n+1}^2+ 2 S_n \Omega_{n+1}$. Now, if $S_n \geq \epsilon_n$, then we have seen in the proof of $(ii)$ that:
$$
\mathbf{1}_{S_n \geq \epsilon_n} \mathbb{E} [ S_{n+1}^2 - S_n^{2} \vert \mathcal{F}_n] = 
\mathbf{1}_{S_n \geq \epsilon_n} \mathbb{E}[  \Omega_{n+1}^2 \vert \mathcal{F}_n] + 2 S_n \mathbf{1}_{S_n \geq \epsilon_n}\mathbb{E}[ \Omega_{n+1} \vert \mathcal{F}_n] 
 \geq 
\mathbf{1}_{S_n \geq \epsilon_n} \mathbb{E}[ \Omega_{n+1}^2 \vert \mathcal{F}_n].
$$
In the other situation, we have $S_n \leq \epsilon_n$, meaning that $n<T$ 
and we have seen in the proof of $(ii)$ that:
\begin{eqnarray*}
\mathbf{1}_{n <T} \mathbb{E}  [ \Omega_{n+1} \vert \mathcal{F}_n] &\geq &\mathbf{1}_{n <T}   \left[ \tilde{\gamma}_{n+1} \kappa \eta(\tZn) + \tilde{\gamma}_{n+1}  \langle D\eta(\tZn),  U_{n+1}\rangle \right] \\
& & - k_{2}  \tilde{\gamma}_{n+1}^{2} \left[  \|\tYn\| + r \|\nabla f(\tXn)\|+ q_{n+1} \|\Delta M_{n+1}\|+\|U_{n+1}\| \right] ^{2}
\end{eqnarray*}
Consequently, because of the positivity of $\eta$, we deduce that:
$$
\mathbf{1}_{n <T} \mathbb{E}  [ \Omega_{n+1} \vert \mathcal{F}_n]  \geq - \|D \eta(\tZn)\| \times O(\tilde{\gamma}_{n+1} \Gamma_n^{-1/2}) - O(\tilde{\gamma}_{n+1}^2).
$$
We know that $D\eta$ is locally bounded on $\mathcal{N}$, we then obtain:
\begin{eqnarray*}
\mathbf{1}_{S_n  \leq \epsilon_n} \mathbb{E}  [ \Omega_{n+1} \vert \mathcal{F}_n] &
= & \mathbf{1}_{S_n  \leq \epsilon_n} \mathbf{1}_{n <T} \mathbb{E}  [ \Omega_{n+1} \vert \mathcal{F}_n] = \mathbf{1}_{\eta(\tZn)  \leq \epsilon_n} \mathbf{1}_{n <T}\mathbb{E}  [ \Omega_{n+1} \vert \mathcal{F}_n]  \\
& \geq & - \mathbf{1}_{\eta(\tZn)  \leq \epsilon_n} \mathbf{1}_{n <T} \left[ \|D \eta(\tZn)\| \times O(\tilde{\gamma}_{n+1} \Gamma_n^{-1/2})+ O(\tilde{\gamma}_{n+1}^2)\right]. \\
& \geq &- C \tilde{\gamma}_{n+1} \left[\Gamma_n^{-1/2} +  \tilde{\gamma}_{n+1}\right],
\end{eqnarray*}
for a large enough constant $C$.
In the two situations, we then have:
$$
\mathbb{E} [ S_{n+1}^2 - S_n^{2} \vert \mathcal{F}_n]  \geq \mathbb{E}[  \Omega_{n+1}^2 \vert \mathcal{F}_n] - 2 C \epsilon_n  \tilde{\gamma}_{n+1}^2 - 2 C \epsilon_n  \tilde{\gamma}_{n+1} \Gamma_n^{-1/2}.
$$
Finally,  Lemma 9.7 of \cite{Benaim} and our hypoelliptic assumption  $\Hellip$ implies that for small enough $c$:
$$
 \mathbb{E}[  \Omega_{n+1}^2 \vert \mathcal{F}_n] \geq c \tilde{\gamma}_{n+1}^2
$$
The conclusion follows if $  \epsilon_n \tilde{\gamma}_{n+1} \Gamma_n^{-1/2} = o\left( \tilde{\gamma}_{n+1}^2\right) $.
Since $\epsilon_n$ is chosen on the order $\Gamma_n^{-1/2} \sim \tilde{\gamma}_{n+1}^{\alpha}$ with $\alpha=(1-\beta)/(1+\beta)$,
this condition is equivalent to:
$$
\tilde{\gamma}_{n+1}^{1+2 \alpha} =o\left( \tilde{\gamma}_{n+1}^2\right).
$$
 meaning that $\alpha  > 1/2$. It then implies that $\beta$ should be less than $1/3$. \hfill $\diamond \square$

\medskip
\noindent
We use now the key estimations derived from Proposition \ref{prop:key} to obtain the proof of  the main result of this Section.

\noindent
\textit{Proof of Theorem \ref{theo:aspoly}:} The proof is split into three parts. We consider:
 $$
 S_n = S_0 + \sum_{k=1}^n \Omega_k \qquad \text{and define} \qquad \delta_n = \sum_{i\geq n} \tilde{\gamma}_i^{2}.
 $$
 In our case, we have chosen $\beta \in (0,1/3)$ and we can check that:
\begin{equation}\label{eq:gamma_delta}
\tilde{\gamma}_n \sim n^{-(1+\beta)/2}\quad \text{so that} \quad \delta_n\sim n^{-\beta}.
\end{equation}
We consider the sequence $\epsilon_n$ defined in Proposition \ref{prop:key}:
$$\epsilon_n \sim \Gamma_n^{-1/2} \sim \tilde{\gamma}_{n+1}^\alpha \quad \text{with} \quad \alpha=\frac{1-\beta}{1+\beta}>1/2.$$ In this case, we have: $$\epsilon_n =n^{-(1-\beta)/2}=o(n^{-\beta/2}) = o(\sqrt{\delta_n}) \quad \text{because} \quad \beta<1/3< 1/2.$$
The proof now proceeds  by considering the sequential crossings $S_n \leq c \sqrt{\delta_n}$ and $S_n \geq c \sqrt{\delta_n}$ for a suitable value of $c$.\\

\vspace{0.5em}

\noindent
\textit{\underline{Step 1:} $S_n$ becomes greater than $\sqrt{b \delta_n}$ with a positive probability.}

\noindent
For a given constant $b$ and a positive $n \in \mathbb{N}$, we introduce the stopping time:
$$
\mathcal{T} = \inf \left\{ i \geq n \, : S_i \geq \sqrt{b \delta_i} \right\},
$$
and we show that an $\epsilon>0$ exists such that
$
\mathbb{P}\left(\mathcal{T}<\infty \right) \geq 1 - \epsilon.
$
For $a$  given by $(iii)$ of Proposition \ref{prop:key}, we consider: $$\mathcal{M}_k = S_k^2 - a \sum_{i=0}^k \tilde{\gamma}_{i}^2.$$
 $(\mathcal{M}_k)_{k \geq n}$ is a submartingale, so that $(\mathcal{M}_{k \wedge \mathcal{T}})_{k \geq n}$ is also a stopped submartingale. This yields:
\begin{equation}\label{eq:minoP}
\mathbb{E}\left[S^2_{m\wedge \mathcal{T}} - S_n^2 \vert \mathcal{F}_n \right] \geq a \mathbb{E} \left[ \sum_{n+1}^{m \wedge \mathcal{T}}  \tilde{\gamma}_i^2 \vert \mathcal{F}_n \right] \geq a \left( \sum_{n+1}^m \tilde{\gamma}_i^2\right) \mathbb{P}\left(\mathcal{T}> m \vert \mathcal{F}_n \right).
\end{equation}
In the meantime, we can decompose $S_{m\wedge \mathcal{T}}^2 - S_n^2$ into:
\begin{eqnarray*}
S_{m\wedge \mathcal{T}}^2 - S_n^2 & = & S_{m\wedge \mathcal{T}}^2-S_{m\wedge \mathcal{T}-1}^2 + S_{m\wedge \mathcal{T} - 1}^2 - S_n^2\\
& \leq & 2 S_{m\wedge \mathcal{T} - 1} \Omega_{m\wedge \mathcal{T}} + \Omega_{m \wedge \mathcal{T}}^2 + S_{m\wedge \mathcal{T} - 1}^2\\
& \leq &2  S_{m \wedge \mathcal{T} - 1}^2+ 2\Omega_{m \wedge \mathcal{T}}^2\\
& \leq & 2 b \delta_{m \wedge \mathcal{T} -1}+2 \Omega_{m \wedge \mathcal{T}}^2.
\end{eqnarray*}
Since $(\delta_k)_{k \geq n}$ is decreasing, we then have $\delta_{m \wedge \mathcal{T} -1} \leq \delta_n$. We then study the remaining term. We can use Equation \eqref{eq:algo_normalized} and the Lipschitz continuity of $\eta$ over the neighborhood $\mathcal{N}$ (before time $T$) to obtain a large enough $C$ such that:
\begin{eqnarray*}
\Omega_{m \wedge \mathcal{T}}^2 &=&  \Omega_{m \wedge \mathcal{T}}^2 \left[ \mathbf{1}_{m \wedge \mathcal{T} - 1 < T} +  \mathbf{1}_{m \wedge \mathcal{T} - 1 \geq T} \right] \\
&  = &
\left[\eta(\tilde{Z}_{m \wedge \mathcal{T}})-\eta(\tilde{Z}_{m \wedge \mathcal{T}-1})\right]^2  \mathbf{1}_{m \wedge \mathcal{T} - 1 < T}
+ \tilde{\gamma}^{2}_{m \wedge \mathcal{T} } \mathbf{1}_{m \wedge \mathcal{T} - 1 \geq T} \\
&\leq &  C [ \tilde{\gamma}^{2}_{m \wedge \mathcal{T} } +  \tilde{\gamma}^{2}_{m \wedge \mathcal{T} } \| \Delta M_{m \wedge \mathcal{T} } \|^2 ].
\end{eqnarray*}
However, nothing more is known about the stopped process $ \| \Delta M_{m \wedge \mathcal{T} } \|^2$ and 
 we are forced to use:
$$
\mathbb{E} \left[ S_{m\wedge \mathcal{T}}^2 - S_n^2  \vert \mathcal{F}_n\right] \leq  2  b \delta_{n} + 
 2 C  \left[ \tilde{\gamma}_n^2 + \mathbb{E} \left[ \sup_{k \geq n} \tilde{\gamma}_k^2 \|\Delta M_k\|^2\right] \right].
$$
Given that all $\Delta M_k$ are independent sub-Gaussian random variables that satisfy Inequality \eqref{eq:subgaussian}, we can use Theorem
\ref{theo:sup_sub} and obtain that a constant $C$ large enough exists such that for any $\epsilon>0$:
\begin{equation}\label{eq:DeltaS}
\mathbb{E} 
\left[ S_{m\wedge \mathcal{T}}^2 - S_n^2  \vert \mathcal{F}_n\right] \leq 2 b \delta_n + 2C \tilde{\gamma}_n^{2} \log (\tilde{\gamma}_n^{-2} ).
\end{equation}
We can plug the estimate \eqref{eq:DeltaS} into Inequality \eqref{eq:minoP} to obtain:
$$
\mathbb{P} \left(\mathcal{T}>m \vert \mathcal{F}_n \right) \leq \frac{ 2 b \delta_n + 2 C  \tilde{\gamma}_n^{2} \log (\tilde{\gamma}_n^{-2})}{a \sum_{i=n+1}^m \tilde{\gamma}_i^2}.
$$
Letting $m \longrightarrow + \infty$, we deduce that:
$$
\mathbb{P} \left(\mathcal{T} =  \infty  \vert \mathcal{F}_n \right) 
\leq \frac{ 2 b}{a} + \frac{ 2 C  \tilde{\gamma}_n^{2} \log (\tilde{\gamma}_n^{-2})}{a \delta_n}.
$$
According to the calibration \eqref{eq:gamma_delta}, we have
$\tilde{\gamma}_n^{2} \log (\tilde{\gamma}_n^{-2}) = o(\delta_n).$
Consequently,   we can choose $n$ large enough  such that:
$$\mathbb{P} \left(\mathcal{T} <  \infty  \vert \mathcal{F}_n \right) 
\geq 1- \frac{ 3 b}{a}.$$ \hfill $\diamond$

\vspace{0.5em}

\noindent
\textit{\underline{Step 2:} The sequence $(S_k)_{k \geq n}$ may remain larger than $\sqrt{b/2 \delta_n}$ with a positive probability.}

\noindent
We introduce the stopping time $\mathcal{S}$ and the event $E_n \in \mathcal{F}_n$:
$$
\mathcal{S}= \inf\{i \geq n \, : \, S_i < \frac{\sqrt{b}}{2} \sqrt{\delta_n} \} \qquad \text{and} \qquad E_n = \left\{ S_n \geq \sqrt{b} \sqrt{\delta_n} \right\}.
$$
Since the sequence $(\delta_i)_{i \geq n}$ is non-increasing,   $(ii)$ of Proposition \ref{prop:key} yields:
\begin{eqnarray*}
\mathbb{E} \left[ S_{(i+1)\wedge \mathcal{S}} - S_{i \wedge \mathcal{S}}\vert \mathcal{F}_i \right] &=& \mathbf{1}_{\mathcal{S}>i} 
\mathbb{E} \left[ S_{i+1} - S_{i}\vert \mathcal{F}_i \right]  = 
\mathbf{1}_{\mathcal{S}>i} \mathbf{1}_{S_i \geq \sqrt{b/2 \delta_n}}
\mathbb{E} \left[ S_{i+1} - S_{i}\vert \mathcal{F}_i \right] \\
& \geq & \mathbf{1}_{\mathcal{S}\geq i} \mathbf{1}_{S_i \geq \sqrt{b/2 \delta_i}}
\mathbb{E} \left[ X_{i+1} \vert \mathcal{F}_i \right]   \geq \mathbf{1}_{\mathcal{S}\geq i} \mathbf{1}_{S_i \geq \epsilon_i}
\mathbb{E} \left[ X_{i+1} \vert \mathcal{F}_i \right]  \geq 0.\\
\end{eqnarray*}
Hence, $(S_{i \wedge \mathcal{S}})_{i \geq n}$ is a submartingale and the Doob decomposition reads $S_{i \wedge \mathcal{S}} = M_i+I_i$ where  $(M_i)_{i \geq n}$ is a Martingale and $(I_i)$ is a predictable increasing process such that $I_n=0$. Hence, 
$$
\mathbb{P}(\mathcal{S}=\infty \vert \mathcal{F}_n) = \mathbb{P}_{\vert \mathcal{F}_n}\left( \forall i \geq n \, : S_i \geq \frac{\sqrt{b}}{2} \sqrt{\delta_n} \right)
\geq \mathbb{P}_{\vert \mathcal{F}_n}\left( \forall i \geq n \, : M_i \geq \frac{\sqrt{b}}{2} \sqrt{\delta_n} \right)
$$
On the event $E_n$,  $S_n =M_n \geq \sqrt{b} \sqrt{\delta_n}$ so that $M_i-M_n \leq M_i -  \sqrt{b} \sqrt{\delta_n}$. Therefore:
$$
\mathbb{P}\left( \forall i \geq n \, : M_i \geq \frac{\sqrt{b}}{2} \sqrt{\delta_n} \, \vert \,\mathcal{F}_n \right)\mathbf{1}_{E_n}
\geq \mathbb{P} \left( \forall i \geq n \, : M_i-M_n \geq - \frac{\sqrt{b}}{2} \sqrt{\delta_n} \, \vert \,\mathcal{F}_n \right)\mathbf{1}_{E_n}. 
$$
The rest of the proof follows a standard martingale argument:
\begin{eqnarray*}
\mathbb{E}\left( (M_i-M_n)^2 \vert \mathcal{F}_n \right) & = & \sum_{j=n}^{i-1}
\mathbb{E}\left( (M_{j+1}-M_j)^2 \vert \mathcal{F}_n \right)  =  \sum_{j=n}^{i-1}
\mathbb{E}\left( \mathbb{E}\left(  (M_{j+1}-M_j)^2 \vert \mathcal{F}_j \right) \vert \mathcal{F}_n \right) \\
& = & \sum_{j=n}^{i-1}
\mathbb{E}\left( \mathbb{E}\left(  (S_{j+1}-S_j)^2 \vert \mathcal{F}_j \right) - (I_{j+1}-I_j)^2 \vert \mathcal{F}_n \right) \\
& \leq  &  \sum_{j=n}^{i-1}
\mathbb{E}\left(  (S_{j+1}-S_j)^2  \vert   \mathcal{F}_n \right)  \leq    \sum_{j=n}^{i-1}
\mathbb{E}\left(  \Omega_{j+1}^2  \vert   \mathcal{F}_n \right) \leq c \sum_{j=n}^{i} \tilde{\gamma}_{j+1}^2 \leq c \delta_n.
\end{eqnarray*}
where we used the upper bound given by $(i)$ of Proposition \ref{prop:key} in the last line. Now, the Doob inequality implies that:
\begin{eqnarray*}
\mathbb{P}( \inf_{n \leq i \leq m} (M_i-M_n) \leq -s \vert \mathcal{F}_n) &=& \mathbb{P}( \inf_{n \leq i \leq m} (M_i-M_n-t) \leq -s-t \vert \mathcal{F}_n)\\
& \leq & \mathbb{P}( \sup_{n \leq i \leq m} |M_i-M_n-t| \leq s+t \vert \mathcal{F}_n) \\
& \leq & \frac{E  \left((M_m-M_n-t)^2 \vert \mathcal{F}_n \right)}{(s+t)^2}\\
&  = &\frac{E\left( (M_m-M_n)^2 \vert \mathcal{F}_n \right) + t^2}{(s+t)^2} = \frac{c \delta_n + t^2}{(s+t)^2}.
\end{eqnarray*}
We   apply this inequality with $s=\frac{\sqrt{b}}{2} \sqrt{\delta_n}$ and use
$(s+t)^2 \leq (1+\vartheta) s^2 + (1+\vartheta^{-1}) t^2$ for any $\vartheta>0$. It leads to:

$$
\mathbb{P}\left( \inf_{n \leq i \leq m} (M_i-M_n) \leq -\frac{\sqrt{b}}{2} \sqrt{\delta_n} \vert \mathcal{F}_n\right) 
\leq \frac{c \delta_n+t^2}{(1+\vartheta) b \delta_n/4 +  (1+\vartheta^{-1}) t^2 }.
$$
We now choose $\vartheta = 4 c/b$ , $t = \sqrt{\delta_n}$ and deduce that:
$$
\mathbb{P}\left( \inf_{n \leq i \leq m} (M_i-M_n) \leq -\frac{\sqrt{b}}{2} \sqrt{\delta_n} \vert \mathcal{F}_n\right)  
\leq \frac{c+1}{c+1+b/4c}.
$$
Consequently, we deduce that:
$$
\mathbb{P}(\mathcal{S}=\infty \vert \mathcal{F}_n) \mathbf{1}_{E_n} \geq 
\mathbb{P}_{\vert \mathcal{F}_n}\left( \forall i \geq n \, : M_i \geq \frac{\sqrt{b}}{2} \sqrt{\delta_n} \right)\mathbf{1}_{E_n}
\geq \left(1- \frac{c+1}{c+1+b/4c}\right) \mathbf{1}_{E_n}=  \frac{b}{b+4c+4c^2} \mathbf{1}_{E_n}
$$\hfill $\diamond$

\vspace{0.5em}

\noindent
\underline{Step 3:} \textit{$(S_n)_{n \geq 0}$ does not converge to $0$ with probability $1$.}

\noindent
We denote $\mathcal{G}$ as the event that $(S_n)_{n \geq 0}$ does not converge to $0$. For any integer $n$, we have the inclusion:
$$
\left\{ \mathcal{S} = + \infty \right\} = \left\{ \forall i \geq n \, : \, S_i \geq \sqrt{b/4} \sqrt{\delta_n} \right\} \subset \mathcal{G},
$$
which implies:
$$
\mathbb{E}[ \mathbf{1}_{\mathcal{G}} \vert \mathcal{F}_i] \mathbf{1}_{\mathcal{T}=i} =
\mathbb{E}[ \mathbf{1}_{\mathcal{G}} \vert \mathcal{F}_i] \mathbf{1}_{\mathcal{T}=i} 
\mathbf{1}_{E_i} \geq \frac{b}{b+4c+4c^2} \mathbf{1}_{\mathcal{T}=i}\mathbf{1}_{E_i}  =  \frac{b}{b+4c+4c^2} \mathbf{1}_{\mathcal{T}=i}
$$
Hence, 
\begin{eqnarray*}
\mathbb{E}[ \mathbf{1}_{\mathcal{G}} \vert \mathcal{F}_n] &=&\sum_{i\geq n} 
\mathbb{E}[ \mathbf{1}_{\mathcal{G}} \mathbf{1}_{\mathcal{T}=i} \vert \mathcal{F}_n]\, = \, \mathbb{E}\left[ \mathbb{E}[ \mathbf{1}_{\mathcal{G}}  \vert \mathcal{F}_i] \mathbf{1}_{\mathcal{T}=i} \, \vert \mathcal{F}_n \right] \\ & \geq &   \frac{b}{b+4c+4c^2} \sum_{i\geq n} \mathbb{E}\left[\mathbf{1}_{\mathcal{T}=i} \, \vert \mathcal{F}_n \right]
 \geq  \frac{b}{b+4c+4c^2} \mathbb{P}\left( \mathcal{T}<+\infty \vert \mathcal{F}_n \right)  \geq \frac{b}{b+4c+4c^2} \left( 1-\frac{3b}{a}\right)>0.
\end{eqnarray*}
Since $\mathbf{1}_{\mathcal{G}} \in \mathcal{F}_{\infty}$, we have $\lim_{n \longrightarrow + \infty} \mathbb{E}[ \mathbf{1}_{\mathcal{G}} \vert \mathcal{F}_n] = \mathbf{1}_{\mathcal{G}}$. The previous lower bound implies that $\mathcal{G}$ almost surely holds.  \hfill $\diamond$

\vspace{0.5em}

\noindent
\underline{Conclusion of the proof:} \textit{The stochastic algorithm does not converge to a local trap.}

\noindent
Consider $\mathcal{N}$ a neighborhood of a local maximum of $f$, and its associated function $\eta$ given by Proposition \ref{prop:poincare}. We then consider the random variables $(\Omega_n)_{n \geq 0}$ and $(S_n)_{n \geq 0}$. We have seen that $S_n$ does not converge to $0$ with probability $1$. We define:
$$
\mathcal{T}_{\mathcal{N}} := \inf \left\{n \geq 0 \,: \, \tZn \notin \mathcal{N} \right\}.
$$
and assume that $\mathcal{T}_{\mathcal{N}} = +\infty$. In that case, we always have:
$$
\Omega_{n+1} = \eta(\tZnp) - \eta(\tZn) \qquad \text{and} \qquad S_n = \eta(\tZn).
$$
The limit set of $(\tZn)_{n \geq 0}$ is a non empty  compact subset of $\mathcal{N}$, which is left invariant by the flow $(\Phi_t)_{t \geq 0}$ of the O.D.E. whose drift is $F$. Now, consider $z$ in
$\overline{(\tZn)_{n \geq 0}}$ and apply
 $(iii)$ of Proposition \ref{prop:poincare}. We then have  $\eta(\Phi_t(z)) \geq e^{\kappa t } \eta(y)$. Since $\eta(\Phi_t(z)) \leq \sup_{\mathcal{N}} \eta$, we therefore deduce that $\eta(z)=0$. Hence, the unique limiting value for $(S_n)_{n \geq 0}$ is zero,  meaning that $S_n \longrightarrow 0$ as $n \longrightarrow +\infty$.  However, we have seen in Step 3 that $S_n$ does not converge to $0$ with probability $1$. Therefore, $\PP(\mathcal{T}_{\mathcal{N}}=+\infty)=0$ and the process does not converge towards a local maximum of $f$ with probability $1$. \hfill
 $\square$

\section{Convergence rates for strongly convex functions}\label{sec:rates}

This section focuses on the convergence rates of algorithm \eqref{eq:HBF_sto} according to the step-size $\gamma_n = \gamma n^{-\beta}$ for $\lambda$-strongly convex function $f$ with a $L$-Lipschitz gradient, corresponding to the assumptions $\Hscl$ and $\Hs$.

\subsection{Quadratic case}

We first study the benchmark case of a purely quadratic function $f$, meaning that $\nabla f$ is linear. In this case, $f(x)  = \frac{1}{2} \|A x\|^2$ and $\nabla f(x)=  S x$, leading to the following form of the algorithm:

\begin{equation}
\label{eq:HBF_sto_linear}
\left\{
\begin{aligned}
X_{n+1} &= X_n - \gamma_{n+1}Y_n\\
Y_{n+1} &= Y_n+ \gamma_{n+1}r_n(S X_n -Y_n) + \gamma_{n+1}r_n\Delta M_{n+1},
\end{aligned}
\right.
\end{equation}
where $S$ is a $d \times d$ squared matrix defined by $S = A' A$. The matrix $S$ is assumed to be positive definite with lower bounded eigenvalues, \textit{e.g.}, $Sp(S) \subset [\lambda,+\infty[$ when $f$ is $\Hscl$ with $\lambda>0$.

\subsubsection{Reduction to a two dimensional system}\label{subs:redtotwo}
Equation \eqref{eq:HBF_sto_linear} may be parameterized in a simpler form using the spectral decomposition of $S = P^{-1} \Lambda P$, where $P$ is orthogonal, and $\Lambda$ is a diagonal matrix: 
$$\forall (i,j) \in \{1\ldots d\}^2 \qquad \Lambda_{i,j}=\lambda_i \delta_{i,j} \geq \lambda >0.$$
Keeping the notation $(\cXn,\cYn)_{n \geq 1}$ for the change of basis induced by $P$, we define $\cXn = P X_n$ and $\cYn = PY_n$ and obtain:
\begin{equation*}
\left\{
\begin{aligned}
\cXnp &= \cXn - \gamma_{n+1}\tYn\\
\cYnp &= \cYn+ \gamma_{n+1}r_n(\Lambda \cXn -\cYn) + \gamma_{n+1}r_n P \Delta M_{n+1},
\end{aligned}
\right.
\end{equation*}
Since $\Lambda$ is diagonal, we are now led to study the evolution of $d$ couples of stochastic algorithms:
$$
\forall i \in \{1 \ldots d\} \qquad 
\left\{
\begin{aligned}
\wch{x}_{n+1}^{(i)} &= \wch{x}_{n}^{(i)} - \gamma_{n+1}\wch{y}_{n}^{(i)}\\
\wch{y}_{n+1}^{(i)}  &= \wch{y}_{n}^{(i)}+ \gamma_{n+1}r_n(\lambda_i \wch{x}_{n}^{(i)} -\wch{y}_{n}^{(i)}) + \gamma_{n+1}r_n \Delta \wch{M}^{(i)}_{n+1},
\end{aligned}
\right.
$$
where we used the notations $\cXn=(\wch{x}_n^{(i)})_{1 \leq i \leq d}$ and $\cYn=(\wch{y}_n^{(i)})_{1 \leq i \leq d}$. Consequently, in the quadratic case, the stochastic HBF may be reduced to $d$ couples of $2$-dimensional random dynamical systems:
\begin{equation}\label{eq:HBF_sto_linear2}
\forall i \in \{1, \ldots, d\}^2 \qquad \cZnp^{(i)}=  (I_2+\gamma_{n+1} C_n^{(i)}) \cZn^{(i)} + \gamma_{n+1} r_n \Sigma_2 \Delta N^{(i)}_{n+1},
\end{equation}
where  
$$
\cZn^{(i)} := (\wch{x}^{(i)}_n,\wch{y}^{(i)}_n) \quad \text{and} \quad 
C_n^{(i)} = \left( \begin{matrix}
0 & -1 \\ \lambda^{(i)} r_n & - r_n
\end{matrix} \right) \qquad \text{and} \qquad \Sigma_2 = \left( \begin{matrix}
0 & 0 \\  0 & 1
\end{matrix} \right),
$$
$\lambda^{(i)}=\Lambda_{i,i} \geq \lambda >0$ and $(\Delta N_n^{(i)})_{n\ge1}$ is a sequence of martingale increments.
 
  It is worth noting that due to the multiplication by the matrix $P$, the martingale increment $ \Delta N^{(i)}_{n+1}$ potentially depends on the whole coordinate $(\cZn^{(j)})_{1 \leq j \leq d}$. In a completely general case, this involves technicalities mainly due to the fact that the system \eqref{eq:HBF_sto_linear2} is not completely autonomous (in general, the components  $\cZn^{(i)}$ and $\cZn^{(j)}$ do not evolve   independently). To overcome this difficulty, the idea is to obtain some general controls for a system solution to \eqref{eq:HBF_sto_linear2} and  to then bring the controls of each coordinate together. For the sake of simplicity, we propose in the sequel to state the results in the general case but to only make the proof for \eqref{eq:HBF_sto_linear2} with the assumption that: 
\begin{equation}\label{newmartingale}
 \ES[|\Delta N^{(j)}_{n+1}|^2|{\cal F}_n]\le C (1+ \|\cXn^{(j)}\|^2).
 \end{equation}
 


From now on, we will omit the indexation by $j$ to alleviate the notations.
An easy computation shows that the characteristic polynomial of $C_n$ is given by:
$$
\chi_{_{C_n}}(t) = \left(t+\frac{r_n}{2}\right)^2 + \frac{r_n (4 \lambda - r_n)}{4}.
$$
We now  consider the two different cases:
\begin{itemize}
\item For all $n \geq 1$, $C_n$ has two real or complex eigenvalues whose values do not change from $n$ to $n$, which corresponds to $r_n= r$. This case necessarily corresponds to an exponentially-weighted memory and  $r_n$ is thus kept fixed constant: $r_n=r \geq 4 \lambda$
or $r_n=r<4 \lambda$.
\item For a large enough $n$, $C_n$ has two complex conjugate and vanishing eigenvalues. This situation may occur if we use  a polynomially-weighted memory because, in that case, $r_n \longrightarrow 0$ as $n \longrightarrow + \infty$.
\end{itemize}
\subsubsection{Exponential memory $r_n=r$}
We first study the situation when $r_n=r$, which is easier to deal with from a technical point of view.

\begin{pro}\label{prop:vitesseexpo} Assume $(\mathbf{H_{\bold{\sigma,2}}})$. Let $(Z_n)_{n\ge0}$ be defined by  \eqref{eq:HBF_sto_linear} with $Sp(S) \subset [\lambda,+ \infty[$ and $r_n=r$. Set:
$$
\alpha_r = 
\displaystyle\begin{cases}
r \left(1-\sqrt{1-\frac{4 \lambda}{r}}\right), \,\,\,\, \qquad \text{if} \qquad r \geq 4 \lambda\\
r  \qquad \qquad  \qquad \qquad \qquad \text{if} \qquad r < 4 \lambda, \\
\end{cases}.
$$
Assume that $\gamma_n = \gamma n^{-\beta}$, we then have:
\begin{itemize}
\item[$(i)$]  If $\beta<1$,   then a constant $c_{r,\lambda,\gamma} $ exists  such that:
$$
\forall n \geq 1 \qquad \mathbb{E}\left[\|X_n\|^2+\|Y_n\|^2\right] \leq c_{r,\lambda,\gamma} \gamma_n.
$$
\item[$(ii)$] If $\beta=1$, then a constant $c_{r,\lambda,\gamma} $  exists such that:
$$
\forall n \geq 1 \qquad \mathbb{E}\left[\|X_n\|^2+\|Y_n\|^2\right] \leq c_{r,\lambda,\gamma} n^{-\left(1 \wedge \gamma \alpha_r \right)} \log(n)^{\mathbf{1}_{\{\gamma \alpha_r=1\}}}.
$$
\end{itemize}
\end{pro}

\noindent
\underline{\textit{Proof:}} According to Subsection \ref{subs:redtotwo}, we only make the proof for a system solution to \eqref{eq:HBF_sto_linear2} with the assumption that \eqref{newmartingale} holds. We begin with the simplest case where $r \geq 4 \lambda$. The  above   computations show that:
\begin{equation}\label{eq:specC}
Sp(C_n) = \left\{\mu_{+}= \frac{-r+\sqrt{(r-4  \lambda) r}}{2};\mu_{-}= \frac{-r-\sqrt{(r-4  \lambda) r}}{2}\right\},
\end{equation}
while the associated eigenvectors are given by
$
e_+ = \left( \begin{matrix} 1 \\ - \mu_+ \end{matrix} \right)$ and $e_- = \left( \begin{matrix} 1 \\ - \mu_- \end{matrix} \right)$ and are kept fixed throughout the iterations of the algorithm. Consequently, \eqref{eq:HBF_sto_linear2} may be rewritten in an even simpler way:
\begin{equation}\label{eq:cZnp}
\cZnp = \left( \begin{matrix}
1+\gamma_{n+1} \mu_+ & 0 \\ 0 & 1+\gamma_{n+1} \mu_-
\end{matrix} \right) \cZn+r \gamma_{n+1} \widecheck{\xi}_{n+1},
\end{equation}
where $\cZn=  Q Z_n$ ($(Z_n)$ being defined by \eqref{eq:HBF_sto_linear2} ) where $Q$ is an invertible matrix such that $C_n = Q^{-1} \left( \begin{matrix}
  \mu_+ & 0 \\ 0 &  \mu_-
\end{matrix} \right) Q$ and $\widecheck{\xi}_{n+1} = Q \Sigma_2  \Delta {N}_{n+1}$. The squared norm of $(\cZn)_{n \geq 1}$ is now controlled using a standard martingale argument and Assumption $(\mathbf{H_{\bold{\sigma,2}}})$:
\begin{eqnarray*}
\mathbb{E} \left[ \|\cZnp\|^2 \vert \mathcal{F}_n \right]& \leq & [(1+  \mu_{+} \gamma_{n+1})^2+C\gamma_{n+1}^2] \|\cZn\|^2 +C \gamma_{n+1}^2,\\
\end{eqnarray*}
so that by setting $u_n=\ES[\|\cZn\|^2]$, this yields:
\begin{equation}\label{ineq:unn}
u_{n+1}\le   (1+  2\mu_{+} \gamma_{n+1}+C_1\gamma_{n+1}^2)+C_2\gamma_{n+1}^2.
\end{equation}
The result then follows from Propositions \ref{prop:controle_algo} $(iii)$ and \ref{prop:controle_algo2} $(iii)$ (see Appendix  \ref{sec:appendix}).\smallskip



\noindent We now study the situation $r<4 \lambda$. In this case, $C_n$ possesses two conjugate complex eigenvalues: 
$$
Sp(C_n) = \left\{ 
\mu_+ = \frac{-r+ \mathfrak{i} \sqrt{r(4\lambda-r)}}{2}  ; \mu_- = \frac{-r- \mathfrak{i} \sqrt{r(4\lambda-r)}}{2}.
\right\},
$$
Once again, we use the notation $(\cZn)_{n \geq 1}$ defined as $\cZn = Q Z_n$ with $Q$ an invertible (complex) matrix such that $S_n = Q^{-1} \left( \begin{matrix}
  \mu_+ & 0 \\ 0 &  \mu_-
\end{matrix} \right) Q$ and $\widecheck{\xi}_{n+1} = Q \Sigma_2 \Delta N_{n+1}$. The squared norm of $(\cZn)_{n \geq 1}$ may be controlled while paying attention to the modulus of complex numbers, and we obtain an inequality similar to \eqref{ineq:unn}.
\begin{eqnarray*}
\mathbb{E} \left[ \|\cZnp\|^2 \vert \mathcal{F}_n \right] 
& \leq & \max\left( \left| 1+  \mu_{+} \gamma_{n+1}\right|^2; \left| 1+  \mu_{-} \gamma_{n+1}\right|^2 \right) \|\cZn\|^2 +C_2\gamma_{n+1}^2,\\
& \leq &  \left( \left(1-\frac{\gamma_{n+1} r}{2}\right)^2 + C_1 \gamma_{n+1}^2 \right) \|\cZn\|^2+  C_2\gamma_{n+1}^2,\\
& \leq & \left( 1- \gamma_{n+1} r + C_1 \gamma_{n+1}^2  \right) \|\cZn\|^2+ C_2\gamma_{n+1}^2. \\
\end{eqnarray*}
Once again, we can apply $(iii)$ of Propositions \ref{prop:controle_algo}$(iii)$ and \ref{prop:controle_algo2}$(iii)$   to obtain the desired conclusion. \hfill $\square$
\begin{rem}
In the above proposition, the constants $c_{r,\lambda,\gamma}$ are not made explicit. However, it is possible to obtain an estimation if we assume that $\ES[\Delta M_{n+1}|^2]\le \sigma^2$ and $r\ge 4\lambda$. In this particular case, with the notations of \eqref{ineq:unn}, we  have:
$$u_{n+1}\le \left(1-\alpha_r \gamma_n\right)u_n+ r^2 \sigma^2\|Q_r\|^2 \gamma_{n+1}^2,$$
where $u_n = \mathbb{E} \|\cZn\|^2$.
The Propositions \ref{prop:controle_algo} $(iii)$ and \ref{prop:controle_algo2} $(iii)$ now imply that:
$$
 \mathbb{E}\left[\|\cZn\|^2\right] \leq  \mathbb{E}\left[\|\wch{Z}_0\|^2\right]e^{-\alpha_r \Gamma_n} + C_\gamma\frac{2r^2 \|Q_r\|^2 }{\alpha_r} \sigma^2 \gamma_n,
$$
which, in the end, provide an explicit upper bound of $\mathbb{E} \| Z_n\|^2$ since $Z_n = Q^{-1}_r \cZn$. 

A more important issue concerns the rate obtained when $\beta=1$ and we can remark in the statement of Proposition \ref{prop:vitesseexpo} that this rate depends on the size of $\gamma$ and of $\alpha_r$. In particular, the best rate (of order $\mathcal{O}(n^{-1})$) is obtained when $\gamma \alpha_r > 1$, meaning that $\alpha_r$ must be as large as possible to optimize the performance of the algorithm and we therefore obtain a non-adaptive rate. It is easy to see that $r \longmapsto \alpha_r$  increases  on $[0,4 \lambda]$
and decreases on $[4 \lambda,+\infty)$. It attains its maximal value ($\max_{r} \alpha_r = 4\lambda$) when $r=4 \lambda$. This maximal value is twice the size of the eigenvalue of the (standard) stochastic gradient descent (SGD).
Finally, $\lim_{r \longrightarrow + \infty} \alpha_r=2 \lambda$. This limiting value $2 \lambda$ corresponds to the size of the eigenvalue of the SGD. In other words, the limit $r=+\infty$ in HBF may be seen as an almost identical situation to SGD.

 If we compare the rate of convergence of  HBF to the one of SGD using the same step size $\gamma_n=\gamma n^{-1}$, we see that choosing a reasonably large $r$ makes it possible to obtain a less stringent condition on $\gamma$ to recover the (optimal) rate $\mathcal{O}(n^{-1})$. In particular, the rate of  the HBF is better when $r \geq 2\lambda$ than the one attained by the SGD. Unfortunately, it seems impossible  to obtain an adaptive procedure on the choice of $(\gamma,r)$  that guarantees the rate $\mathcal{O}(n^{-1})$, unlike the Polyak-Ruppert averaging procedure.

\end{rem}

\subsubsection{Polynomial memory $r_n =r \Gamma_n^{-1} \longrightarrow 0$}

This case is more intricate because of the variations with $n$ of the eigenvectors of the matrix $C_n$ defined in \eqref{eq:HBF_sto_linear2}. 

\begin{pro}\label{prop:quadpol}
 Assume $(\mathbf{H_{\bold{\sigma,2}}})$. Let $(Z_n)_{n\ge0}$ be defined by  \eqref{eq:HBF_sto_linear} with $Sp(S) \subset [\lambda,+ \infty[$ and $r_n=\frac{r}{\Gamma_n}$.
 \begin{itemize}
\item[$(i)$] If $\beta<1$ and $r>\frac{1+\beta}{2(1-\beta)}$, a constant $c_{\beta,\lambda,r}$ exists such that:
$$
\forall n \geq 1 \qquad 
\mathbb{E} \|X_n \|^2 \leq c_{\beta,\lambda,r}   \gamma_n,
$$
and 
$$
\forall n \geq 1 \qquad 
\mathbb{E} \|Y_n \|^2 \leq c_{\beta,\lambda}  \gamma_n r_n.
$$
\item[$(ii)$] If $\beta=1$, a constant $C$ exists such that:
$$
\forall n \geq 1 \qquad 
\mathbb{E} \|X_n \|^2 \leq \frac{C}{\log n}
$$
and 
$$
\forall n \geq 1 \qquad 
\mathbb{E} \|Y_n \|^2 \leq \frac{C}{n \log n}
$$

\end{itemize}
\end{pro}
\begin{rem}\label{rem:ratepolyy} We can observe that when $\beta<1$, the rates of the exponential case are preserved under a constraint on $r$ which becomes harder and harder when $\beta$ is close to $1$: $r$ needs to be greater than $\frac{1+\beta}{2(1-\beta)}$.   Carefully  following the proof of this result, we could in fact show that 
when $1/2<r<\frac{1+\beta}{2(1-\beta)}$, then $\mathbb{E} \|X_n \|^2 \leq C  n^{-(r-\frac{1}{2})(1-\beta)}$.  Since $(r-\frac{1}{2})(1-\beta)\longrightarrow 0$ as $\beta \longrightarrow 1$, our upper bound in $(\log n)^{-1}$ related to the case $\beta=1$ becomes reasonable. Another possible interpretation of the poor convergence rate in that case is that the size of the negative real part of the eigenvalues of $C_n$ is on the order $\frac{1}{n \log n}$, which leads to a contraction of the bias equivalent to $\mathcal{O}\left( e^{- c \sum_{1}^n \frac{1}{k \log k}} \right)$. Regardless of $c$, we cannot obtain  a polynomial rate of convergence in that case since 
$\sum_{1}^n \frac{1}{k \log k} \sim \log  \log n$.
\end{rem}

\begin{proof}

\noindent  \underline{Proof of $(i)$:} We study the case $\beta<1$ here. According to the arguments used in the proof of Proposition \ref{prop:vitesseexpo} and Subsection \ref{subs:redtotwo}, the dynamical system may be reduced to $d$ couples of systems in the form $(x_n^{(i)},y_n^{(i)})_{n \geq 1}$ so that we only make the proof for a system solution to \eqref{eq:HBF_sto_linear2} under  assumption \eqref{newmartingale}. Another key feature of the polynomial case has been observed in the proof of the a.s. convergence of the algorithm (Theorem \ref{theo:aspoly}): the study of the rate in the polynomial case involves a normalization of the algorithm with a $\sqrt{r_n}$-scaling of the $Y$ coordinate. Therefore, we set  $\tilde{Z}_n=(\tilde{X}_n,\tilde{Y}_n)$ with 
$\tilde{X}_n=X_n$ and $\tilde{Y}_n=Y_n/\sqrt{r_n}$. With these notations, we obtain (similar to Lemma \ref{lem:decompZNcheche}):
\begin{equation}
\tilde{Z}_{n+1}=  (I_2+\tilde{\gamma}_{n+1} \tilde{C}_n) \tilde{Z}_n + \tilde{\gamma}_{n+1} \sqrt{\frac{r_n}{r_{n+1}}} \Sigma_2 \Delta N_{n+1},
\end{equation}
with $\tilde{\gamma}_{n+1}=\gamma_{n+1} \sqrt{r_n}$ and:
$$
\tilde{C}_n = \left( \begin{matrix}
0 & -1 \\ \lambda \sqrt{\frac{r_n}{r_{n+1}}} & \rho_{n}
\end{matrix} \right)
$$
with 
$$\rho_n:=\frac{1}{\tilde{\gamma}_{n+1}}\left(\sqrt{\frac{r_n}{r_{n+1}}}-1\right)-\frac{r_n}{\sqrt{r_{n+1}}}.$$
Since $r_n=r\Gamma_n^{-1}$, the following expansion holds: 
\begin{equation}\label{eq:expansionrho}
\rho_n=\frac{1}{\sqrt{\Gamma_n}}\left(\frac{1}{2\sqrt{r}}-\sqrt{r}\right)+O\left(\frac{\gamma_n}{\Gamma_n^{\frac{3}{2}}}\right).
\end{equation}
In particular, for a large enough $n$, $\rho_n<0$ if and only if $r>1/2$.  Furthermore, an integer $n_0\in\EN$ exists such that for any $n\ge n_0$, $\tilde{C}_n$ has complex eigenvalues given by:
$$\mu_{\pm}^{(n)}=\frac{1}{2}\left(\rho_n\pm i \sqrt{4\lambda  \sqrt{\frac{r_n}{r_{n+1}}}-\rho_n^2}\right)\xrightarrow{n\rightarrow+\infty}\pm i\sqrt{\lambda}.$$
We define the diagonal matrix:
$$\Lambda_n:=\left( \begin{matrix}
\mu_+^{(n)}   &0 \\ 0& \mu_-^{(n)} 
\end{matrix} \right) $$ and let $Q_n$ be the matrix that satisfies
$
Q_n^{-1} \Lambda_n Q_n = \tilde{C}_n.$
We have:
$$Q_n^{-1} =  \left( \begin{matrix}1 & 1 \\
-\mu_+^{(n)}   & -\mu_-^{(n)}
\end{matrix} \right)\quad \textnormal{and}\quad Q_n =  \frac{1}{\mu_+^{(n)}-{\mu_-^{(n)}}}  \left( \begin{matrix} -{\mu_-^{(n)}} & -1 \\
\mu_+^{(n)}   & 1
\end{matrix} \right). $$
We can  now introduce  the change of basis brought by $Q_n$ and the new coordinates $\cZn :=Q_n \tZn$. We have:
\begin{eqnarray}\label{eq:ztildepoll}
\cZnp &= & Q_{n+1}(I_2+\tilde{\gamma}_{n+1}\tilde{C}_n) Q_n^{-1}\cZn + \tilde{\gamma}_{n+1}\sqrt{\frac{r_n}{r_{n+1}}} Q_{n+1}\Sigma_2  \Delta {N}_{n+1}\nonumber\\
& = &  Q_{n+1} Q_{n}^{-1}(I_2+\tilde{\gamma}_{n+1}\Lambda_n)\cZn+  \tilde{\gamma}_{n+1}\sqrt{\frac{r_n}{r_{n+1}}} Q_{n+1}\Sigma_2  \Delta {N}_{n+1}.\label{eq:cznppol} 
\end{eqnarray}
We now observe that:
$$Q_{n+1} Q_{n}^{-1}=I_2+\Upsilon_n \quad \textnormal{with}\quad \Upsilon_n=(Q_{n+1}-Q_n)Q_{n}^{-1}$$
and that for $n$ large enough:
$$\|\Upsilon_n\|_\infty\le C\|Q_{n+1}-Q_n\|_\infty=O(|\mu_+^{(n+1)}-\mu_+^{(n)}|)=O\left(|\rho_{n+1}-\rho_n|+|\mathfrak{Im}(\mu_{+}^{(n+1)}-\mu_{+}^{(n)})|\right).$$
 Expansion \eqref{eq:expansionrho},  the fact that
 $\sqrt{\frac{r_n}{r_{n+1}}}=1+\frac{1}{2}\frac{\gamma_{n+1}}{\Gamma_n}+O\left(\frac{\gamma_{n+1}^2}{\Gamma_n^2}\right)$
and  the Lipschitz continuity of $x\mapsto\sqrt{1+x}$ on $[-1/2,+\infty)$ yield:
$$\|\Upsilon_n\|_\infty=O\left(\frac{\gamma_n}{\Gamma_n^{\frac{3}{2}}}+\frac{\gamma_n-\gamma_{n-1}}{\Gamma_n}\right)=O\left(\frac{\gamma_n}{\Gamma_n^{\frac{3}{2}}}\right)=O\left(n^{-\frac{\beta+3}{2}}\right).$$
From the above, we obtain,  for any $z\in\ER^2$,
$$\|Q_{n+1} Q_{n}^{-1}(I_2+\tilde{\gamma}_{n+1}\Lambda_n)z\|^2\le \left[\left(1+\tilde{\gamma}_{n+1}\frac{\rho_n}{2}+O\left(\frac{\gamma_n}{\Gamma_n^{\frac{3}{2}}}\right)\right)^2+\left(\tilde{\gamma}_{n+1}\mathfrak{Im}(\mu_+^{(n)})+O\left(\frac{\gamma_n}{\Gamma_n^{\frac{3}{2}}}\right)\right)^2\right]\|z\|^2,$$
which after several computations yields:
$$\|Q_{n+1} Q_{n}^{-1}(I_2+\tilde{\gamma}_{n+1}\Lambda_n)z\|^2\le\left(1+\frac{\gamma_{n+1}}{\Gamma_n} \left(\frac{1}{2}-{r}+o(1)\right)\right)\|z\|^2.$$
Note that a universal constant $C$ (independent of $n$) exists such that  $\|Q_{n+1}\|_\infty\le C$ and the upper bounds above can be used into  \eqref{eq:cznppol} to deduce that:
\begin{equation}\label{ineq:refcaspol}
\|\cZnp\|^2\le \left(1+\frac{\gamma_{n+1}}{\Gamma_n} \left(\frac{1}{2}-{r} \right) + b \left(\frac{\gamma_{n+1}}{\Gamma_n}\right)^2\right)\|\cZn\|^2+\tilde{\gamma}_{n+1}\Delta\widecheck{M}_n+C\frac{\gamma_{n+1}^2}{\Gamma_n}\|\Delta N_{n+1}\|^2,
\end{equation}
where $(\Delta\widecheck{M}_n)_{n\ge1}$ is a sequence of martingale increments and $b$ a large enough constant.

When $\gamma_n= \gamma n^{-\beta}$ with $\beta<1$, the fact that $\Gamma_n= \frac{n^{1-\beta}}{1-\beta}+O(1)$ combined with  the upper bound of the variance of the martingale \eqref{newmartingale} imply that: 
\begin{equation}\label{eq:identpolind}
\ES[\|\cZnp\|^2]\le \left(1- \frac{\alpha}{n}+\frac{b}{n^2}\right)\ES[\|\cZn\|^2]+C n^{-1-\beta}
\end{equation}
where $\alpha:=(r-\frac{1}{2})(1-\beta)$.  Under the condition $r>\frac{1+\beta}{2(1-\beta)}$, we observe that:
$$\alpha >\beta.$$
An induction based on Inequality \eqref{eq:identpolind} yields:
\begin{eqnarray*}
\ES[\|\cZnp\|^2]&\le  &\ES[\|\widecheck{Z}_{n_\varepsilon}\|^2]\prod_{\ell=n_\varepsilon}^n\left(1-\frac{\alpha}{\ell} + \frac{b}{\ell^2}\right)+C\sum\limits_{k=n_\varepsilon+1}^n{k}^{-1-\beta}\prod_{\ell=k+1}^n\left(1-\frac{\alpha}{\ell} + \frac{b}{\ell^2}\right)
\\
&\le &   C n^{-\beta}
\end{eqnarray*}
where in the second line, we repeated an argument used in the proof of Propositions \ref{prop:controle_algo2}  and made use of the property $\alpha>\beta$. To conclude the proof, it remains to observe that $\|Q_{n+1}^{-1}\|_\infty\le C$ regardless of $n$. \hfill $\diamond$

\smallskip

$(ii)$ When $\beta=1$, Inequality \ref{ineq:refcaspol} leads to:
\begin{equation*}
\ES[\|\cZnp\|^2]\le \left(1- \frac{\alpha}{n\log n} + \frac{b}{n^2 \log n}\right)\ES[\|\cZn\|^2]+\frac{C}{n^2 \log n}
\end{equation*}
and a procedure similar to the one used above (given that $\sum_{k=1}^n (k\log k)^{-1}\sim \log(\log n)$) leads to the desired result.

\hfill $\diamond \square$
\end{proof}

\subsection{The non-quadratic case  under exponential memory }\label{sec:non_quadra}
The objective of this subsection is to extend the results of the quadratic case to strongly convex functions satisfying $\Hsc$ for a given positive $\alpha$.  As pointed out in Remark \ref{rem:commentsL2rate}, we are not able to obtain neat and somewhat intrinsic results in the polynomial memory case, so we therefore preferred to only consider the exponential memory one. 

With the help of  Subsection \ref{subs:redtotwo}, we can restrain the study to the situation where $d=1$ and $f$ has a unique minimum in $x^\star$ and we denote  
$\lambda=f''(x^\star)$, which is assumed to be positive.
 We also assume that 
$\underline{f''}=\inf_{x\in\ER} f''(x)>0$. 
It is worth noting that in this setting,  we are able to obtain some non-asymptotic bounds with some assumptions on $\lambda$ only. This means that our results do not involve the quantity $\underline{f''}$.
To only involve the value of the second derivative in $x^\star$, the main argument is a  \textit{power increase}  stated in the next lemma.
\begin{lem}\label{lem:powerincrement}
Let $(u_n^{(k)})_{n\ge0,k\ge1}$ be a sequence of non-negative numbers satisfying for every integers $n\ge0$ and $k\ge1$,
 \begin{equation}\label{eq:condun111}
 u_{n+1}^{(k)}\le (1-a_k\gamma_{n+1}+b_k\gamma_{n+1}^2)u_{n}^{(k)}+C_k(\gamma_{n+1}^2+\gamma_{n+1} u_{n}^{(k+1)})
 \end{equation}
 where $(a_k)_{k\ge1}$ and $(b_k)_{k\ge1}$ are  sequences of positive numbers.
Furthermore, assume that   $K\ge2$ exists and a constant $C>0$ exists such that:
\begin{equation}\label{eq:condun222}
\forall n\ge1, \quad u_n^{(K)}\le C\gamma_n.
\end{equation}
Then, suppose that $\gamma_n=\gamma n^{-\beta}$ $(\gamma>0$, $\beta\in(0,1])$ and that $\underline{a}:=\min_{k\le K} a_k>0$ and $\bar{b}:=\max_{k\le K} b_k<+\infty$. \smallskip

\noindent 
(i) If $\beta\in(0,1)$,  a constant $C>0$ exists such that for every $k\in\{1,\ldots,K\}$,
$$\forall n\ge1,\quad u_n^{(k)}\le C\gamma_n.$$

\noindent (ii) If $\beta=1$ and $\underline{a}\gamma>1$,   a constant $C>0$ exists such that  for every $k\in\{1,\ldots,K\}$,
\begin{equation}
\forall n\ge2, \quad u_n^{(k)}\le Cn^{-1}.
\end{equation}
\end{lem}
\begin{proof} Let $K\ge2$.   We proceed by a decreasing induction on $k\in\{1,\ldots,K\}$.  The initialization is given by \eqref{eq:condun222}. Then, let $k\in\{1,\ldots,K-1\}$ and assume that $u_n^{(k+1)}\le C_{k+1}\gamma_n$ (where $C_k$ is a positive constant that does not depend on $n$). We can use this upper bound in the second term of the right hand side of \eqref{eq:condun111} and obtain:
$$u_{n+1}^{(k)}\le (1-\underline{a}\gamma_{n+1}+\bar{b}\gamma_{n+1}^2)u_{n}^{(k)}+C\gamma_{n+1}^2$$
where $C$ is a constant that  does not depend on $n$.\smallskip

\noindent 
When $\beta<1$, it follows from Proposition \ref{prop:controle_algo}$(iii)$ that:
$$\forall n\ge1, \quad u_n^{(k)}\lesssim \gamma_n.$$\hfill $\diamond$

\noindent If $\beta=1$ and $\underline{a}\gamma>1$ now, the above control is a consequence of  Proposition \ref{prop:controle_algo2}$(iii)$. This concludes the proof. \hfill $\diamond\square$
\end{proof}

We will  apply this lemma to $u_{n}^{(k)}= \ES[|\cZn|^{2k}]$ where $\cZn$ is an appropriate linear transformation of $Z_n$

Therefore, we will mainly have to check that Conditions \eqref{eq:condun111} and \eqref{eq:condun222} hold.

\begin{pro}\label{pro:controVPP}
Assume  $\Hs$,  $\Hsc$ and $\mathbf{(H_{\sigma,\infty})}$ with $p\ge 1$. Let $a$ and $b$ be some positive numbers such that \eqref{eq:condabr} holds. Then,   an integer $K\ge1$ exists  such that
for any  $p\ge K$:
\begin{equation}\label{eq:controlvnp}
 \ES[V_n^p(X_n,Y_n)]\le C_p\gamma_n.
 \end{equation}
Furthermore, if $r_n=r$ and $\gamma_n=\gamma n^{-\beta}$ with $\beta\in(0,1)$, then \eqref{eq:controlvnp} holds for $p=K=1$ under $\mathbf{(H_{\sigma,2})}$ instead of $\Hsiginfty$.
As a consequence,
\begin{equation}\label{eq:controL2bbbb}
\ES[\|X_n-x^\star\|^{2K}+\|Y_n\|^{2K}]\le C\gamma_n.
\end{equation}
\end{pro}

\begin{rem} \label{rem:polylyapou} Note that the second assertion \eqref{eq:controL2bbbb} easily follows  from Equations \eqref{eq:minoVn} and \eqref{eq:controlvnp} and from the fact that under $\Hsc$, a constant $c$ exists such that for all $x$, $f(x)\ge c\|x\|^2.$

Moreover, note that this proposition is not restricted to the exponential memory case. In particular, as suggested in Remark \ref{rem:commentsL2rate}, this Lyapunov approach could lead to some (rough) controls of the quadratic error in the polynomial case when the function is not quadratic.
\end{rem}

\begin{proof} We begin by the first assertion under Assumption $\mathbf{(H_{\sigma,\infty})}$. Going back to the proof of Lemma \ref{lem:11111} (and to the associated notations), we obtain the existence of some positive $a$ and $b$ such that
\begin{equation*}
\begin{split}
&V_{n+1}(X_{n+1},Y_{n+1})\le V_n(X_n,Y_n)+\gamma_{n+1}\Delta_{n+1}\quad\textnormal{with}\\
&\Delta_{n+1}=-c_{a,b} \|Y_n\|^2- r_{n} b\|\nabla f(X_n)\|^2
-br_n\langle \nabla f(X_n),\Delta M_{n+1}\rangle+\Delta R_{n+1}\quad (c_{a,b}>0).
\end{split}
\end{equation*}
Denoting the smallest (positive) eigenvalue of  $D^2 f(x^\star)$ by $\underline{\lambda}$, we have:
$$\|\nabla f(x)\|^2\ge \underline{\lambda}\|x\|^2\ge C\ \underline{\lambda} f(x).$$
Following the arguments of the proof of Lemma \ref{lem:11111} once again, we can easily deduce the existence of some  positive $\varepsilon$ and $C$ such that:
$$\ES[\Delta_{n+1}|{\cal F}_n]\le (-\varepsilon +C\gamma_{n+1}) r_n V_n(X_n,Y_n)+C\gamma_{n+1} r_n.$$
Using $\mathbf{(H_{\sigma,\infty})}$, we also obtain for every $r\ge 1$:
$$\ES[\|\Delta_{n+1}\|^r|{\cal F}_n]\le C_r(1+V_n^{r}(X_n,Y_n)).$$
As a consequence, a binomial expansion of $(V_n(X_n,Y_n)+\gamma_{n+1}\Delta_{n+1})^K$ yields:
$$\ES[V_{n+1}^K(X_{n+1},Y_{n+1})|{\cal F}_n]\le (1-K\varepsilon\gamma_{n+1} r_n+C\gamma_{n+1}^2 r_n)V_n^K(X_n,Y_n)+C\gamma_{n+1}^2 r_n.$$
Setting $u_n=\ES[V_{n+1}^K(X_{n+1},Y_{n+1})]$, we obtain:
$$u_{n+1}\le (1-K\varepsilon\gamma_{n+1} r_n+C\gamma_{n+1}^2 r_n)u_n+C\gamma_{n+1}^2 r_n.$$
Now, assume that $\gamma_n=\gamma n^{-\beta}$ with $\beta\in(0,1]$ and successively consider  exponential and polynomial cases:\smallskip
\begin{itemize}
\item{} If $r_n=r$ and $\beta<1$, the result holds with $K=1$ by Proposition \ref{prop:controle_algo}$(iii)$. \hfill $\diamond$
\item{} If $r_n=r$ and $\beta=1$, we have to choose $K$ large enough in order that $K\varepsilon\gamma >1$. In this case, Proposition  \ref{prop:controle_algo2}$(iii)$ yields the result. \hfill $\diamond$
\item{} If  $r_n=r/{\Gamma_n}$  and $\beta<1$ now, then the above inequality yields the existence of a $\rho>\beta$ and a $n_0\ge1$ for $K$ large enough  such that:
$$\forall n\ge n_0,\quad u_{n+1}\le \left(1-\frac{\rho}{n}\right) u_n+C n^{-\beta-1}.$$
We have:
$$u_n\le u_{n_0}\prod_{k=n_0}^n \left(1-\frac{\rho}{k}\right)+ C \sum_{k=n_0+1}^n k^{-\beta-1}\prod_{\ell=k+1}^n\left(1-\frac{\rho}{k}\right).$$
Given that $1-x\le \exp(-x)$  and that $\sum_{k=1}^n \frac{1}{k}=\log n+ O(1)$, we obtain:
$$u_n\le C n^{- \rho}(1+\sum_{k=n_0+1}^n k^{-\beta-1+\rho})\le Cn^{-\beta}$$
where in the last inequality, we deduced that $-\beta-1+\rho>-1$ since $\rho<\beta$.
\end{itemize}
\hfill $\diamond \square$
\end{proof}
\begin{pro} \label{pro:vitesseexpononquad}Assume  $\Hs$,  $\Hsc$and $\Hsiginfty$ and $r_n=r$ for all $n\ge1$.   Set $\lambda=f''(x^\star)$. Then, assume that $\gamma_n=\gamma n^{-\beta}$ with $\beta\in(0,1]$.
\begin{itemize}
\item If $\beta<1$,
then:
$$\ES[\|X_n-x^\star\|^2]+\ES[\|Y_n\|^2]\le C\gamma_n.$$
\item If $\beta=1$, then for every $\varepsilon>0$,   a constant $C_\varepsilon$ exists such that
$$\ES[\|X_n-x^\star\|^2]\le C_\varepsilon n^{-((r+\varepsilon-\sqrt{r^2-4\lambda r}1_{r\ge 4\lambda}) \gamma)\wedge 1}.$$

\end{itemize}
\end{pro}
\begin{proof} The starting point is to linearize the gradient:
$$ f'(X_n)=\lambda (X_n-x^\star)+\phi_n\quad \textnormal{where}\quad \phi_n= (f''(\xi_n)-f''(x^\star))(X_n-x^\star).$$
Since  $f''$ is Lipschitz continuous, then:
\begin{equation}\label{eq:controllphinnn}
|\phi_n|\le C (X_n-x^\star)^2.
\end{equation}
Let us begin with the case where the matrix $C_n$ defined in \eqref{eq:HBF_sto_linear2} has real eigenvalues $\mu_+$ and $\mu_-$ (given by \eqref{eq:specC}).  With the notations introduced in \eqref{eq:cZnp},
\begin{equation}
\cZnp = \left( \begin{matrix}
1+\gamma_{n+1} \mu_+ & 0 \\ 0 & 1+\gamma_{n+1} \mu_-
\end{matrix} \right) \cZn+r\gamma_{n+1}Q \begin{pmatrix}0\\ \phi_n\end{pmatrix}+r \gamma_{n+1} \widecheck{\xi}_{n+1}.
\end{equation}
As a consequence, 
$$\|\cZnp\|^{2} \le  (1+  \mu_{+} \gamma_{n+1})^{2} \|\cZn\|^{2}+C\gamma_{n+1} \|\cZn\|^{3}+\gamma_{n+1}^2 ( \|\cZn\|^{4}+\|\Delta N_{n+1}\|^2)+\Delta {\cal M}_{n+1}$$
where $(\Delta {\cal M}_{n})$ is a sequence of martingale increments. Using the elementary inequality $|x|\le \varepsilon +C_\varepsilon |x|^2$, $x\in\ER$ (available for any $\varepsilon>0$),
$$\|\cZnp\|^{2} \le  [(1+  (2\mu_{+}+\varepsilon) \gamma_{n+1}+C\gamma_{n+1}^2) ]\|\cZn\|^{2}+C_\varepsilon \gamma_{n+1}  \|\cZn\|^{4}+C\gamma_{n+1}^2\|\Delta N_{n+1}\|^2+\Delta N_{n+1}.$$
Then, by Assumption $\Hsiginfty$  and the fact $\sup_{n}\ES[|\cZn|^{r}]<+\infty$ for any $r>1$ (by Proposition \ref{pro:controVPP} for example), we obtain, for any $k\ge1$,  
\begin{eqnarray*}
\mathbb{E} \left[ \|\cZnp\|^{2k} \right] \le  (1+  k(2\mu_{+}+\varepsilon) \gamma_{n+1}+C_k\gamma_{n+1}^2) \ES[\|\cZn\|^{2k}]+ C_{k,\varepsilon}(\gamma_{n+1}\ES[\|\cZn\|^{2k+2}]+\gamma_{n+1}^2).
\end{eqnarray*}
At this stage, we observe that Assumption \eqref{eq:condun111} is satisfied with $u_n^{(k)}=\ES[\|\cZn\|^{2k}]$ and $a_k=k(2\mu_{+}+\varepsilon) $. Using Proposition \ref{pro:controVPP} and Lemma \ref{lem:11111}$(i)$, we check that the second assumption of Lemma \ref{lem:powerincrement} also holds. Thus, the result follows in this case from this lemma.\smallskip
\hfill $\square$

\end{proof}

\section{Limit of the rescaled algorithm}\label{sec:TCL}

In this paragraph, we establish a (functional) Central Limit Theorem when the memory is exponential, $i.e.$, when $r_n=r$ 
and when $\Hsc$ holds. In particular, $f$ admits a unique minimum $x^\star$.  
Without loss of generality, we assume that $x^\star=0$.

\subsection{Rescaling stochastic HBF}

We start with an appropriate rescaling  by a factor $\sqrt{\gamma_n}$. More precisely,  we define
a sequence $(\babar{Z}_n)_{n\ge1}$:
$$\babar{Z}_n =\frac{ Z_n}{\sqrt{ \gamma_{n}}}=\left(\frac{X_n}{\sqrt{\gamma_n}},\frac{Y_n}{\sqrt{\gamma_n}}\right).$$ Given that $f$ is ${\cal C}^2$ (and that $x^\star=0$), we  ``linearize" $\nabla f$ around $0$ with a Taylor formula and obtain that $\xi_n\in[0,X_n]$ exists such that:
$$\nabla f(X_n)=D^2 f(\xi_n)X_{n}.$$
Therefore, we can compute that:
$$\babar{Z}_{n+1}=\babar{Z}_n+ \gamma_{n+1}b_n(\babar{Z}_n)+\sqrt{\gamma_{n+1}}\begin{pmatrix} 0\\ \Delta M_{n+1}\end{pmatrix}$$
where $b_n$ is defined by: 

\begin{equation}\label{eq:defbn}b_n(z)=\frac{1}{\gamma_{n+1}}\left(\sqrt{\frac{\gamma_{n}}{\gamma_{n+1}}}-1\right) z+ \bar{C}_n z, \quad z\in\ER^{2d},\end{equation}
where:
\begin{equation}\label{eq:defCn}
\bar{C}_n\,:=\,\sqrt{\frac{\gamma_{n}}{\gamma_{n+1}}}
\begin{pmatrix}
0 & -I_d\\
r D^2 f(\xi_n) &-r I_d\end{pmatrix}.
\end{equation}
It is important to observe that: 
\begin{equation}\label{eq:drift_normalise}
\frac{1}{\gamma_{n+1}} \left(\sqrt{\frac{\gamma_{n}}{\gamma_{n+1}}}-1\right) = \gamma^{-1} (n+1)^{\beta} \left[1+\frac{\beta}{2n}+o(n^{-1})-1\right] =
\left \{
\begin{array}{c @{ \, } c}
    o(n^{\beta-1}) &\qquad  \textit{if} \qquad  \beta<1 \\
    \frac{1}{2 \gamma} +o(1) &\qquad   \textit{if}\qquad   \beta=1\\
\end{array}
\right.
\end{equation}

We associate to the sequence $(\babar{Z}_n)_{n\geq 1}$ a sequence $(\babar{Z}^{(n)})_{n\ge1}$ of continuous-time processes defined by:
\begin{equation}\label{eq:babar}\babar{Z}^{(n)}_t= \babar{Z}_n+ B_t^{(n)}+M_t^{(n)}, \quad t\ge0,\end{equation}
where: 
$$B_t^{(n)} = \sum\limits_{k=n+1}^{\Ntilde(n,t)} \gamma_{k} b_{k-1}(\babar{Z}_{k-1})+(t-\underline{t}_n) b_{\Ntilde(n,t)}(\bar{Z}_{\Ntilde(n,t)}),$$
$$M_t^{(n)}=\sum\limits_{k=n+1}^{\Ntilde(n,t)}\sqrt{ \gamma_{k}}\begin{pmatrix} 0\\ \Delta M_{k}\end{pmatrix}+\sqrt{t-\underline{t}_n}\begin{pmatrix} 0\\ \Delta M_{\Ntilde(n,t)+1}\end{pmatrix}.$$
We used the standard notations $\underline{t}_n= \Gamma_{\Ntilde(n,t)}- \Gamma_n$ above where
$N(n,t) = \min\left\{m\geq n,  \sum\limits_{k=n+1}^m \gamma_{k}>t\right\}.$
\smallskip

\noindent To obtain a CLT, we  show that $(\babar{Z}^{(n)})_{n\ge1}$ converges in distribution to a stationary diffusion, following a classical roadmap based on a tightness result and on an identification of the limit as a solution to a martingale problem.

\subsection{Tightness}

The next lemma holds for any sequence of processes that satisfy \eqref{eq:babar}.
\begin{lem}\label{lem:tightnesszbarr} Assume that $D^2f$ is bounded, that $\sup_{k\ge1}\ES[\|\babar{Z}_k\|^2]<+\infty$ and that a $p>2$ exists such that $\sup_{k\ge1}\ES[\|\Delta M_k\|^{p}]<+\infty$, then $(\babar{Z}^{(n)})_{n\ge1}$ is tight  for the weak topology induced by the weak convergence on compact intervals.

\end{lem}
\begin{proof} 
\noindent First, note that $\babar{Z}_0^{(n)}=\babar{Z}_n$,  the assumption $\sup_{k\ge1}\ES[\|\babar{Z}_k\|^2]<+\infty$ implies the tightness of $(\babar{Z}_0^{(n)})_{n\ge1}$ (on $\ER^{2d}$). Then, by a classical  criterion  (see, $e.g.$, \cite[Theorem 8.3]{billingsley}), we deduce that  a sufficient condition for the tightness of $(\babar{Z}^{(n)})_{n\ge1}$ (for the weak topology induced by the uniform convergence on compacts intervals) is the following property: for any $T>0$, for any positive $\varepsilon$ and $\eta$,  a $\delta>0$ exist and an integer $n_0$  such that for any $t\in[0,T]$ and $n\ge n_0$, 
$$
\PE(\sup_{s\in[t,t+\delta]}\| \babar{Z}^{(n)}_s-\babar{Z}^{(n)}_t\|\ge\varepsilon)\le \eta\delta.
$$
We consider $B^{(n)}$ and $M^{(n)}$ separately and begin by the drift term $B^{(n)}$. On the one hand, 
$$
\PE\left(\sup_{s\in[t,t+\delta]}\| {B}^{(n)}_s-{B}^{(n)}_t\|\ge\varepsilon\right)\le \PE\left( \sum\limits_{k=N(n,t)}^{N(n,t+\delta)+1}\gamma_{k} \|b_{k-1}(\babar{Z}_{k-1})\|\ge\varepsilon\right).
$$
The Chebyschev inequality and the fact that $\|b_k(z)\|\le C(1+\|z\|)$ (where $C$ does not depend on $k$) yield:
$$ 
\PE\left(\sup_{s\in[t,t+\delta]}\| {B}^{(n)}_s-{B}^{(n)}_t\|\ge\varepsilon\right)\le   \varepsilon^{-2} \ES\left[\left(\sum\limits_{k=N(n,t)}^{N(n,t+\delta)+1}\gamma_{k} (1+\|(\babar{Z}_{k-1})\|)\right)^2\right]
$$
The Jensen inequality and the fact that $\sum_{k=N(n,t)}^{N(n,t+\delta)+1} \gamma_k\le 2\delta$
 when $n$ is large enough imply   that a constant $C$ exists such that for large enough $n$ and for a small enough $\delta$:
$$ 
\PE\left(\sup_{s\in[t,t+\delta]}\| {B}^{(n)}_s-{B}^{(n)}_t\|\ge\varepsilon\right)\le \varepsilon^{-2}\times C\delta^2(1+ \sup_{k\ge 1}\ES[\|\bar{Z}_k\|^2])\le \eta\delta$$
 \hfill $\diamond $

We now consider the martingale component  $M^{(n)}$: if we denote $\alpha=\sqrt{\frac{t-\underline{t}_n}{\gamma_{N(n,t)+1}}}$, we have for any $t\ge 0$, 
$$M_s^{(n)}=(1-\alpha) M_{N(n,s)}^{(n)}+\alpha M_{N(n,s)+1}^{(n)}$$
so that $\|M_s^{(n)}-M_t^{(n)}\|\le \max \{\|M_{N(n,s)}^{(n)}-M_t^{(n)}\|,\|M_{N(n,s)+1}^{(n)}-M_t^{(n)}\|\}$. As a consequence, 
$$ 
\PE\left(\sup_{s\in[t,t+\delta]}\| {M}^{(n)}_s-{M}^{(n)}_t\|\ge\varepsilon\right)\le 
\PE\left(\sup_{N(n,t)+1 \leq k \leq N(n,t+\delta)+1}\| {M}^{(n)}_{ \Gamma_k}-{M}^{(n)}_t\|\ge\varepsilon\right)$$
Let $p>2$ and applying the Doob inequality, the assumption of the lemma leads to:
 $$ 
\PE\left(\sup_{s\in[t,t+\delta]}\| {M}^{(n)}_s-{M}^{(n)}_t\|\ge\varepsilon\right)\le  \varepsilon^{-p} \ES\left[\| {M}^{(n)}_{\Ntilde(n,t+\delta)+1}-{M}^{(n)}_t\|^{p}\right]
$$
and the Minkowski inequality yields:
 $$ 
\PE\left(\sup_{s\in[t,t+\delta]}\| {M}^{(n)}_s-{M}^{(n)}_t\|\ge\varepsilon\right)\le \varepsilon^{-p}  \sum_{k=N(n,t)+1}^{N(n,t+\delta)+1} \gamma_k^{\frac{p}{2}}\ES\left[\|\Delta M_k\|^p\right].
$$
Under the assumptions of the lemma,  $\ES[[\|\Delta M_k\|^p]\le C$. Furthermore, we can use the rough upper bound:
$$\sum_{k=N(n,t)+1}^{N(n,t+\delta)+1} \gamma_k^{\frac{p}{2}}\le  \gamma_n^{\frac{p}{2}-1}\sum_{k=N(n,t)+1}^{N(n,t+\delta)+1} \gamma_k\le \eta\delta$$
for large enough $n$. This concludes the proof. \hfill $\diamond\square$



\end{proof}

\begin{cor} Let the assumptions of Theorem \ref{theo:TCL} hold, then $(\babar{Z}^{(n)})_{n\ge1}$ is tight.
\end{cor}

\begin{proof} To prove this result, it is enough to check that the assumptions of Lemma \ref{lem:tightnesszbarr} are satisfied. First, one remarks that the assumptions of Theorem \ref{theo:TCL} imply the ones of Theorem \ref{theo:rates}$(a)$ so that $\ES[\|Z_n-z^\star\|^2]\le C\gamma_n$ (this also holds when $\beta=1$ since we assume that $\gamma\alpha_r>1$). As a consequence, $\sup_{k\ge1}\ES[\|\babar{Z}_k\|^2]<+\infty$.\smallskip

\noindent On the other hand, since $\mathbf{(H_{\sigma,p})}$ holds for a given $p>2$, we can derive by following the lines of the proof of Proposition \ref{pro:controVPP} that $\sup_{n\ge1}\ES[V^p(X_n,Y_n)]<+\infty$.
As a consequence, $\sup_n\ES[f^p(X_n)]<+\infty$ and $\mathbf{(H_{\sigma,p})}$ leads to:
$$\sup_{n\ge1}\ES[\|\Delta M_n\|^{p}]\lesssim \sup_n\ES[f^p(X_n)]<+\infty.$$
\hfill $\square$
\end{proof}

\subsection{Identification of the limit}

\noindent Starting from our compactness result above, we now characterize the  potential weak limits of $(\bar{Z}^{(n)})_{n\ge1}$. This step is strongly based on the following lemma.
\begin{lem}\label{lem:identificationzbarr} Suppose that the assumptions of Lemma \ref{lem:tightnesszbarr} hold and that:
$$\ES[\Delta M_n(\Delta M_n)^t|{\cal F}_{n-1}]\xrightarrow{n\rightarrow+\infty} \varlimit\quad \textnormal{in probability},$$
where $\sigma^2$ is a positive symmetric $d\times d$-matrix. Then, for every $C^2$-function 
 $g:\ER^{2d}\rightarrow\ER$, compactly supported with   Lipschitz continuous second derivatives, we have:
$$\mathbb{E}(g(\babar{Z}_{n+1})-g(\babar{Z}_{n})|\mathcal{F}_n) =  \gamma_{n+1}\mathcal{L}g(\babar{Z}_{n})+R_n^{g}$$
where $ \gamma_{n+1}^{-1} R_n^{g}\rightarrow 0$ in $L^1$ and $\mathcal{L}$ is the infinitesimal generator defined in Theorem \ref{theo:TCL}.

\end{lem}
\begin{rem}\label{rem:underdif} We recall that  $\mathcal{L}$ is the infinitesimal generator  of the following stochastic differential equation:
$$d\babar{Z}_t = \babar{H}\babar{Z}_tdt + \Sigma dB_t$$
where: $\babar{H}=\frac{1}{2\gamma }1_{\{\beta=1\}} I_{2d}+H$ and $\Sigma$ is defined in Theorem \ref{theo:TCL}. $(\babar{Z}_t)_{t \geq 0}$ lies in the family of  Ornstein-Uhlenbeck processes: on the one hand, the drift and diffusion coefficients being respectively linear and constant, $(\babar{Z}_t)_{t \geq 0}$ is a Gaussian diffusion; on the other hand,  since $\babar{H}$ has negative eigenvalues, $(\babar{Z}_t)_{t \geq 0}$ is ergodic. 
\end{rem}

\begin{proof}
$C$ will denote an absolute constant whose value may change from line to line, for the sake of convenience.
We use a Taylor expansion between $\babar{Z}_{n}$ and $\babar{Z}_{n+1}$ and obtain that $\theta_n$ exists in $[0,1]$ such that:
\begin{eqnarray}
g(\babar{Z}_{n+1})-g(\babar{Z}_{n}) &=& \langle\nabla g(\babar{Z}_{n}),(\babar{Z}_{n+1}-\babar{Z}_{n})\rangle+\frac{1}{2}(\babar{Z}_{n+1}-\babar{Z}_{n})^TD^2 g(\babar{Z}_{n})(\babar{Z}_{n+1}-\babar{Z}_{n})\label{eqfopiopidpisdfoipsd}\\
&+& \underbrace{\frac{1}{2}(\babar{Z}_{n+1}-\babar{Z}_{n})^T(D^2g(\theta\babar{Z}_{n} + (1-\theta)\babar{Z}_{n+1} )-D^2 g(\babar{Z}_{n}))(\babar{Z}_{n+1}-\babar{Z}_{n})}_{R_{n+1}^{(1)}}.\nonumber 
\end{eqnarray}
We first deal with the remainder term  $R_{n+1}^{(1)}$ and observe that $(\bar{C}_n)$ introduced in \eqref{eq:defCn} is uniformly bounded so that a constant $C$ exists such that  $\|b_n(z)\|\le C\| z\|$. We thus conclude that:
$$\|  \babar{Z}_{n+1}-\babar{Z}_{n}\|\le C\left(\gamma_{n+1}\| \babar{Z}_{n}\|+\sqrt{\gamma_{n+1}}\|\Delta M_{n+1}\|\right).$$
Using $\Hsig$, we deduce that for any ${\bar p}\le p$,
\begin{equation}\label{dsiudiqoidu}
\mathbb{E}\left[\|  \babar{Z}_{n+1}-\babar{Z}_{n}\|^{\bar{p}}\right]\le C \gamma_{n+1}^{\frac{\bar{p}}{2}}.
\end{equation}
Since $D^2g$ is Lipschitz continuous and compactly supported, $D^2 g$ is also $\varepsilon$-Hölder for all $\varepsilon\in(0,1]$. We choose $\varepsilon$   such that 
$2+\varepsilon\le p$ and obtain:
\begin{eqnarray*}
\mathbb{E}\left[\vert R_{n+1}\vert\right]\leq C\mathbb{E}\left[\|  \babar{Z}_{n+1}-\babar{Z}_{n}\|^{2+\varepsilon}\right] \leq C\gamma_{n+1}^{1+\frac{\varepsilon}{2}}.
\end{eqnarray*}
We deduce that $\gamma_{n+1}^{-1} R_{n+1}^{(1)}\rightarrow 0$ in $L^1$.  \hfill $\diamond$


\noindent
Second, we can express \eqref{eq:drift_normalise} when $\gamma_n=\gamma n^{-\beta}$ with $\beta\in(0,1]$ in the following form:
$$\epsilon_n:= \frac{1}{\gamma_{n+1}}\left(\sqrt{\frac{\gamma_{n}}{\gamma_{n+1}}}-1\right)  - \frac{1}{2\gamma }1_{\{\beta=1\}} = o(1).$$
Then, given that $D^2f $ is Lipschitz (and that $x^\star=0$), it follows that:
$$\forall z \in \RR^d \times \RR^d \qquad \left\| b_n (z)-\left(\frac{1}{2\gamma}1_{\{\beta=1\}} I_{2d}+H\right)  z\right\|\le (\varepsilon_{n}+\|\bar{X}_n\|) \|z\|$$
where $(\varepsilon_n)_{n\ge1}$ is a deterministic sequence such that $\lim_{n\rightarrow+\infty}\varepsilon_n=0$. 

Under the conditions of Theorem \ref{theo:TCL}, we may apply the convergence rates obtained in Theorem \ref{theo:rates} and observe that $ \sup_{n}\ES[\|X_n\|^2] \lesssim \gamma_n $, meaning that $\sup_{n}\ES[\|\babar{Z}_n\|^2]<+\infty$.  Since $\|\babar{X}_n\| \leq \|\babar{Z}_n\|$, we deduce that: 
$$\ES[\langle\nabla g(\babar{Z}_{n}),(\babar{Z}_{n+1}-\babar{Z}_{n})\rangle|{\cal F}_n]= \gamma_{n+1}\langle \nabla g(\babar{Z}_n),(\frac{1}{4\gamma \sqrt{r}}1_{\{\beta=1\}} I_{2d}+H) \babar{Z}_n\rangle+R_n^{(2)}$$
where  $\gamma_{n+1}^{-1} R_{n}^{(2)}\rightarrow 0$ in $L^1$ as $n\rightarrow+\infty$. Let us now consider the second term of the right-hand side of \eqref{eqfopiopidpisdfoipsd}. We have:
$$\ES[(\babar{Z}_{n+1}-\babar{Z}_{n})^TD^2 g(\babar{Z}_{n})(\babar{Z}_{n+1}-\babar{Z}_{n})|{\cal F}_n]= \gamma_{n+1}\sum_{i,j} D^2_{y_i y_j} g(\babar{Z}_{n})\ES[\Delta M_{n+1}^i\Delta M_{n+1}^j|{\cal F}_n]
+ R_{n}^{(3)}$$
where 
$$| \gamma_{n+1}^{-1}R_n^{(3)}|\le C  \gamma_{n+1}\|\babar{Z}_n\|^2\xrightarrow{n\rightarrow+\infty}0\quad \textnormal{in $L^1$}$$
under the assumptions of the lemma. To conclude the proof, it remains to note that under the assumptions of the lemma for any $i$ and $j$, $(\ES[\Delta M_{n+1}^i\Delta M_{n+1}^j|{\cal F}_n])_{n\ge1}$ is a uniformly integrable sequence
that satisfies:
 $$\ES[\Delta M_{n+1}^i\Delta M_{n+1}^j|{\cal F}_n]=\varlimit_{i,j}\quad\textnormal{in probability.}$$
Thus, the convergence also holds in $L^1$.
The conclusion of the lemma easily follows from the boundedness of $D^2g$.\hfill $\diamond \square$
\end{proof}
\smallskip

We are now able to prove Theorem  \ref{theo:TCL}:
\smallskip

\noindent \textbf{Proof of Theorem \ref{theo:TCL}, $(i)$}: Note that under the assumptions of Theorem \ref{theo:TCL}, we can apply  Lemma \ref{lem:tightnesszbarr}  and Lemma \ref{lem:identificationzbarr} and obtain that the sequence of processes $(\babar{Z}^{(n)})_{n\ge1}$ is tight. The rest of the proof is then  divided into two steps. In the first one, we prove that every weak limit of $(\babar{Z}^{(n)})_{n\ge1}$ is a  solution of  the  martingale problem $(\mathcal{L},\mathcal{C})$ where ${\cal C}$ denotes the class of ${\cal C}^2$-functions with compact support and Lipschitz-continuous second derivatives.  Before going further, let us recall that, owing to the Lipschitz continuity of the coefficients,  this martingale problem is well-posed, $i.e.$, tha,t existence and uniqueness hold for the weak solution starting from a given initial distribution $\mu$ (see, $e.g.$, \cite{ethier-kurtz} or \cite{SV}).

In a second step, we prove the uniqueness of the invariant distribution related to the operator $\mathcal{L}$ and the convergence in distribution to this invariant measure.  We end this proof by showing that  $(\babar{Z}^{(n)})$ converges to this invariant distribution, so that  the sequence $(\babar{Z}^{(n)})_{n\ge1}$ converges to a stationary solution of the previously introduced martingale problem.
 We will characterize this invariant (Gaussian) distribution in the next paragraph.\smallskip

\noindent \textbf{Step 1}: Let $g$ belong to ${\cal C}$ and let $({\cal F}_t^{(n)})_{t\ge0}$ be the natural filtration of $\babar{Z}^{(n)}$. To prove that any weak limit of $(\babar{Z}^{(n)})_{n\ge1}$ solves  the martingale problem $(\mathcal{L}, {\cal C})$, it is enough to show  that: 
$$\forall t\ge 0,\quad  g(\babar{Z}^{(n)}_t)-g(\babar{Z}^{(n)}_0)-\int_0^t \mathcal{L} g(\babar{Z}^{(n)}_s)ds={\cal M}_t^{(n,g)}+{\cal R}_t^{(n,g)}$$
where $({\cal M}_t^{(n,g)})_{t\ge0}$ is an $({\cal F}_t^{(n)})$-adapted martingale and ${\cal R}_t^{(n,g)}\rightarrow0$ in probability for any $t\ge0$.
We set: 
$${\cal M}_t^{(n,g)}=\sum_{k=n+1}^{N(n,t)}g(\babar{Z}_{k+1})-g(\babar{Z}_{k})-\ES[g(\babar{Z}_{k+1})-g(\babar{Z}_{k})|{\cal F}_{k-1}].$$
By construction, $({\cal M}_t^{(n,g)})_{t\ge0}$ is an $({\cal F}_t^{(n)})$-adapted martingale (given that ${\cal F}_s^{(n)}={\cal F}_{\underline{s}_n}^{(n)}$) and: 
\begin{align*}
{\cal R}_t^{(n,g)}=g(\babar{Z}^{(n)}_t)-g(\babar{Z}^{(n)}_{\underline{t}_n})-\int_{\underline{t}_n}^t \mathcal{L} g(\babar{Z}^{(n)}_s)ds +\int_{0}^{\underline{t}_n} \left(\mathcal{L} g(\babar{Z}^{(n)}_{\underline{s}_n})- \mathcal{L} g(\babar{Z}^{(n)}_s)\right)ds
+\sum_{k=n}^{N(n,t)-1} R_k^g
\end{align*}
where  $(R_k^g)_{k\ge1}$ has been defined in Lemma \ref{lem:identificationzbarr}. Using an argument similar to \eqref{dsiudiqoidu},
we can check that for any $t\ge0$:
$$\sup_{s\le t}\ES[\|\babar{Z}^{(n)}_s-\babar{Z}^{(n)}_{\underline{s}_n}\|^2]\le C\sqrt{\gamma}_n.$$
This inequality combined with the Lipschitz continuity of $g$ and its derivatives  implies that the  first three terms tend to $0$ when $n\rightarrow+\infty$. Now, concerning   the last one, the previous lemma yields: 
$$\ES\left[\left|\sum_{k=n}^{N(n,t)-1} R_k^g\right|\right]\le C t\sup_{k\ge n}\ES\left[\left| \gamma_{k}^{-1}R_k^g\right|\right]\xrightarrow{n\rightarrow+\infty}0.$$ \hfill $\diamond $

\noindent \textbf{Step 2}:  First, let us prove that uniqueness holds for the invariant distribution  related to  $\mathcal{L}$. We denote it by $\mu_{\infty}^{(\beta)}$ below.  In this simple setting where the coefficients are linear, we could use the fact that 
the process, which is solution to the martingale problem, is Gaussian so that any invariant distribution is so. Uniqueness could then be deduced through  the characterization of the mean and the variance through the  relationship $\int {\cal L}f(x) \mu_{\infty}^{(\beta)}(dx) =0$ (see next subsection for such an approach). However,  at this stage, we prefer to use a more general strategy related to the hypoellipticity of ${\cal L}$  (see, $e.g.$, \cite{GadatPanloup} for a similar approach). More precisely, 
set 
$L_D := - \langle y , \partial_x \rangle + r \langle D^2 f(x^\star) x - y],  \partial_y \rangle $ and $\sigma_i := \sum_{j=1}^d \sigma_i^j \partial_{y_j}$, 
where $\sigma$  satisfies $\sigma\sigma^t=\mathcal{V}$ (where $\mathcal{V}$ is defined by \eqref{eq:hypdeltamnproba}). We have assumed that $\sigma$ is invertible, so that:
$$
\text{span} (\sigma_1, \ldots, \sigma_d ) = \text{span}(\partial_{y_1}, \ldots, \partial_{y_d}).
$$
Therefore, 
$$
\text{Lie} \left( L_D,\sigma_1,\ldots, \sigma_d\right) = \text{Lie} \left( L_D,\partial_{y_1}, \ldots, \partial_{y_d}\right) 
$$
Now, it is straightforward to check that:
$$
\forall i \in \{1, \ldots, d\} \qquad 
\left[ L_D ,\partial_{y_i} \right](f) = - \partial_{x_i} (f),
$$
and we deduce that $\text{Lie} \left( L_D,\sigma_1,\ldots, \sigma_d\right) = \text{Lie} \left( \partial_{x_1},\ldots,\partial_{x_d},\partial_{y_1}, \ldots, \partial_{y_d}\right)$.
This means  that the Hormand\"er bracket condition holds at any point $z$ of $\ER^{2d}$, which implies that the process admits a density $(p_t(z,.))_{t\ge0}$ such that for any $t>0$, $(z,z')\mapsto p_t(z,z')$, which is smooth on $\ER^{2d}\times\ER^{2d}$. It is moreover possible to show that these densities are positive, for any $t>0$, given that the linear vector field is approximately controllable: for any time $T>0$, any $\eta>0$ and any couple of initial points $(x_0,y_0)$ and ending points $(x_T,y_T)$, we can build a function $\varphi$  such that $\dot \varphi \in \mathbb{L}^2$ and such that the controlled trajectory:
\begin{equation}\label{controlpb}
\left\{
\begin{array}{ccl}
\dot {x}(t) &= &-y(t)   \\
\dot {y}(t)& = &r(t)(\nabla U (x(t)) - y(t)) + \sigma \dot \varphi,
\end{array}
\right.
\end{equation}
satisfies: $z_0=(x_0,y_0)$ and $\|z_T-(x_T,y_T)\|\leq \eta$. This implies the irreducibility of the diffusion and, therefore, the uniqueness of the invariant distribution. We refer to \cite{GadatPanloup} for more details on this controllability problem.

Then, checking that ${\cal L}\|x\|^2\le \beta-\alpha \|x\|^2$ for positive $\alpha$ and $\beta$, it can be classically deduced from the Meyn-Tweedie-type arguments (see \cite{meyn-tweedie}) that the process converges locally uniformly, exponentially fast in total variation  to $\mu_{\infty}^{(\beta)}$.  For more details, we refer to    \cite[Theorem 4.4]{mattingly_stuart}. Below, we will only use the following corollary: for any bounded Lipschitz-continuous function $f$,  for any compact set $K$ of $\ER^{2d}$, 
\begin{equation}\label{eq:convuniformtoinvmes}
\sup_{z\in K}|P_t f(z)-\mu_{\infty}^{(\beta)}(f)|\xrightarrow{t\rightarrow+\infty}0
\end{equation}
where $(P_t)_{t\ge0}$ denotes the  semi-group related to the (well-posed) martingale problem $({\cal L},{\cal C})$. \hfill $\diamond$
\smallskip

\noindent \textbf{Step 3}: 
\noindent Let $(\babar{Z}_{n_k})_{k\ge1}$ be a (weakly) convergent subsequence of $(\babar{Z}_n)_{n\ge1}$ to a probability $\nu$. We have to prove that $\nu=\mu_\infty^{(\beta)}$. To do this, we take advantage of the ``shifted'' construction of the sequence $(\babar{Z}^{(n)})_{n \in \mathbb{N}}$. More precisely, as a result of construction, for any positive $T$, a sequence $(\psi(n_k,T))_{k\ge1}$  exists  such that: 
 $$N(T,\psi(n_k,T))=n_k.$$
 In other words,
 $$\babar{Z}^{(\psi(n_k,T))}_{\underline{T}_{\psi(n_k,T)}}=\bar{Z}_{n_k}.$$
 At the price of a potential extraction, $(\babar{Z}^{(\psi(n_k,T))})_{k\ge1}$ is convergent to a continuous process, which is denoted by $Z^{\infty,T}$ below. Given that 
 $\babar{Z}^{(n)}_T-\babar{Z}^{(n)}_{\underline{T}_n}$ tends to $0$ as $n\rightarrow+\infty$ in probability, it follows that 
 $Z^{\infty,T}_T$ has distribution $\nu$. However, according to Step 1, $Z^{\infty,T}$ is also a solution to the martingale problem $({\cal L},{\cal C})$ so that  for any Lipschitz continuous function $f$,
 $$\ES[f(Z^{\infty,T}_T)]-\mu_\infty^{(\beta)}(f)=\int_{\ER^{2d}} \left(P_T f(z)-\mu_\infty^{(\beta)}(f)\right)\PE_{Z_0^{\infty,T}}(dz).$$
 Denote by ${\cal P}$, the set of weak limits of $(\babar{Z}_n)_{n\ge1}$. ${\cal P}$ is tight and as a result of construction, 
 $Z_0^{\infty,T}$ belongs to ${\cal P}$. Thus, for any $\varepsilon>0$,  a compact set $K_\varepsilon$  exists such that for any $T>0$, 
$$\left |\int_{K_\varepsilon^c} \left(P_T f(z)-\mu_\infty^{(\beta)}(f)\right)\PE_{Z_0^{\infty,T}}(dz)\right|\le 2\|f\|_\infty\sup_{\mu\in{\cal P}}\mu(K_\varepsilon^c)\le2\|f\|_\infty \varepsilon.$$
On the other hand, 
$$\left| \int_{K_\varepsilon}\left( P_T f(z)-\mu_\infty^{(\beta)}(f)\right)\PE_{Z_0^{\infty,T}}(dz)\right|\le \sup_{z\in K_\varepsilon} |P_T f(z)-\mu_\infty^{(\beta)}(f)|$$
and it follows from Step 2 that the right-hand member tends to $0$ as $T\rightarrow+\infty$.  From this, we can therefore conclude that 
for any bounded Lipschitz-continuous function $f$, a large enough  $T$ exists such that:
$$\left|\ES[f(Z^{\infty,T}_T)]-\mu_\infty^{(\beta)}(f)\right|\le C_f\varepsilon.$$
Since $\ES[f(Z^{\infty,T}_T)]=\nu(f)$, it follows that $\nu(f)=\mu_\infty^{(\beta)}(f)$. Finally, the set $\cal P$ is reduced to a single element ${\cal P}=\{\mu_\infty^{(\beta)}\}$, and the whole sequence $(\babar{Z}_n)_{n \geq 1}$ converges to $\mu_{\infty}^{(\beta)}$.

\smallskip

\noindent  Before ending this section, let us note that   $\mu_{_{\infty}}^{(\beta)}$ is a Gaussian centered distribution is a simple consequence of Remark \ref{rem:underdif}. We therefore leave this point to the reader.
\hfill $\diamond  \square$
\subsection{Limit variance}

We end this section on the analysis of the rescaled algorithm with some considerations on the invariant measure $\mu_{_{\infty}}^{(\beta)}$ involved in Theorem \ref{theo:TCL} for the exponential memoried stochastic HBF, \textit{i.e.} when $r_n=r$. As shown in the above paragraph,  this invariant measure describes the exact asymptotic variance of the initial algorithm. We now focus on its characterization $i.e.$, on the proof of Theorem \ref{theo:TCL}$(ii)$. In particular, to ease the presentation, we assume that the covariance matrix $\varlimit$ related to  $(\Delta M_{n+1})_{n \geq 1}$ is proportional to the identity matrix:

\vspace{0.5em}
\noindent
 
\begin{equation}\label{eq:HSigma}
\lim_{n \longrightarrow + \infty} \mathbb{E}\left[ \Delta M_{n+1} \left(\Delta M_{n+1}\right)^t \vert \mathcal{F}_n \right] = \sigma_0^2 I_d\quad\textnormal{in probability.}
\end{equation}
We also assume that  $\gamma_n =\gamma n^{-\beta}$ with $\beta<1$. Then, $(i)$ of Theorem \ref{theo:TCL} states that   $(\babar{Z}_n)_{n \geq 1}$ weakly converges toward a diffusion process, whose generator $\mathcal{L}$ is the one of an Ornstein-Uhlenbeck process. Assumption \eqref{eq:HSigma} leads to a simpler expression:
\begin{equation}\label{eq:Lbeta}
\mathcal{L}(\phi)(x,y)  = - \langle y, \nabla_{x} \phi \rangle + r  \langle D^2(f)(x^{\star}) x - y ,\nabla_{y} \phi \rangle + r^2 \frac{{\sigma_0}^2}{2} \Delta_y   \phi.
\end{equation}
A particular feature of Equation \eqref{eq:Lbeta} when $\gamma_n=\gamma n^{-\beta}$ is that   $\mathcal{L}$ does not depend on $\beta$ nor $\gamma$. The invariant measure $\mu_{_{\infty}}^{(\beta)}$ is a multivariate Gaussian distribution that may be well described in the basis given by the eigenvectors of the Hessian $D^2(f)(x^\star)$. The reduction to $d$ couples of two-dimensional system used in Section \ref{subs:redtotwo} makes it possible to use the spectral decomposition of $D^2(f)(x^\star)=P^{-1} \Lambda P$ where $P$ is an orthonormal matrix and $\Lambda$ a diagonal matrix with positive eigenvalues.
The process $(\cXn,\cYn)=(P \babar{X}_n,P\babar{Y}_n)$ is therefore centered and Gaussianly distributed asymptotically. This process is associated with $d$ $2 \times 2$ blockwise independent Ornstein-Uhlenbeck processes, whose generator is now
$$
\check{\mathcal{L}}(\phi)(\check{x},\check{y}) = - \langle \check{y}, \nabla_{\check{x}} \phi \rangle + r   \langle \Lambda \check{x} - \check{y} ,\nabla_{\check{y}} \phi \rangle + r^2 \frac{{\sigma_0}^2}{2} \Delta_{\check{y}}   \phi,
$$ 
where we used $\text{Tr}\left( P^t D^{2}_{\check{y}} P \right) = \text{Tr}\left(  D^{2}_{\check{y}} P P^t \right) = 
\text{Tr}\left(  D^{2}_{\check{y}} \right)$ in the last line  because $P^t P=I_d$.
If we denote $\check{\mu}_{_{\infty}}^{(\beta)}$ the associated invariant gaussian measure, the tensor structure of $\check{\mathcal{L}}$ leads to 
\begin{equation}\label{eq:covij}
\forall i \neq j \qquad \mathbb{E}_{(\check{x},\check{y}) \sim \check{\gamma}_{_{\infty}}^{\beta}} [\check{x}^{(i)} \check{x}^{(j)} ]=
\mathbb{E}_{(\check{x},\check{y}) \sim \check{\gamma}_{_{\infty}}^{\beta}} [\check{x}^{(i)} \check{y}^{(j)} ] = \mathbb{E}_{(\check{x},\check{y}) \sim \check{\gamma}_{_{\infty}}^{\beta}} [\check{y}^{(i)} \check{y}^{(j)} ] =0.
\end{equation}
Now, using the relationship $\displaystyle \int \check{\mathcal{L}}(\phi) d  \check{\mu}_{_{\infty}}^{(\beta)}=0$ for some well chosen functions $\phi$, we can identify the rest of the covariance matrix. Denote $i$ any integer in $\{1,\ldots,d\}$. We chose  $\phi (\check{x},\check{y})= \frac{\left\{ \check{x}^{(i)}\right\}^2}{2}$ and obtain that
$ \check{\mathcal{L}}\left(\frac{\left\{ \check{x}^{(i)}\right\}^2}{2} \right)(\check{x},\check{y}) = -   \check{x}^{(i)}  \check{y}^{(i)}$. It then implies that 
\begin{equation}\label{eq:covxy}
 \mathbb{E}_{(\check{x},\check{y}) \sim \check{\mu}_{_{\infty}}^{(\beta)}} [\check{x}^{(i)} \check{y}^{(i)} ] = 0.
\end{equation}
Picking now $ \phi (\check{x},\check{y})=\frac{\left\{ \check{y}^{(i)}\right\}^2}{2} $, we obtain $ \check{\mathcal{L}}\left(
\frac{\left\{ \check{y}^{(i)}\right\}^2}{2} \right)(\check{x},\check{y}) = r \lambda_i  \check{x}^{(i)}\check{y}^{(i)} - r  \left\{ \check{y}^{(i)}\right\}^2 + \frac{r^2 {\sigma_0}^2}{2}$ so that
\begin{equation}\label{eq:vary}
 \mathbb{E}_{(\check{x},\check{y}) \sim \check{\mu}_{_{\infty}}^{(\beta)}} [ \{\check{y}^{(i)}\}^2 ] = \frac{r {\sigma_0}^2}{2}.
\end{equation}
Finally, we chose $ \phi (\check{x},\check{y})= \check{x}^{(i)}  \check{y}^{(i)} $ and obtain 
$ \check{\mathcal{L}}\left( \check{x}^{(i)}  \check{y}^{(i)}  \right)(\check{x},\check{y}) = - \left\{ \check{y}^{(i)}\right\}^2 +  r \lambda_i \{ \check{x}^{(i)} \}^2 - r  \check{x}^{(i)}  \check{y}^{(i)}$. Therefore, we deduce that:
\begin{equation}\label{eq:varx}
 \mathbb{E}_{(\check{x},\check{y}) \sim \check{\mu}_{_{\infty}}^{(\beta)}} [ \{\check{x}^{(i)}\}^2 ] = \frac{ {\sigma_0}^2}{2 \lambda_i}.
\end{equation}
We can sum-up formulae \eqref{eq:covij}-\eqref{eq:varx} in $\check{\mu}_{_{\infty}}^{(\beta)} = \mathcal{N}\left(0,D_{r,{\sigma_0}}\right)$ with
$
D_{r,{\sigma_0}} = \frac{{\sigma_0}^2}{2} \left( \begin{matrix} \Lambda^{-1} & \mathbf{0}_{d\times d} \\
 \mathbf{0}_{d\times d} & r I_d
 \end{matrix}\right).
$
Since $(\babar{X}_n,\babar{Y}_n)=(P^{-1}\cXn,P^{-1}\cYn)$, we deduce that: 
$$\mu_{_{\infty}}^{(\beta)} = \mathcal{N}\left(0, \frac{{\sigma_0}^2}{2} 
 \left( \begin{matrix} \{D^2f(x^\star)\}^{-1} & \mathbf{0}_{d\times d} \\
 \mathbf{0}_{d\times d} & r I_d
 \end{matrix}\right)
\right).$$ \hfill $\diamond$

\smallskip

\noindent \textbf{Theorem \ref{theo:TCL}- Step size $\gamma_n =\gamma n^{-1}$}
 
This situation is more involved since we can observe that the drift of the limit diffusion is modified according to the size of $\gamma$. In particular, the generator $\mathcal{L}$  in that case is shifted from the one above by $\frac{1}{2 \gamma}  I$ so that:
$$
\mathcal{L}(\phi)(x,y)= \frac{1}{2 \gamma} \left[ \langle \nabla_x \phi,x \rangle + 
 \langle \nabla_y \phi,y \rangle \right] -\langle y, \nabla_x \phi\rangle + r \langle D^2 f(x^\star) x - y, \nabla_y \phi \rangle +r^2 \frac{ {\sigma_0}^2}{2}  \Delta_y \phi.
$$ 
Again, we can use the decomposition $D^2f(x^\star) = P^{-1} \Lambda P$ where $P$ is an orthonormal matrix, and the generator of the rotated process $(\cXn,\cYn)=(P \babar{X}_n,P\babar{Y}_n)$ is:
$$
\check{\mathcal{A}}(\phi)(x,y)=  \left\langle \frac{\check{x}}{2 \gamma} - \check{y},\nabla_{\check{x}} \phi \right\rangle + \left\langle r \Lambda \check{x} +\left(\frac{1}{2 \gamma}-r\right) \check{y},\nabla_{\check{y}} \phi \right\rangle + r^2 \frac{ {\sigma_0}^2}{2} \Delta_{\check{y}} \phi
$$
The associated Ornstein-Uhlenbeck process has a unique Gaussian invariant measure $\check{\mu}_{_{\infty}}^{(1)}$ if and only if $\gamma \alpha_{r} > 1$ where $\alpha_r$ is the constant defined in the statement of Proposition \ref{prop:vitesseexpo}. The following equations still hold:
\begin{equation}\label{eq:covij2}
\forall i \neq j \qquad \mathbb{E}_{(\check{x},\check{y}) \sim  \check{\mu}_{_{\infty}}^{(1)}} [\check{x}^{(i)} \check{x}^{(j)} ]=
\mathbb{E}_{(\check{x},\check{y}) \sim \check{\mu}_{_{\infty}}^{(1)} } [\check{x}^{(i)} \check{y}^{(j)} ] = \mathbb{E}_{(\check{x},\check{y}) \sim \check{\mu}_{_{\infty}}^{(1)}} [\check{y}^{(i)} \check{y}^{(j)} ] =0.
\end{equation}
To determine the rest of the covariance matrix, we follow the same strategy and only address the case $d=1$ for the sake of convenience.
We define:
$\sigma_{x}^2:=\EE_{(\check{x},\check{y}) \sim \check{\mu}_{_{\infty}}^{(1)} } \left[\check{x}^{2}\right] $, 
$\sigma_{y}^2:=\EE_{(\check{x},\check{y}) \sim \check{\mu}_{_{\infty}}^{(1)} } \left[\check{y}^{2}\right] $ and
$\sigma_{x,y}:=\EE_{(\check{x},\check{y}) \sim \check{\mu}_{_{\infty}}^{(1)} } \left[ \check{x}\check{y}\right] $.

We start by  chosing $\phi(\check{x},\check{y}) = \frac{\check{x}^{2}}{2}$ and obtain
$\check{\mathcal{A}}(\phi)(x,y)=  \frac{\check{x} ^{2}}{2 \gamma} -  \check{x} \check{y}$. Therefore, we deduce that:
\begin{equation}\label{eq:covxy2} 2 \gamma\,
\sigma_{x,y} = \sigma_{x}^2
\end{equation}
Now we pick $\phi(\check{x},\check{y}) = \frac{\check{y}^{2}}{2}$ and obtain
$\check{\mathcal{A}}(\phi)(x,y)=  r \lambda \check{x}\check{y} + \left(\frac{1}{2 \gamma}-r\right) \check{y}^2 + \frac{r^2 {\sigma_0}^2}{2}$ so that:
\begin{equation}\label{eq:varyy}
\left(r-\frac{1}{2 \gamma}\right)
\sigma_{y}^2 = r  \lambda \sigma_{x,y}+\frac{r^2 {\sigma_0}^2}{2}.
\end{equation}
Finally, the function  $\phi(\check{x},\check{y}) =   \check{x} \check{y} $ yields $\check{\mathcal{A}}(\phi)(x,y)= 
\check{x} \check{y}  \left(\frac{1}{ \gamma} - r \right) - \check{y}^2 + r \lambda \check{x}^2$, which implies:
\begin{equation}\label{eq:varxx}
\sigma_{y}^2 = r \lambda \sigma_{x}^2 +\left(\frac{1}{ \gamma} - r \right)\sigma_{x,y}
\end{equation}
We are led to the introduction of:
$$
\check{\alpha}_{-} =  1-\sqrt{1-\frac{4 \lambda}{r}}  \quad \text{and} \quad \check{\alpha}_{+} = 1+\sqrt{1-\frac{4 \lambda}{r}},
$$
which leads to:
$$
\sigma_{x}^2=  {\sigma_0}^2 \frac{ 2  \lambda r \gamma^3}{ (\gamma r - 1) ( 2  \lambda\gamma - \check{\alpha}_{-}) (2  \lambda\gamma - \check{\alpha}_{+})}, \quad \sigma_{y}^2 ={\sigma_0}^2 \frac{   \lambda r \gamma(2 \lambda r \gamma^2-r \gamma +1)}{ (\gamma r - 1) ( 2  \lambda\gamma - \check{\alpha}_{-}) (2  \lambda\gamma - \check{\alpha}_{+})}, 
$$
and
$$ 
\sigma_{x,y} =   {\sigma_0}^2 \frac{   \lambda r \gamma^2}{ (\gamma r - 1) ( 2  \lambda\gamma - \check{\alpha}_{-}) (2  \lambda\gamma - \check{\alpha}_{+})}.
$$
\hfill $\diamond \square$

\section{Numerical experiments}
In this short paragraph, we briefly investigate the behavior of several  algorithms, widely used in the field of stochastic approximation. In particular, we are interested in the convergence rates of each algorithm, as well as their behavior in the case of non-convex potential $f$ with multiple wells, to illustrate both Theorem \ref{theo:rates} and Theorem \ref{theo:asmin}.

\paragraph{Convergence rates}
We are first concerned by the 
typical behavior of the heavy ball stochastic approximation algorithm in the convex case. In particular, we are interested in the role played by the parameter $r$ that varies, both in the polynomial case and in the exponential case.  Figure \ref{Fig:HBF_evor} represents the logarithmic loss of the algorithms with respect to the logarithm of the number of iterations in the $1$ dimensional case with $f(x)=\frac{x^2}{2}$. The step size used is $\gamma_k=k^{-1}$.
\begin{figure}[h!]
\begin{center}
\includegraphics[scale=0.15]{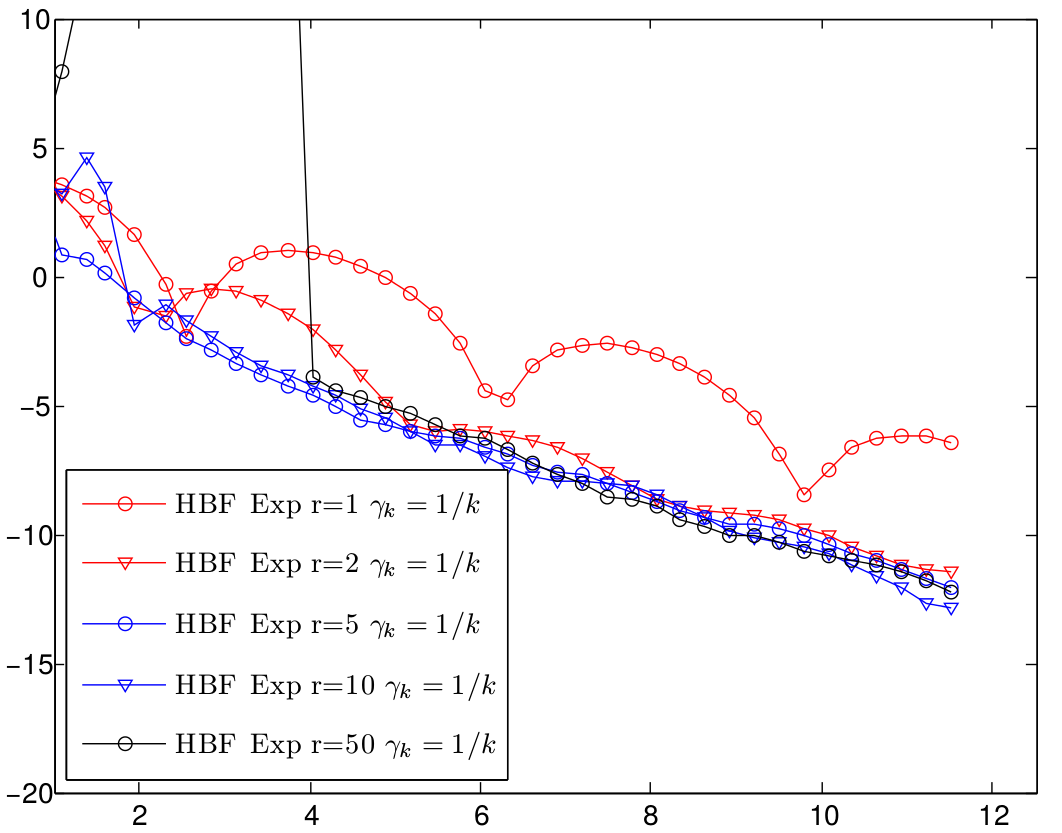}
\includegraphics[scale=0.15]{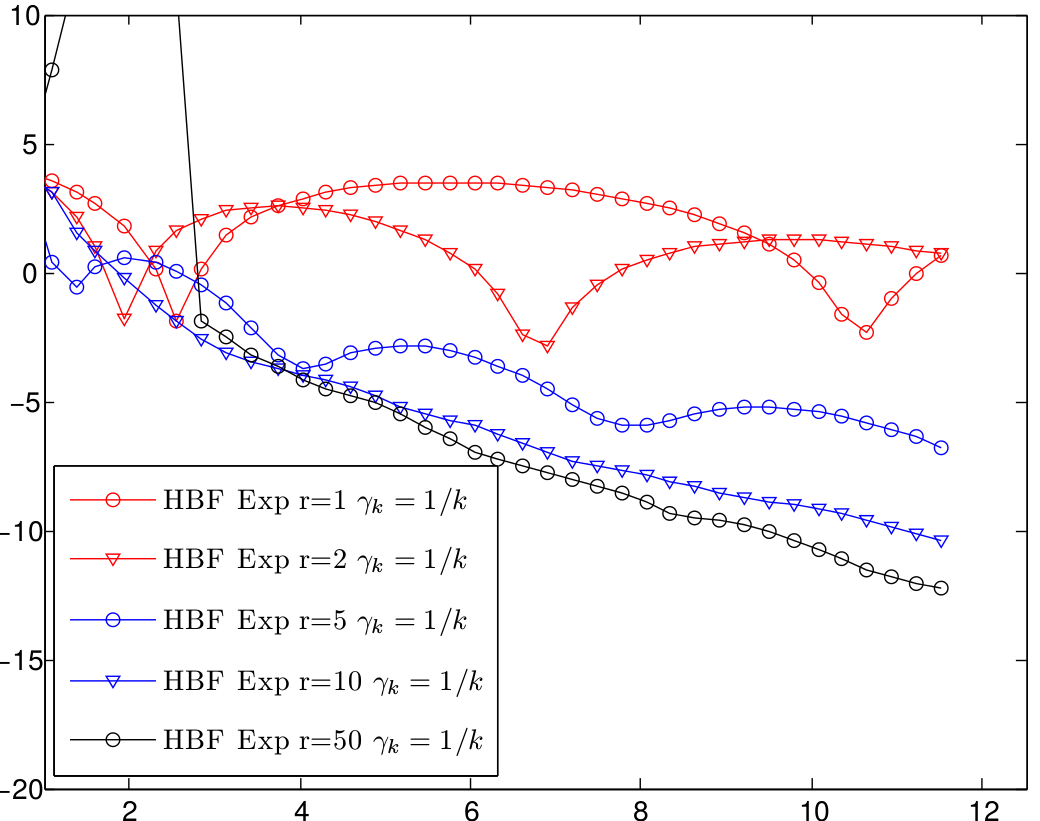}
\end{center}
\caption{Evolution of $\log(f(X_k))$ with respect to $\log(k)$. Left: Exponential memory. Right: Polynomial memory.\label{Fig:HBF_evor}}
\end{figure}
We immediately observe that for small values of $r$ (that correspond to a long-term memory case), the algorithm possesses a lengthy oscillating behavior, which is a feature of second-order algorithms with a very mild damping effect. This phenomenon has also been observed in previous works (see, \textit{e.g.}, \cite{BachFlammarion} and the references therein). We also observe that the use of an excessively large value of $r$ (say, when $r$ is greater than $10$) creates a numerical instability at the beginning of the iterations. This could be fixed by using a supplementary truncating trick introduced in \cite{Lemaire}.
Finally, the obtained rates are better (from a numerical point of view) when $r$ is chosen at around $5$ in the exponential case, and at around $10$ in the polynomial case, mainly because of the oscillations that deteriorate the convergence when $r$ is too small.

Figure \ref{Fig:rates} then compares several stochastic optimization algorithms in the toy example
  $f(x)=|x|^p/p$: the standard  Robbins-Monro stochastic gradient descent introduced in \cite{RobbinsMonro} (SGD) and several second order algorithms: the ``optimal" Ruppert-Polyak averaging algorithm (see \cite{polyakjuditsky,Ruppert}), the Nesterov accelerated gradient descent  \cite{Nesterov1} adapted in the stochastic framework in a straightforward way using an unbiased evaluation of the gradient in each iteration, and the recent SAGE method introduced in \cite{SAGE}. Note that the Rupper-Polyak averaging algorithm is used according to the recommendation of \cite{Bach_averaging} with $\gamma_k=k^{-1/2}$.

\begin{figure}[h!]
\begin{center}
\includegraphics[scale=0.15]{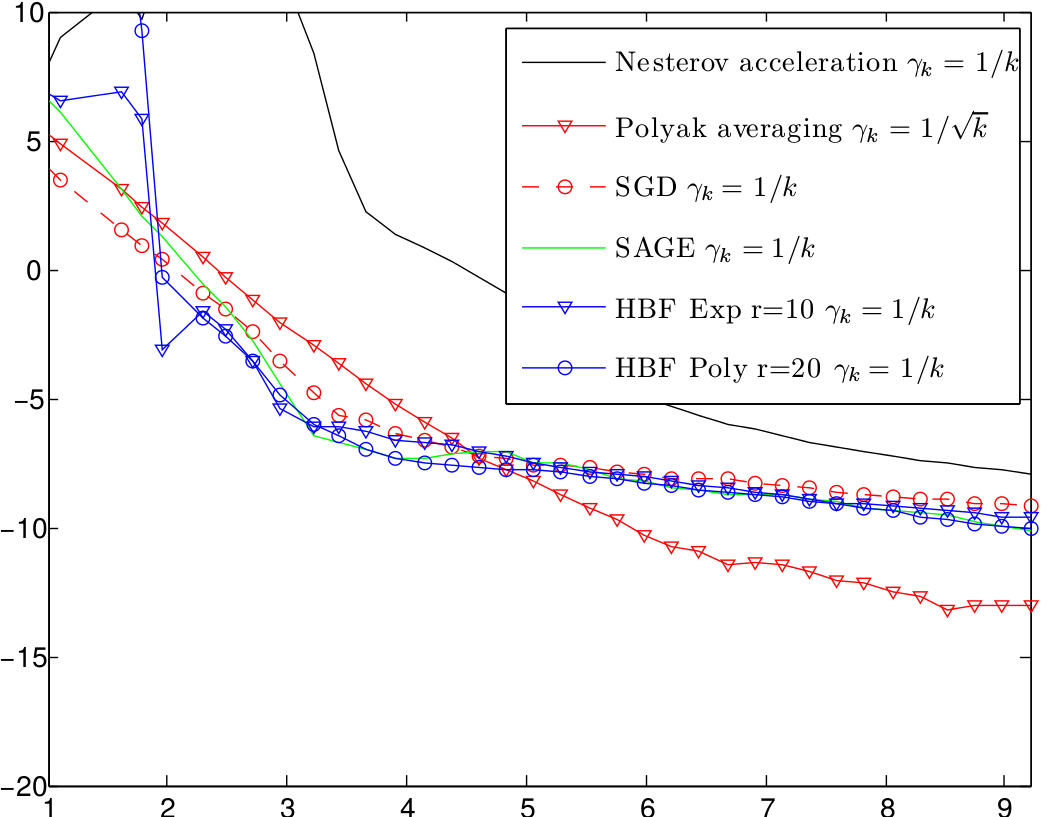}
\includegraphics[scale=0.15]{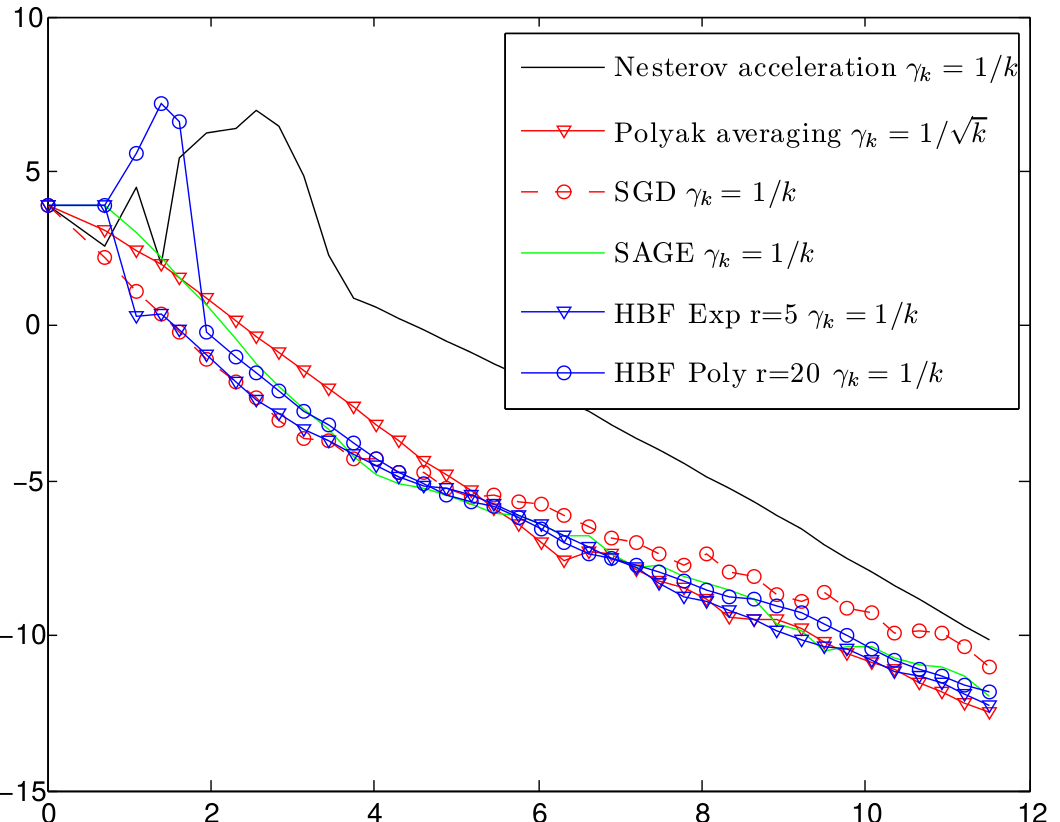}
\end{center}
\caption{Evolution of $\log(f(X_k))$ with respect to $\log(k)$ with $f(x)=|x|^p/p$. Left: Convex case $p=4$. Right: Strongly convex case $p=2$.\label{Fig:rates}}
\end{figure}
The first elementary remark is that the rate is of course deteriorated by the loss of strong convexity (left side, Figure \ref{Fig:rates}). In this case, the Ruppert-Polyak averaging outperforms other methods and attains the $O(1/\sqrt{n})$ minimax rate (see \cite{NemirovskiYudin}). When $f$ is strongly convex,  the second-order algorithms then all share an equivalent efficiency with, apparently a $\mathcal{O}(1/n)$ convergence rate. This corresponds to $(ii)$ of Theorem \ref{theo:rates} when the Hessian at the critical point is sufficiently large to make this minimax optimal rate possible. Nevertheless, the ability of the stochastic heavy ball in a more general situation may deserve further numerical investigation, which is beyond the scope of this paper.
 The SGD seems to be a little bit less effective in the strongly convex case. Finally, the Nesterov adaptation to the stochastic case does not lead to an efficient algorithm (in comparison to the other  methods tested). However, this remark should be balanced by the fact that we did not use the Lan  adaptation of the Nesterov accelerated gradient descent introduced in \cite{Lan}. It appears that this modification that consists in an addition of an intermediary point in the NAGD seems important to optimize the behavior of the algorithm in the stochastic case.

\paragraph{Non-convex case}

In this paragraph, we investigate the ability of the stochastic algorithm to avoid local traps and, in particular, we focus on the behavior of second order algorithms that may be an intermediary step towards global optimization methods such as simulated annealing.  For this purpose, we defined $f$ as:
$$
\forall x \in \RR \qquad f(x) = a x^4 + b(x-1)^2.
$$ 
with $a=1/40$ and $b=-1/5.$ These values have been fixed to guarantee the numerical stability of the stochastic procedures, but the results we obtained may be replicated for other values. The values of $a$ and $b$ above yield a double-well potential with a global minimizer of $f$ of around $x^\star\simeq -4.9$, although $f$ has a local trap on the positive part at  around $x_+ \simeq 4$. The function $f$ is represented on the top left of Figure \ref{Fig:multi}.

We used $\gamma_k=k^{-1}$ for all of the methods and we varied the initialization point of each algorithm from $-10$ to $10$ with $100$ Monte-Carlo replications. For each simulation, we arbitrarily stopped the evolution of the algorithm after $T=10^4$ iterations, and considered that optimization was successful when $|x_T-x^\star| \leq 1$. This criterion may be replaced by a more stringent inequality, at the price of an increase of $T$, without really changing the main conclusions below.

\begin{figure}[h!]
\begin{center}
\includegraphics[scale=0.2]{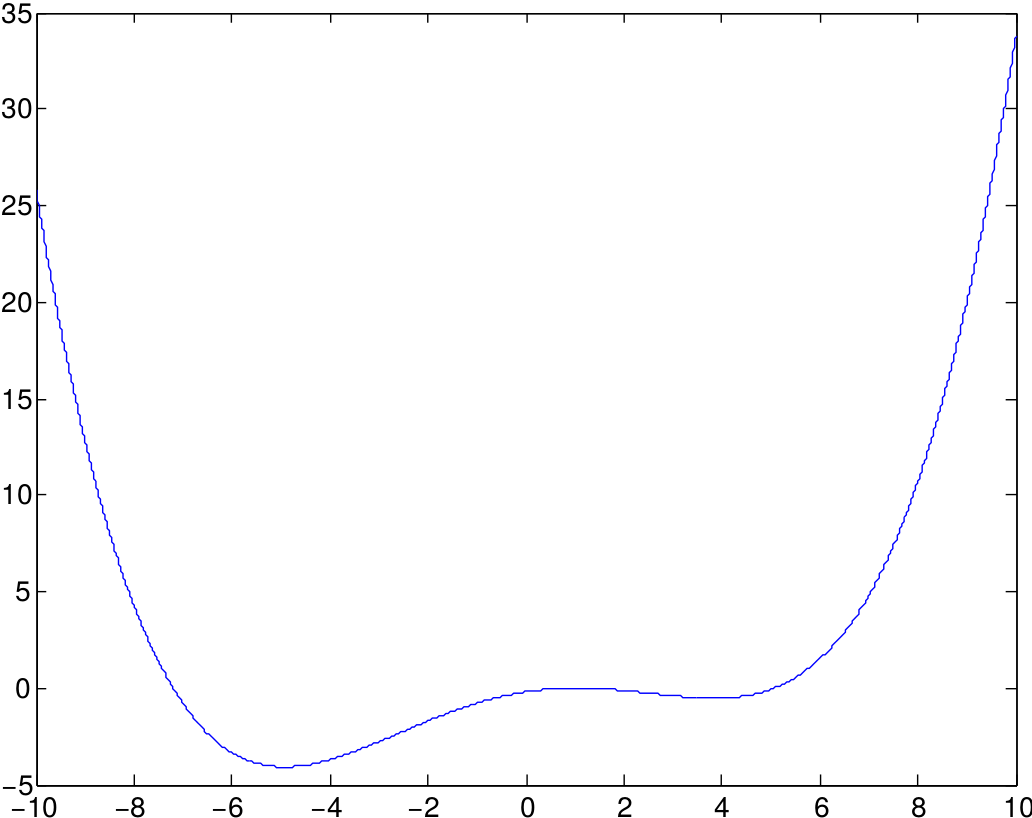}
\includegraphics[scale=0.2]{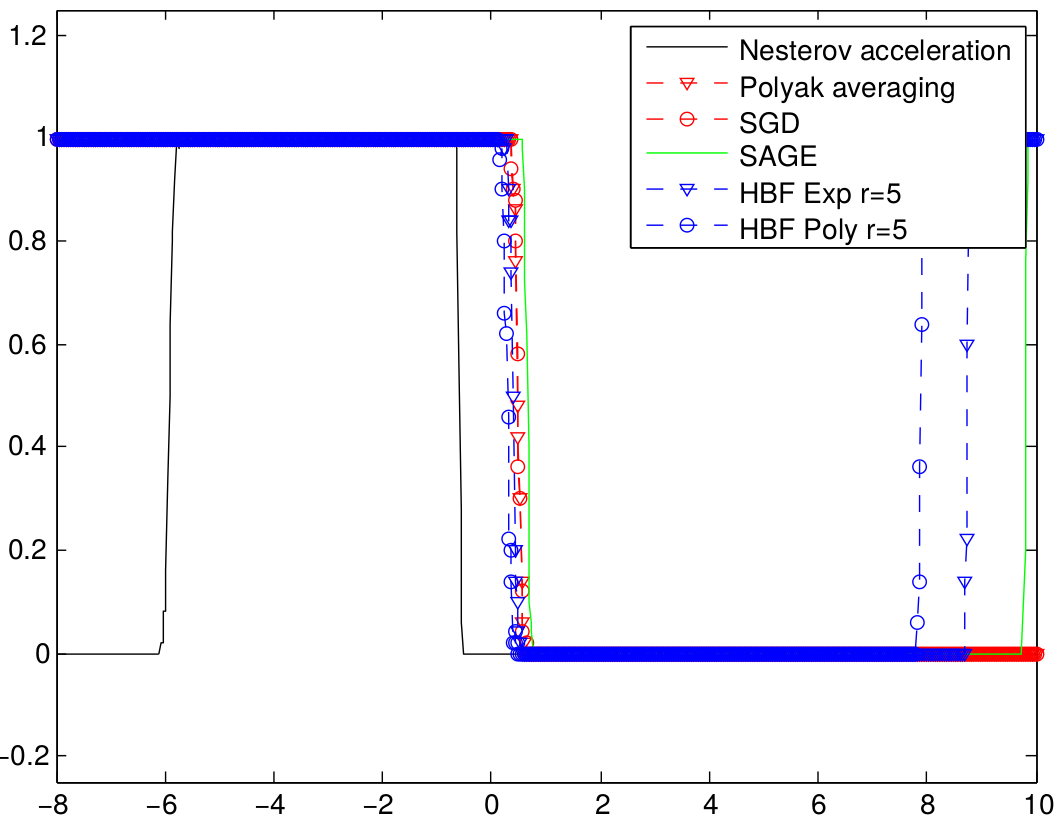}
\includegraphics[scale=0.2]{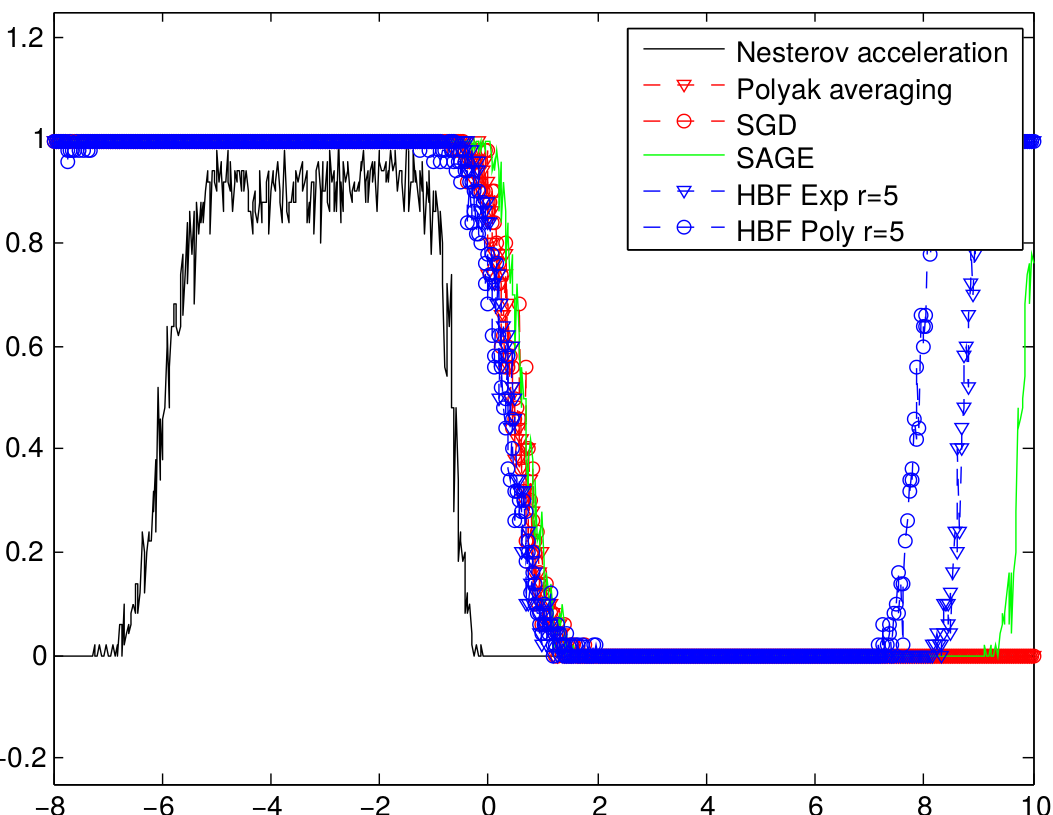}
\includegraphics[scale=0.2]{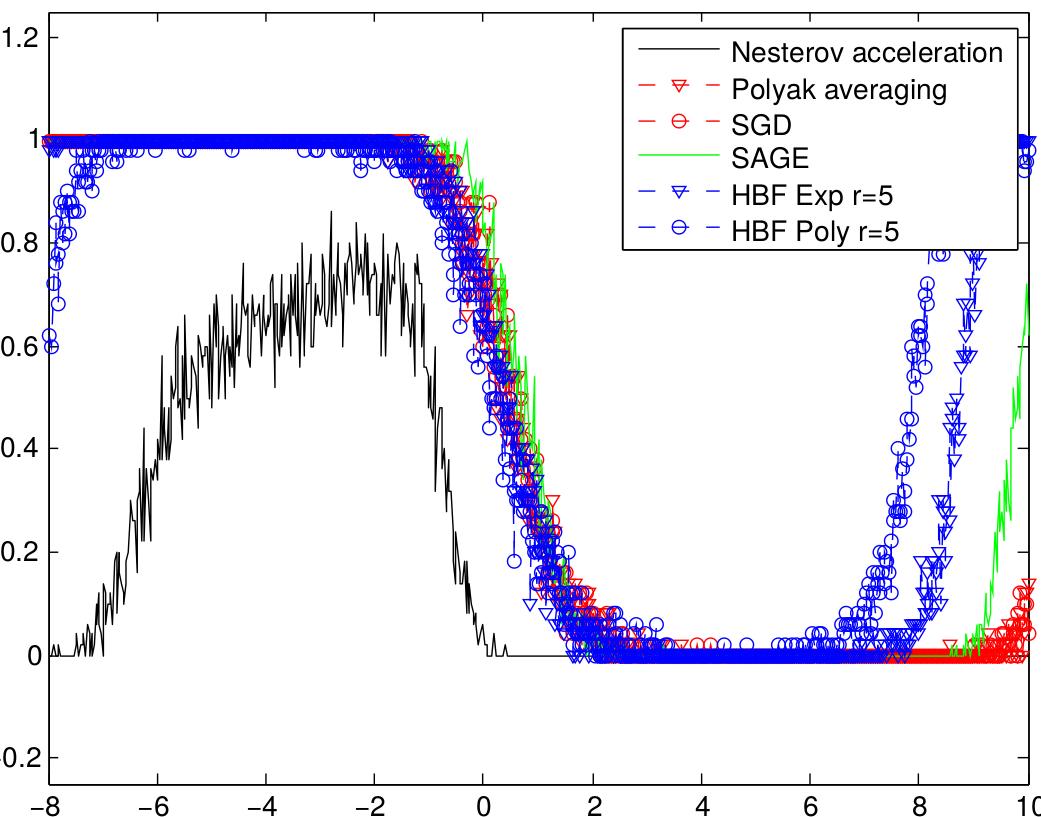}
\end{center}
\caption{Top left: function $f$ to be minimized. Top right: probability of success of the stochastic algorithms with respect to the initialization point with small variance: $\sigma=0.1$.
Bottom left: $\sigma=1$. Bottom right: $\sigma=2$.\label{Fig:multi}}
\end{figure}

Performances are reported in Figure \ref{Fig:multi}. We observe that both SGD and Ruppert-Polyak algorithms have the same behavior. This fact is absolutely clear because Polyak averaging is built with a Cesaro average of SGD. The target convergence point of SGD and of Polyak averaging are thus the same. We can also note that in the almost no noise setting, the basin of attraction of $x^\star$ for SGD may be  roughly approximated by $]-\infty,1]$. Nevertheless, both SAGE and HBF seem to behave better behaviour with a somewhat larger basin of attraction: in particular, it is possible to start from an initialization point $x_1=8$ and still obtain convergence of SAGE or HBF towards $x^\star$. This last point is clearly impossible with SGD. The same conclusions hold for different values of $\sigma$ (see Figure \ref{Fig:multi}, bottom left and right).
Finally, we observe that NAGD does not present  very good behavior: the probability of failure when the algorithm is initialized at $-4$ is lower than $1$ for $\sigma=1$ or $\sigma=2$.

We can calculate a more quantitative indicator of this behavior with the computation of the average rate of success of each algorithm when the initialization point is sampled uniformly over $[-10;10]$.
\begin{table}
\begin{center}
\begin{tabular}{|c|c|c|c|c|c|c|}
$\sigma$ & SGD & AV SGD & SAGE&  NAGD & HBF Poly r=5 & HBF Expo r=5 \\
0.1 &0.47 & 0.47 & 0.49 & 0.29 & 0.58 & 0.52\\
1   & 0.47 & 0.47 & 0.49 & 0.27 & 0.58 & 0.55\\
2   & 0.47 & 0.47 & 0.49 & 0.20 & 0.58 & 0.54 \\
\end{tabular}
\end{center}
\caption{Average rate of success of each stochastic algorithm with a uniformly sampled initialization over $[-10;10]$ when $\sigma$ varies.\label{table:rate}}
\end{table}
Table \ref{table:rate} seems to indicate that the stochastic heavy ball leads to a better exploration of the state space, in particular, with reasonable values of $r$ (see Table \ref{table:rate2}). These conclusions should be understood as numerical observations of experimental results on this particular type of synthetic case, but we do not have any theoretical arguments to strengthen these final observations at this time.

\begin{table}
\begin{center}
\begin{tabular}{|c|c|c|c|c|c|c|c|}
 Exp 1  & Exp 2 & Exp 5&  Exp 10 &Poly 1 & Poly 2 & Poly 5 & Poly 10\\
0.51 & 0.53 & 0.55 & 0.58 & 0.26 & 0.43 & 0.58 & 0.50\\
\end{tabular}
\end{center}
\caption{Average rate of success of  heavy ball stochastic algorithm for several values of $r$, when $\sigma=1$ and the initialization point is sampled  uniformly over $[-10;10]$.\label{table:rate2}}
\end{table}

\appendix
\section{Technical results}\label{sec:appendix}
\subsection{Standard tools of stochastic algorithms}\label{subsec:Rob-Sieg}
We recall below a standard version of the so-called Robbins-Siegmund Theorem (see $e.g.$ \cite{Duflo}): 
\begin{theo}\label{theo:Robbins_Siegmund} Given a filtration $\mathcal{F}_n$ and four positive, integrable and $\mathcal{F}_n$-adapted sequences $(\alpha_n)_n$,$(\beta_n)_n$, $(U_n)_n$ and $(V_n)_n$ satisfying:
\begin{itemize}
\item (i) $(\alpha_n)_n$,$(\beta_n)_n$, $(U_n)_n$ are predictible sequences.\\
\item (ii) $\sup_{\omega}\prod\limits_n(1+\alpha_n(\omega))<\infty$,  $\sum\limits_n\mathbb{E}(\beta_n)<\infty$.\\
\item(iii) $\forall n\in \mathbb{N}$, $$\mathbb{E}(V_{n+1}\vert\mathcal{F}_n)\leq V_n(1+\alpha_{n+1})+\beta_{n+1}-U_{n+1}$$
\end{itemize}
Then:
\begin{itemize}
\item[(i)] $V_n$ converges to $V_\infty$ in $L^1$ and $\sup_n \ES[{V}_n ]<\infty.$\\
\item[(ii)] $\sum\limits_n\mathbb{E}(U_n)<\infty$, $\sum\limits_n U_n<\infty$ a.s.
\end{itemize}

\end{theo}

\subsubsection{Step sizes $\gamma_n= \gamma\, n^{-\beta}$ with $\beta < 1$}

\begin{pro}\label{prop:controle_algo}
For any positive values $a>0$ and $b>0$, for any $\beta \in (0,1)$ and any sequence $(\gamma_n)_{n \geq 1}$ defined by
$\gamma_n = \gamma n^{-\beta}$, one has:
\begin{itemize}
\item[$(i)-a$] If $\beta < 1/2$, then $\sum\limits_{k=1}^n a \gamma_{k} - b \gamma_{k}^2 \geq \frac{a \gamma}{1-\beta} n^{1-\beta} - \frac{b \gamma^2}{1-2 \beta} n^{1-2 \beta}$\\
\item[$(i)-b$] If $\beta > 1/2$, then $\sum\limits_{k=1}^n a \gamma_{k} - b \gamma_{k}^2 \geq \frac{a \gamma}{1-\beta} n^{1-\beta} - \frac{b \gamma^2}{2 \beta-1}$\\
\item[$(i)-c$] If $\beta = 1/2$, then $\sum\limits_{k=1}^n a \gamma_{k} - b \gamma_{k}^2 \geq \frac{a \gamma}{1-\beta} n^{1-\beta} - b \gamma^2 \log n$\\

\item[$(ii)$]  An integer $n_0$ exists such that $\forall n \geq n_0 \quad \sum\limits_{k=1}^n\gamma_{k}^2\prod\limits_{l=k+1}^n(1- a\gamma_{l})^2 \leq \frac{2}{a}\gamma_{n+1}$\\
\item[$(iii)$] An integer $n_0$ exists such that $\forall n \geq n_0 \quad \sum\limits_{k=1}^n\gamma_{k}^2\prod\limits_{l=k+1}^n(1- a\gamma_{l}+b \gamma^2_l)\leq \frac{2}{a}\gamma_{n+1}$\\
\end{itemize}
\end{pro}
\noindent
\textit{Proof:} The upper bounds involved in $(i)-a$, $(i)-b$ and $(i)-c$ are straightforward. \hfill$\diamond$

\noindent  \underline{Proof of $(ii)$:} Using $\Gamma_n$ introduced in the beginning of Section \ref{sec:HBF}, we write:
\begin{eqnarray*}
\sum\limits_{k=1}^n\gamma_{k}^2\prod\limits_{l=k+1}^n(1- a\gamma_{l})^2& \leq & \sum\limits_{k=1}^n\gamma_{k}^2 e^{- a \sum_{k+1}^n \gamma_l}\\
&= & 
\sum\limits_{k=1}^n\gamma_{k}^2 e^{- a \Gamma_n+a \Gamma_k } \leq \gamma^2 e^{- a \Gamma_n }\sum\limits_{k=1}^n  k^{-2 \beta} e^{\frac{a \gamma}{1-\beta} k^{1-\beta} }
\end{eqnarray*}

The function $x\longmapsto x^{-2 \beta} e^{\frac{a \gamma}{1-\beta} x^{1-\beta}}$ being increasing for $x \geq c_{a,\gamma,\beta}$, we then obtain, considering an integer $t>c_{a,\gamma,\beta}$:
$$
\sum\limits_{k=1}^n\gamma_{k}^2\prod\limits_{l=k+1}^n(1- a\gamma_{l})^2 \leq 	\gamma^2 e^{- a \Gamma_n } \left( C_t + \int_{t}^n 
x^{-2 \beta} e^{\frac{a \gamma}{1-\beta} x^{1-\beta}} dx\right).
$$
We can write $x^{- 2 \beta} e^{K x^{1-\beta}} = \left( e^{K x^{1-\beta}}\right)' x^{-\beta} K^{-1} (1-\beta)^{-1}$ and integrating by parts, we obtain for a large enough $n$:
$$
\sum\limits_{k=1}^n\gamma_{k}^2\prod\limits_{l=k+1}^n(1- a\gamma_{l})^2 \leq   
\gamma^2 e^{- a \Gamma_n } \left(C_t + \frac{e^{a \Gamma_n} }{a \gamma} n^{-\beta} \right) \leq  \frac{2}{a} \gamma_n. 
$$
\hfill $\diamond$

\noindent  \underline{Proof of $(iii)$:}
We only deal with $\beta<1/2$, which is the most involved situation. Using $\Gamma_n$ and $\Gamma_n^{(2)}$ introduced in the beginning of Section \ref{sec:HBF}, we write:
\begin{eqnarray*}
\sum\limits_{k=1}^n\gamma_{k}^2\prod\limits_{l=k+1}^n(1- a\gamma_{l}+b \gamma^2_l)& \leq &\sum\limits_{k=1}^n\gamma_{k}^2 e^{- a \Gamma_n+a \Gamma_k + b \Gamma_n^{(2)} - b \Gamma_k^{(2)}}\\
& \leq & e^{- a \Gamma_n + b \Gamma_n^{(2)} } \sum\limits_{k=1}^n\gamma_{k}^2 e^{a \Gamma_k - b \Gamma_k^{(2)}} \\
& \leq &\gamma^2 e^{- a \Gamma_n + b \Gamma_n^{(2)}}\sum\limits_{k=1}^n  k^{-2 \beta} e^{\frac{a \gamma}{1-\beta} k^{1-\beta} - \frac{b \gamma^2}{1-2\beta} k^{1-2 \beta}}
\end{eqnarray*}
The function $x\longmapsto x^{-2 \beta} e^{\frac{a \gamma}{1-\beta} x^{1-\beta} - \frac{b \gamma^2}{1-2\beta} x^{1-2 \beta}}$ being increasing for $x \geq c_{a,b,\gamma,\beta}$, we then obtain considering an integer $t>c_{a,b,\gamma,\beta}$:
\begin{eqnarray*}
\sum\limits_{k=1}^n\gamma_{k}^2\prod\limits_{l=k+1}^n(1- a\gamma_{l}+b \gamma^2_l) &\leq &
\gamma^2 e^{- a \Gamma_n + b \Gamma_n^{(2)}}\left( \sum_{k=1}^{t} k^{-2 \beta} e^{\frac{a \gamma}{1-\beta} k^{1-\beta} - \frac{b \gamma^2}{1-2\beta} k^{1-2 \beta}} + \int_{t}^n x^{-2 \beta} e^{\frac{a \gamma}{1-\beta} x^{1-\beta}- \frac{b \gamma^2}{1-2\beta} x^{1-2 \beta}} dx \right) \\
& \leq &\gamma^2 e^{- a \Gamma_n + b \Gamma_n^{(2)}}\left( C_{t}    +
\int_{t}^n x^{- 2\beta} e^{\frac{a \gamma}{1-\beta} x^{1-\beta} - \frac{b \gamma^2}{1-2\beta} x^{1-2 \beta}} dx \right)\\
& \leq & \gamma^2 e^{- a \Gamma_n + b \Gamma_n^{(2)}}\huge\left( C_{t} \right.\\
& 
+& \left.
\int_{t}^n x^{- \beta} \left[ \frac{3}{2} \frac{a \gamma x^{-\beta} - b \gamma^2 x^{-2 \beta}}{a \gamma}+ \frac{3 b \gamma x^{-2 \beta} - a \gamma x^{-\beta}}{2}\right] e^{\frac{a \gamma}{1-\beta} x^{1-\beta} - \frac{b \gamma^2}{1-2\beta} x^{1-2 \beta}} dx \right)\\
\end{eqnarray*}
Now choosing  $t \geq (3 b /a)^{\beta^{-1}}$ yields $3 b \gamma x^{-2 \beta} \leq a \gamma x^{-\beta}$ for any $x \geq t$. Integrating by parts, we obtain:
\begin{eqnarray*}
\sum\limits_{k=1}^n\gamma_{k}^2\prod\limits_{l=k+1}^n(1- a\gamma_{l}+b \gamma^2_l) &\leq &
\gamma^2 e^{- a \Gamma_n + b \Gamma_n^{(2)}}\left( C_{t} + \frac{ n^{-\beta}}{a\gamma} e^{-a \Gamma_n+n \Gamma_n^{(2)}} \right) \\
& \leq & \frac{ \gamma  n^{-\beta}}{a} + \gamma^2 C_t e^{- a \Gamma_n + b \Gamma_n^{(2)}}.
\end{eqnarray*}
Then, choosing $n_0$ large enough (that depends on $a,b,\gamma$ and $\beta$), we deduce that:
$$
\forall n \geq n_0 \qquad \sum\limits_{k=1}^n\gamma_{k}^2\prod\limits_{l=k+1}^n(1- a\gamma_{l}+b \gamma^2_l) \leq \frac{2}{a} \gamma_n.
$$
\hfill $\diamond$

\hfill $\square$
\subsubsection{Step sizes $\gamma_n= \gamma\, n^{-1}$}

\begin{pro}\label{prop:controle_algo2}
For any positive values $a>0$ and $b>0$ and any sequence $(\gamma_n)_{n \geq 1}$ defined by
$\gamma_n = \gamma n^{-1}$, we have:
\begin{itemize}
\item[$(i)$] $\sum\limits_{k=1}^n a \gamma_{k} - b \gamma_{k}^2 \geq a \log n - b \pi^2/6$\\
\item[$(ii)$] $\sum\limits_{k=1}^n\gamma_{k}^2\prod\limits_{l=k+1}^n(1- a\gamma_{l})^2\leq 
C_{\gamma}
\displaystyle\begin{cases}
\frac{1}{a \gamma - 1} n^{-1} \,\,\, \qquad \text{if} \qquad a \gamma > 1\\
\log n\,n^{-1} \,\,\qquad \text{if} \qquad a \gamma = 1\\
\frac{1}{1-a \gamma} n^{-a \gamma}  \qquad \text{if} \qquad a \gamma < 1
\end{cases}$\\

\item[$(iii)$] $\sum\limits_{k=1}^n\gamma_{k}^2\prod\limits_{l=k+1}^n(1- a\gamma_{l}+b \gamma^2_l)\leq 
C_{\gamma,b}
\displaystyle\begin{cases}
\frac{1}{a \gamma - 1} n^{-1} \,\,\, \qquad \text{if} \qquad a \gamma > 1\\
\log n\,n^{-1} \,\,\qquad \text{if} \qquad a \gamma = 1\\
\frac{1}{1-a \gamma} n^{-a \gamma}  \qquad \text{if} \qquad a \gamma < 1
\end{cases}$\\
 \item[$(iv)$] For any $\epsilon>0$, $a>0$ and $b>0$: $\sum\limits_{k=1}^n\gamma_{k+1}\prod\limits_{l=k+1}^n(1- a\gamma_{l}+b\gamma_l^{1+\epsilon})\leq \frac{2e^{b \Gamma_{\infty}^{(1+\epsilon)}}}{a}$.
\end{itemize}
\end{pro}

\noindent
\textit{Proof:} The upper bounds involved in $(i)$ and $(ii)$ are straightforward. \hfill$\diamond$

\noindent \underline{Proof of $(iii)$:} The situation is easier than the one involved in point $(ii)$ of Proposition \ref{prop:controle_algo} because in that case, we have:  
$$ \forall n \geq 1 \qquad \Gamma_n^{(2)} \leq \gamma^2 \pi^2/6.$$
Therefore, we can repeat the computations above and get:
\begin{eqnarray*}
\sum\limits_{k=1}^n\gamma_{k}^2\prod\limits_{l=k+1}^n(1- a\gamma_{l}+b \gamma^2_l)&\leq&  \sum\limits_{k=1}^n\gamma_{k}^2 e^{- a \Gamma_n+a \Gamma_k + b \Gamma_n^{(2)} - b \Gamma_k^{(2)}}\\
& \leq & e^{- a \Gamma_n + b \gamma^2 \pi^2/6 } \sum\limits_{k=1}^n\gamma_{k}^2 e^{a \Gamma_k } \\
& \leq &\gamma^2 e^{ b \gamma^2 \pi^2/6} n^{- a \gamma} \sum\limits_{k=1}^n  k^{-2+a \gamma}.
\end{eqnarray*}
We then deduce that: 
$$
\sum\limits_{k=1}^n\gamma_{k}^2\prod\limits_{l=k+1}^n(1- a\gamma_{l}+b \gamma^2_l)=
\gamma^2 e^{ b \gamma^2 \pi^2/6}
\displaystyle\begin{cases}
\frac{1}{a \gamma - 1} n^{-1} \,\,\, \qquad \text{if} \qquad a \gamma > 1\\
\log n\,n^{-1} \,\,\qquad \text{if} \qquad a \gamma = 1\\
\frac{1}{1-a \gamma} n^{-a \gamma}  \qquad \text{if} \qquad a \gamma < 1
\end{cases}
$$
\hfill $\diamond$

\noindent \underline{Proof of $(iii)$:} 
We follow the same guideline: remark that $(\Gamma_n^{(1+\epsilon)})_{n \geq 1}$ is a bounded sequence and write
\begin{eqnarray*}
\sum\limits_{k=1}^n\gamma_{k+1}\prod\limits_{l=k+1}^n(1- a\gamma_{l}+b\gamma_l^{1+\epsilon})&\leq&
\sum\limits_{k=1}^n\frac{\gamma}{k+1} e^{-a \gamma \log n + a \gamma \log k + b \Gamma_n^{(1+\epsilon)}} \\
& \leq & \gamma e^{b \Gamma_{\infty}^{(1+\epsilon)}} n^{-a \gamma} \int_{1}^{n} x^{a \gamma-1} dx\\
& \leq & \frac{e^{b \Gamma_{\infty}^{(1+\epsilon)}}}{a}.
\end{eqnarray*}
\hfill $\square$
\subsection{Expectation of the supremum of the square of sub-Gaussian random variables}

We consider a sequence of independent random variables $(\xi_i)_{i \geq n}$ of $\mathbb{R}^d$ such that each coordinate satisfies a sub-Gaussian assumption $\Hsub$:
\begin{equation}\label{eq:subgaussian}
\forall \lambda \in \mathbb{R} \qquad 
\forall j \in \{1,\ldots,d\} \qquad \forall i \geq n \qquad 
\log \mathbb{E} \left[ e^{ \lambda \xi_i^j} \right] \leq \lambda^2 \frac{ \sigma^2}{2},
\end{equation}
where $\sigma^2$ is a variance factor. If  $(\gamma_k)_{k \geq n}$ is a decreasing sequence in $\ell^{2}(\mathbb{N})$, we are looking for an upper bound of:
\begin{equation}\label{eq:defsup}
m^{\star}_n =  \mathbb{E} \left[  \sup_{k \geq n} \left\{ \gamma_k^2 \|\xi_k\|^2 \right\} \right].
\end{equation}
For any $\nu>0$ and any decreasing sequence $\gamma_n \sim \gamma n^{-\nu}$, we  establish the following result (useful for Theorem \ref{theo:aspoly}).

\begin{theo}\label{theo:sup_sub}
If each coordinate $\xi^j_i$ is absolutely continuous w.r.t. the Lebesgue measure and satisfies $\Hsub$, then:
$$
m_n^{\star} \lesssim \sigma^2 d \, \gamma_n^2 \log(\gamma_n^{-2}),
$$
where $\lesssim$ refers to an inequality up to a universal constant.
\end{theo}
We begin with a preliminary lemma.

\begin{lem}\label{lemma:couplage}
Assume that $X$ is a real random variable that satisfies $\Hsub$ with median $0$:
$$
\mathbb{P} \left(X>0\right) = \mathbb{P} \left(X<0\right) = \frac{1}{2}.
$$
Then, we can find $Y \sim \mathcal{N}(0, \sigma^2)$ on the same probability space and $c$ large enough s.t.
$$
|X| \leq c |Y| \qquad a.s.
$$
\end{lem}
\noindent
\textit{Proof:}

\noindent
We use a coupling argument. 
We denote $F_{X}$ as the cumulative distribution function:
$$
F_X(t) = \int_{-\infty}^t f_X(u) du= \mathbb{P}[X \leq t].
$$
Similarly, we also denote $\Psi_{\sigma^2}$ as the cumulative distribution function of a Gaussian random variable $\mathcal{N}(0,\sigma^2)$:
$$
\Psi_{\sigma^2}(t) = \int_{-\infty}^t \frac{e^{-x^2/2 \sigma^2}}{\sqrt{2 \pi} \sigma} dx= \mathbb{P}[ \mathcal{N}(0,\sigma^2) \leq t].
$$
Our assumption on the distribution on $X$ shows that the generalized inverse of $F_X$ (denoted $F_X^{-1}$) exists and 
if $\mathcal{U}$ is a uniform random variable between on $[0,1]$, then $X \sim F_X^{-1}(\mathcal{U})$.
We now consider the random variable $Y \sim  F_{\sigma^2}^{-1}(\mathcal{U})$ built with the same realization of $\mathcal{U}$.
Of course, $Y$ is distributed according to a Gaussian random variable $\mathcal{N}(0,\sigma^2)$.

We need to show that a sufficiently large $c>0$ exists such that $|X| \leq c |Y|$, that is:
\begin{equation}\label{eq:couple}
\left| F_X^{-1}(u) \right| \leq c \left| \Psi_{\sigma^2}^{-1}(u) \right|.
\end{equation}
Using the fact that $F_X$ is an increasing function, and letting $u = \Psi_{\sigma^2}(y)$, it is then equivalent to show that:
\begin{equation}\label{eq:couple2}
\forall y \in \mathbb{R} \qquad  F_X(- c |y|) \leq \Psi_{\sigma^2}(|y|) \leq F_X( c |y|) 
\end{equation}
We now study two different situations for $y$.
If $y = 0$, then Inequality \eqref{eq:couple2} holds since   the median of $X$ is $0$. If $|y| \leq \eta$ is close to $0$, the same inequality is satisfied with a first-order Taylor expansion. For example, the right hand side reads:
$$
F_X(c|y|) \sim \frac{1}{2} + \int_{0}^{c|y|} f_X(u) du \geq \frac{1}{2} + c f_X(0) |y| + o(|y|),
$$
which is greater than $\Psi_{\sigma^2}(|y|)$ for $c$ large enough. Hence, we deduce that Inequality \eqref{eq:couple2} holds around $0$.

Now, we assume that $|y| > \eta>0$, the desired upper bound \eqref{eq:couple2} is equivalent to:
$$
1-F_X(c|y|) \leq 1-\Psi_{\sigma^2}(|y|).
$$
The Chernoff bound associated with the sub-Gaussian assumption $\Hsub$ on the distribution of $X$ implies that:
$$
\mathbb{P}(X>c|y|) \leq \displaystyle e^{\inf_{\lambda >0} \left\{ \lambda^2 \sigma^2/2 - \lambda c |y|\right\}} = e^{- \frac{c^2 |y|^2}{2 \sigma^2}}.
$$
At the same time, the lower bound of the Gaussian tail is given by:
$$
1-\Psi_{\sigma^2}(c|y|) \geq \frac{e^{-|y|^2/2 \sigma^2}}{\sqrt{2 \pi} \sigma} \left[ |y|^{-1}-|y|^{-3}\right] \geq \kappa(\delta) e^{-|y|^2/2 \sigma^2},
$$
with $\kappa(\delta)$ a constant independent of $|y|\geq \delta$. Hence, the right hand side of \eqref{eq:couple2} holds for a large enough $c$ (independent on $\sigma^2$). A symmetry argument permits to conclude for the left hand side of \eqref{eq:couple2}.

\noindent
Inequality \eqref{eq:couple2} being  equivalent to \eqref{eq:couple}, the conclusion of the proof follows.
\hfill$\square$

\medskip
\noindent
We are now looking at to the proof of Theorem \ref{theo:sup_sub}.

\noindent
\textit{Proof of Theorem \ref{theo:sup_sub}:}

\noindent
 We will shift all of the coordinates  of the random variables $(\xi_{i})_{i \geq n}$ by their corresponding medians. Assuming $\Hsub$, the coordinates $(\xi_i^j)_{1 \leq j \leq d}$ are centered and have a second-order moment upper bounded by $\sigma^2$ (see \cite{Stromberg}, for example):
$$
\forall i \geq n \qquad \forall j \in \{1,\ldots,d\} \qquad \mathbb{E}[\{\xi_j^i\}^2] \leq \sigma^2.
$$
The Tchebychev inequality implies that each median $m_i^j$ of the random variables $\xi_i^j$ are bounded by:
\begin{equation}\label{eq:bound_median}
\forall i \geq n \qquad \forall j \in \{1,\ldots,d\} \qquad |m_i^j| \leq \sqrt{2}  \sigma.
\end{equation}
We then consider the centered (w.r.t. their medians) random variables:
$$
\tilde{\xi}_i^j = \xi_i^j - m_i^j,
$$
and use  the inequality $(a+b)^2 \leq 2 a^2 + 2 b^2$ together with the upper bound \eqref{eq:bound_median} to deduce that:
\begin{eqnarray*}
m_n^{\star} &=& \mathbb{E} \sup_{k \geq n} \gamma_k^2 \|\xi_k\|^2 = \mathbb{E} \sup_{k \geq n} \gamma_k^2 \sum_{j=1}^d \{\xi^j_k \}^2 \\
& \leq& 
\mathbb{E} \sup_{k \geq n} \gamma_k^2  \left[  2 \sum_{j=1}^d \{\xi^j_k - m_j^k \}^2 + 2 d  \sigma^2 \right] \\
& \leq & 2 d \sigma^2 \gamma_n^{2} + 2 \mathbb{E}  \sup_{k \geq n} \gamma_k^2 \|\tilde{\epsilon}_k\|^2.
\end{eqnarray*}
We can use Lemma \ref{lemma:couplage} and deduce that up to a multiplicative universal constant:
$$m_n^{\star} \lesssim 2d \sigma^2 \gamma_n^2 + 2 \sigma^2 \mathbb{E} \sup_{k \geq n} \gamma_k^2 \|Z_k\|^2,$$ where each $(Z_k)_{k \geq n}$ are  i.i.d. realizations of  Gaussian random variables $\mathcal{N}(0,\sigma^2 I_d)$.

We now aim to apply a chaining argument to control the supremum of the empirical process above. To apply Lemma 
\ref{lemma:chain}, we define $\mathcal{T}_n := \llbracket n ; + \infty  \llbracket$ and compute the Laplace transform of the chi-square-like random variables:
$$
\log \mathbb{E} e^{\lambda [\gamma_k^2 \|Z_k\|^2 -\gamma_j^2 \|Z_j\|^2]} =\frac{d}{2} \log \left( \frac{1-2 \lambda \gamma_j^2}{1-2 \lambda \gamma_k^2} \right)
$$
We can check that up to a universal multiplicative constant, we have:
$$
\forall \lambda \in \mathbb{R}_+ \quad \forall (a,b) \in \mathbb{R}_+ \times \mathbb{R}_+: \qquad 
\log \frac{1-a \lambda}{1-b \lambda} \lesssim \lambda |a-b| + \frac{|a-b|^2 \lambda^2}{1-\lambda |a-b|}.
$$
We are naturally driven to define the pseudo-metric on $\mathcal{T}_n$ by:
$$
\forall (i,j) \in \mathcal{T}_n^2 \qquad d(i,j) = \left| \gamma_i^2 - \gamma_j^2 \right|.
$$
It remains to upper bound the covering number of $\mathcal{T}_n$ according to $d$ for any radius $\epsilon>0$. Indeed, when $2 \gamma_n^2 \leq \epsilon$, we have $N(\epsilon, \mathcal{T}_n)=1$ although when $\epsilon \leq 2 \gamma_n^2$, we use the rough bound:
$$
N(\epsilon,\mathcal{T}_n) \leq \inf \left\{j \geq n \, : 2 \gamma_j^2 \leq \epsilon \right\}.
$$
In particular, if $\gamma_j = \gamma j^{-\nu}$, we then obtain
$$
N(\epsilon,\mathcal{T}_n) \sim \epsilon^{-1/2 \nu}.
$$
We apply Lemma \ref{lemma:chain} and obtain an upper bound for the right hand side of \eqref{eq:chain_bound}. The first term is proportionnal to $\gamma_n^2$. The  other terms lead to the computation of the two integrals (up to some universal multiplicative constants):
$$
\int_{0}^{\gamma_n^2} \sqrt{\log(\epsilon^{-1})} d\epsilon  \qquad \text{and} \qquad  \int_{0}^{\gamma_n^2}  \log(\epsilon^{-1}) d\epsilon 
$$
 The change of variable $\epsilon =e^{-x}$ and an integration by parts leads to  an upper bound whose size is $\log(\gamma_n^{-2}) \gamma_n^2$.
\hfill $\square$
\medskip
 The next Lemma, borrowed from \cite{BLM} (see Lemma 13.1, Chapter 13), provides a key estimate for the expectation of the suppremum of an empirical process indexed by a pseudo metric space $(\mathcal{T},d)$. This estimate involves the covering numbers $N(\delta,\mathcal{T})$ associated with the set $\mathcal{T}$ and the pseudo-metric $d$.

\begin{lem}\label{lemma:chain}
Let $\mathcal{T}$ be a separable metric space and $(X_t)_{t \in \mathcal{T}}$  be a collection of random variables such that for some constants $a,v,c>0$,
$$
\log \mathbb{E} e^{\lambda [X_i-X_j]} \leq a \lambda d(i,j) + \frac{v \lambda^2 d^2(i,j)}{2(1-c\lambda d(i,j)}
$$
for all $(i,j) \in \mathcal{T}^2$ and all $0 < \lambda < \{c d(i,j)\}^{-1}$. Then, for any $i_0\in \mathcal{T}$:
\begin{equation}\label{eq:chain_bound}
\mathbb{E} \sup_{i \in \mathcal{T}} [X_t - X_{i_0}] \leq 3 a \delta + 12 \sqrt{v} \int_{0}^{\delta/2} \sqrt{H(u,\mathcal{T})} du + 12 c
\int_{0}^{\delta/2}  H(u,\mathcal{T}) du 
\end{equation}
\end{lem}

 \bibliographystyle{alpha}
 \bibliography{GPS_2016_new}
\vskip2cm
\hskip70mm\box5

\end{document}